\documentclass[11pt]{article}
\usepackage{amsthm, amssymb, bm}
\usepackage{subfigure}
\usepackage{soul, color}
\usepackage[]{amsmath}
\usepackage[]{amsfonts}
\usepackage[]{fancyhdr}
\usepackage[]{graphicx}
\graphicspath{{/EPSF/}{../figures/}{figures/}}

\newtheorem{theorem}{Theorem}[section]
\newtheorem{lemma}[theorem]{Lemma}
\newtheorem{proposition}[theorem]{Proposition}

\newtheorem{remark}[theorem]{Remark}

\def \Cm {\mathbb{C}}
\def \Imm {\mathbb{I}}

\def \Rm {\mathbb{R}}
\def \Sm {\mathbb{S}}
\def \Um {\mathbb{U}}

\def\C{\mathcal{C}}

\def\G{\mathcal{G}}
\def\H{\mathcal{H}}

\def\M{\mathcal{M}}
\def\O{\mathcal{O}}

\newcommand{\eps}{\varepsilon}

\newcommand{\wtA}{ {\widetilde A} }

\newcommand{\where}{\quad\text{ where }}
\newcommand{\qandq}{\quad\text{ and }\quad}

\newcommand{\bfe}{ {\bf e}}

\newcommand{\bfg}{ {\bf g}}

\newcommand{\bfk}{ {\bf k}}
\newcommand{\bfr}{ {\bf r}}
\newcommand{\bfu}{ {\bf u}}
\newcommand{\bfv}{ {\bf v}}
\newcommand{\bfw}{ {\bf w}}

\newcommand{\bfP}{ {\bf P}}

\newcommand{\cout}[1]{}

\def \Sone { {\mathbb{S}^1}}

\title{Inverse anisotropic diffusion from power density measurements in two dimensions}

\author{Fran\c cois Monard\thanks{Department of Applied Physics and Applied Mathematics, Columbia University, New York NY, 10027; fm2234@columbia.edu} \and Guillaume Bal\thanks{Department of Applied Physics and Applied Mathematics, Columbia University,  New York NY, 10027; gb2030@columbia.edu}}

\begin{document}

\maketitle
%\tableofcontents
\begin{abstract}
  This paper concerns the reconstruction of an anisotropic diffusion tensor $\gamma=(\gamma_{ij})_{1\leq i,j\leq 2}$ from knowledge of internal functionals of the form $\gamma\nabla u_i\cdot\nabla u_j$ with $u_i$ for $1\leq i\leq I$ solutions of the elliptic equation $\nabla \cdot \gamma \nabla u_i=0$ on a two dimensional bounded domain with appropriate boundary conditions. We show that for $I=4$ and appropriately chosen boundary conditions, $\gamma$ may uniquely and stably be reconstructed from such internal functionals, which appear in coupled-physics inverse problems involving the ultrasound modulation of electrical or optical coefficients. Explicit reconstruction procedures for the diffusion tensor are presented and implemented numerically.
 
\end{abstract}
%\keywords{Hybrid methods, inverse conductivity, power density measurements, quasiconformal maps, complex geometric optics}

\section{Introduction}\label{sec:intro}

Coupled-physics modalities are being extensively studied in medical imaging as a means to combine high contrast with high resolution. Such imaging modalities typically couple a high-contrast low-resolution modality with a low-contrast high-resolution modality. In this context, Ultrasound Modulated Optical Tomography (UMOT) or Ultrasound Modulated Electrical Impedance Tomography (UMEIT) aim to improve the low resolution in the reconstruction of diffusion (in OT) or conductivity (in EIT) coefficients by perturbing the medium with focused or delocalized ultrasound when making measurements. We assume here that the electric potential in EIT and the photon density in OT are modeled by the following elliptic model
%Upon making a few assumptions about the form of the coupling between acoustic waves and the elliptic equation \eqref{eq:diff} used for modelling both techniques,
\begin{align}
    -\nabla\cdot(\gamma\nabla u) = -\sum_{i,j=1}^n \partial_i \left( \gamma_{ij} \partial_j u \right) = 0,\quad \mbox{ in } X, \qquad u_{|\partial X} = g \quad \mbox{ on } \partial X,
    \label{eq:diff}
\end{align}
where $X$ is a subset in $\Rm^n$, where $n$ will equal $2$ in this paper, and where $\gamma$ is a symmetric positive definite tensor. In the case of OT, equation \eqref{eq:diff} is an approximation of a more accurate model that takes into account absorption effects by adding a zeroth order term $\sigma_a u$ with $\sigma_a\ge 0$ in the left-hand side of \eqref{eq:diff}. Whether this additional term can be handled with the present approach will be the object of future research. We also assume that the ultrasound perturbation of the medium and of the corresponding boundary (current) measurements allow us to stably reconstruct the {\em power density} of two solutions $u,v$ of \eqref{eq:diff} with prescribed boundary conditions, namely, the quantity $H[u,v]:=\gamma\nabla u\cdot\nabla v$ over $X$. How to obtain power densities in practice has been addressed in e.g. \cite{ABCTF} using highly focused ultrasonic waves, and in e.g. \cite{KK0,BBMT} in the context of synthetic focusing. 

The main objective of this paper is the reconstruction of the symmetric tensor $\gamma$ in \eqref{eq:diff} and in dimension $n=2$ from knowledge of internal measurements $\{H[u_i, u_j]\}_{i,j=1}^m$, where each $u_i$ corresponds to a different, properly chosen boundary illumination $g_i$. 

The isotropic case $\gamma\equiv \sigma\Imm_n$ was analyzed in two dimensions in \cite{CFGK} and extended to three dimensions in \cite{BBMT} and to all dimensions $n\ge 2$ and more general types of measurements of the form $\sigma^{2\alpha}|\nabla u|^2, \alpha\in\Rm, (n-2)\alpha+1\ne 0$ in \cite{MB}. The case of internal current densities of the form $|\gamma\nabla u|$ arises in current density impedance imaging, see e.g. \cite{NTT} and \cite{BRev} for a review and bibliography of recent results obtained in similar problems. Although the results in this paper also extend to general values of $\alpha$, we restrict ourselves to the case $\alpha=\frac12$ to simplify the presentation. Similar problems in dimensions $n=2$ and $n=3$ were also analyzed in a linearized context in \cite{KK}; see also the recent paper \cite{KS} on general linearized hybrid inverse problems. 

This paper extends the analyzes in \cite{MB} to the case of two-dimensional anisotropic diffusion tensors. We show that with $4$ (or, in practice, in fact $3$) properly chosen illuminations, the measurements $\{H_{ij}\}_{i,j=1}^4$ allow for a unique and stable reconstruction of the full anisotropic tensor $\gamma$. In particular, the internal functionals considered here do allow us to uniquely reconstruct the conformal structure (or normalized anisotropy tensor) $\tilde\gamma := (\det \gamma)^{-\frac{1}{2}} \gamma$, unlike the case of Dirichlet-to-Neumann boundary measurements as they appear in the Calder\'on problem, where $\gamma$ can be reconstructed up to a push-forward by an arbitrary change of variables leaving each point on $\partial X$ invariant \cite{ALP}.

More precisely, using the decomposition $\gamma = (\det \gamma)^{\frac{1}{2}} \tilde\gamma$ with $\det\tilde\gamma=1$, the anisotropy tensor $\tilde\gamma$ can be reconstructed in an algebraic and pointwise manner, after which the quantity $(\det \gamma)^{\frac{1}{2}}$ can be obtained in two possible ways, either via inverting two consecutive gradient (or, after taking divergence, Poisson) equations, or by inverting a {\em strongly coupled elliptic system} followed by a gradient (or Poisson) equation. The reconstruction of $(\det\gamma)^{\frac{1}{2}}$ is similar to the case where $\gamma$ is known up to multiplication by a scalar function. Since this problem has the same dimensionality as that of the reconstruction of an isotropic diffusion tensor (treated in \cite{CFGK, BBMT, MB}), only $m=n$ illuminations are necessary and the reconstruction can be done following the second step of the previously described approach. Although some of the techniques presented here generalize to higher dimensions, we restrict ourselves to the two-dimensional setting in this paper.% and will appear in subsequent work. 

Finally, some numerical simulations confirm the theoretical predictions:
both the isotropic and the anisotropic parts of the tensor can be stably reconstructed, with a robustness to noise that is much better for the former than the latter.

Our main results are presented in section \ref{sec:statement}. Their derivation occupies sections \ref{sec:lemma}-\ref{sec:Ggamma}. Numerical simulations are shown in \ref{sec:numerics} and concluding remarks offered in section \ref{sec:conclu}.

\section{Statement of the main results}\label{sec:statement}
%%%%%%%%%%%%%%%%%%%%%%%%%%%

Let $X\subset \Rm^2$ be an open, bounded, simply connected domain. Borrowing notation from \cite{AN}, for $\kappa\ge 1$, a real $2\times 2$ symmetric matrix $\gamma = \{\gamma_{ij}\}_{i,j=1}^2$ belongs to $\M_\kappa^s(2)$ if and only if it satisfies the uniform ellipticity condition
\begin{align}
    \kappa^{-1} |\xi|^2 \le \gamma_{ij}\xi^i \xi^j \le \kappa |\xi|^2 \quad \text{for every } \xi\in \Rm^2.
    \label{eq:ellipticity}
\end{align}
In the following, we consider the problem of reconstructing an anisotropic conductivity function $\gamma\in L^\infty(X,\M^s_{\kappa_0}(2))$ where $\kappa_0 \ge 1$ is fixed, from the knowledge of power-density measurements of the form 
\begin{align*}
    H_{ij}(x) = \gamma\nabla u_i \cdot\nabla u_j, \quad 1\le i,j\le m,
\end{align*}
where for $1\le i\le m$, the function $u_i$ satisfies the conductivity equation
\begin{align}
    -\nabla\cdot (\gamma\nabla u_i) = 0 \quad \mbox{ in X}, \qquad u_i|_{\partial X} = g_i \quad \mbox{ on }\partial X.
    \label{eq:diffi}
\end{align}
The $g_i$'s are prescribed illuminations. We denote by $A$ the unique positive $\M^s(2)$-valued function that satisfies the pointwise relation $A^2(x) = \gamma(x)$. Clearly, $A\in L^\infty (X,\M^s_{\sqrt{\kappa_0}}(2))$. 

We first change unknown functions by defining for $1\le i\le m$, the vector fields  $S_i:= A\nabla u_i$. The elliptic equation \eqref{eq:diffi} thus reads
\begin{align}
    \nabla\cdot(AS_i) = 0.
    \label{eq:divSi}
\end{align}
Furthermore, denoting the ``curl'' operator in 2D by $J\nabla\cdot$ where $J = \left[ \begin{smallmatrix} 0 & -1 \\ 1 & 0 \end{smallmatrix} \right]$, the fact that $A^{-1} S_i = \nabla u_i$ implies that it is curl-free, that is,
\begin{align}
    J\nabla\cdot (A^{-1}S_i) = 0.
    \label{eq:dSi}
\end{align}
The data become $H_{ij} = S_i\cdot S_j$. We now decompose $A$ into
\begin{align}
    A = |A|^{\frac{1}{2}} \wtA, \qquad |A|:= \det A \qandq \det\wtA = 1.
    \label{eq:Adecomp}
\end{align}
From the uniform bounds $\kappa_0^{-\frac{1}{2}}\le |A|^\frac{1}{2}\le \kappa_0^{\frac{1}{2}}$, it follows immediatly that $\wtA\in L^\infty (X,\M^s_{\kappa_0})$. 

Over any open, connected set $\Omega\subseteqq X$, where two solutions $S_1,S_2$ satisfy the following positivity condition 
\begin{align}
    \inf_{x\in \Omega} (\det H)^{\frac{1}{2}} = \inf_{x\in \Omega} \det(S_1,S_2) \ge c_0 > 0, \quad H = \{S_i\cdot S_j\}_{i,j=1}^2,
    \label{eq:positivity}
\end{align}
we can derive the first important relation
\begin{align}
    \nabla\log |A| = \frac{1}{2} \nabla\log |H| + (\nabla H^{jl} \cdot \wtA S_l) \wtA^{-1} S_j = |H|^{-\frac{1}{2}} (\nabla (|H|^\frac{1}{2} H^{jl})\cdot \wtA S_l) \wtA^{-1} S_j.
    \label{eq:nlda}
\end{align}

We now orthonormalize the frame $S=(S_1,S_2)$ into a $SO(2)$-valued frame $R = (R_1,R_2)$, via a transformation of the form $R_i = t_{ij} S_j$ (or, in matrix notation $R = ST^T$, $T:= \{t_{ij}\}$), where $T$ is known from the data. For further use we denote $T^{-1} = \{t^{ij}\}$ and define the vector fields $\{V_{ij}\}_{i=1}^2$ as 
\begin{align}
    V_{ij} := \nabla (t_{ik}) t^{kj}, \quad 1\le i,j\le 2, \quad V_{ij}^a:= \frac{1}{2}(V_{ij}-V_{ji}).
    \label{eq:Vij}
\end{align}
This orthonormalization can be obtained by the Gram-Schmidt procedure or by setting $T = H^{-\frac{1}{2}}$ for instance. Since $R$ is $SO(2)$-valued, we parameterize it with a $\Sm^1$-valued function $\theta$ such that $R = \left[ \begin{smallmatrix} \cos\theta &-\sin\theta \\ \sin\theta & \cos\theta \end{smallmatrix} \right]$. One is then able to derive the following second important equation by using \eqref{eq:divSi} and geometric arguments:
\begin{align}
    \wtA^2\nabla\theta + [\wtA_2, \wtA_1] = \wtA^2 V_{12}^a - \frac{1}{2}JN,
    \label{eq:gradtheta}
\end{align}
where $N:=\frac{1}{2}\nabla\log |H|$ and $V_{12}^a$ are known from the data, and $[\wtA_2, \wtA_1] := (\wtA_2\cdot\nabla) \wtA_1 - (\wtA_1\cdot\nabla) \wtA_2$, where $\wtA_j$ denotes the $j$-th column of $\wtA$. 

\paragraph{Reconstruction of the anisotropy $\wtA$:}
We start by defining the set of admissible illuminations $\G_\gamma$ for some $\gamma\in L^\infty(X,\M_{\kappa_0}^s(2))$, by stating that a quadruple $\bfg = (g_1,g_2,g_3,g_4)$ belongs to $\G_\gamma$ if the following conditions are satisfied for some constant $c_0 >0$:
\begin{align}
    \inf_{x\in \overline{X}} \min (\det(\nabla u_1,\nabla u_2),\det(\nabla u_3,\nabla u_4) )&\ge c_0 >0, \label{eq:cond1} \\
    Y_\bfg := \frac{1}{2} J\nabla\log \left( \det (\nabla u_1,\nabla u_2) / \det (\nabla u_3, \nabla u_4) \right) &\ne 0 \quad \text{for every } x\in X,
    \label{eq:cond2}
\end{align}
where $u_i$ solves \eqref{eq:diffi} for $1\le i\le 4$. Condition \eqref{eq:cond2}, which is directly motivated by calculations later, expresses the fact that the relative variations of the volumes $\det(\nabla u_1,\nabla u_2)$ and $\det(\nabla u_3,\nabla u_4)$ differ at every point, which seems to be the condition that guarantees that our measurements are rich enough to ``see'' the anisotropy. 

Condition \eqref{eq:cond1} is rather easy to ensure by virtue of \cite[Theorem 4]{AN}, which guarantees that \eqref{eq:cond1} holds as soon as both maps $\partial X\ni x\mapsto (g_1(x),g_2(x))$ and $\partial X\ni x\mapsto (g_3(x),g_4(x))$ are homeomorphisms onto their images. Based on a construction that uses Complex Geometrical Optics (CGO) solutions, we construct in the next lemma solutions that satisfy condition \eqref{eq:cond2} under some regularity assumption on $\gamma$. This in turn guarantees that $\G_\gamma$ is not empty when $\gamma$ is smooth enough.
\begin{lemma} \label{lem:Ggamma}
    Let $\gamma\in L^\infty(X,\M_{\kappa_0}^s(2))$ for some $\kappa_0 \ge 1$ be such that $|\gamma|^{\frac{1}{2}} \in H^{5+\varepsilon}(X)$ for some $\varepsilon>0$ and the function $\nu:X\to \Cm$ defined as 
    \begin{align}
	X\ni x \mapsto \nu(x) = \frac{\gamma_{22} - \gamma_{11} - 2i\gamma_{12}}{\gamma_{11}+\gamma_{22}+2|\gamma|^{\frac{1}{2}}}(x)
	\label{eq:nu_lem}
    \end{align}
    is locally of class $\C^4$ over $X$. Then the set $\G_\gamma$ defined by conditions \eqref{eq:cond1} and \eqref{eq:cond2} is not empty and contains an open set of sufficiently smooth boundary conditions.
\end{lemma}

Consider now $(g_1,g_2,g_3,g_4)\in\G_\gamma$, and let us denote $\bfg^{(1)} = (g_1,g_2)$ and $\bfg^{(2)} = (g_3,g_4)$. Define two orthonormal frames $R^{(i)} = [R_1^{(i)}| R_2^{(i)}],\ i=1,2$, obtained after orthonormalization of $S^{(1)} = [S_1|S_2]$, and $S^{(2)} = [S_3|S_4]$. Taking the difference of equations \eqref{eq:gradtheta} for each system, we obtain the algebraic equation
\begin{align}
    \wtA^2 X_\bfg = Y_\bfg, \where\quad X_\bfg:= \nabla (\theta_2-\theta_1) - V_{12}^{a(2)} + V_{12}^{a(1)},
    \label{eq:algA}
\end{align}
and $Y_\bfg$ is defined in \eqref{eq:cond2}. Both vector fields $X_\bfg$ and $Y_\bfg$ are uniquely determined by the data, see below. Since $\wtA$ is only described by two scalar parameters, equation \eqref{eq:algA} together with the condition \eqref{eq:cond2} allow us to reconstruct these two parameters algebraically and in a pointwise manner. When orthonormalization uses the Gram-Schmidt procedure, $X_\bfg$ and $Y_\bfg$ satisfy the following boundedness and stability inequalities for some constants $C_1, C_2$:
\begin{align}
    \begin{split}
	\max ( \|X_\bfg\|_{L^\infty(X)}, \|Y\|_{L^\infty(X)} ) &\le C_1 \|H\|_{W^{1,\infty}(X)}, \\
	\max ( \|X_\bfg-X_\bfg'\|_{L^\infty(X)}, \|Y_\bfg-Y_\bfg'\|_{L^\infty(X)} ) &\le C_2 \|H-H'\|_{W^{1,\infty}(X)},	
    \end{split}
    \label{eq:XYGS}    
\end{align}
where $H = \{H_{ij}\}_{i,j=1}^4$ and $H' = \{H'_{ij}\}_{i,j=1}^4$ respectively come from $\bfg\in\G_\gamma$ and $\bfg\in\G_{\gamma'}$. We arrive at the following result:
\begin{theorem}[Anisotropy reconstruction in 2d with $m=4$] \label{thm:anisotropy}
    For $\bfg \in \G_\gamma$, the measurements $H = \{H_{ij}\}_{i,j=1}^4$ uniquely determine the tensor $\wtA$ via the explicit algebraic equation \eqref{eq:algA}. 
    Moreover, for $\gamma,\gamma'$ with $\bfg\in\G_\gamma$ and $\bfg\in\G_{\gamma'}$ with the corresponding measurements $H_{ij},H'_{ij} \in W^{1,\infty}(X)$ for $1\le i,j\le 2$, and in the case where orthonormalization is done using the Gram-Schmidt procedure, the following stability statement holds:
    \begin{align}
	\|\wtA-\wtA'\|_{L^\infty(X)} \le C \|H-H'\|_{W^{1,\infty}},
	\label{eq:stabA}
    \end{align}
    for some constant $C$. 
\end{theorem}

\begin{remark}
    In practice, we have observed numerically that $m=3$ was enough to reconstruct $\gamma$ if we chose $\bfg\in\G_\gamma$ of the form $(g_1,g_2,g_2,g_3)$. 
\end{remark}

\paragraph{Reconstruction of $(\theta,\log |A|)$ or $(u_1,u_2,|A|^{-1})$:}
Once the anisotropy $\wtA$ is known, the problem of reconstructing $|A|$ now has the same dimensionality as that of reconstructing an isotropic conductivity. It requires only $m=2$ internal functionals satisfying \eqref{eq:positivity} in $X$.

A first approach towards reconstructing $|A|$ consists in solving a gradient equation for the scalar quantities $\theta$ and then $\log |A|$, the right-hand sides of which are successively known. If $\theta$ and $|A|$ are known throughout the domain's boundary, one may take the divergence of said gradient equations instead and solve Poisson equations with Dirichlet boundary conditions. 

Such an approach provides Lipschitz-stable reconstructions as is summarized in 

\begin{theorem}[Stability of $|A|$, first approach]\label{thm:stab1}
    Assume that $\wtA$ is either known or reconstructed as in theorem \ref{thm:anisotropy} with the stability estimate \eqref{eq:stabA}. Then $|A|$ is uniquely determined by $\{H_{ij}\}_{1\le i,j\le 2}\in W^{1,\infty}$ satisfying \eqref{eq:positivity}. Moreover, if two set $H$ and $H'$ jointly satisfy the previous assumptions, the corresponding reconstructed coefficients $|A|$ and $|A'|$ satisfy the stability inequality
    \begin{align}
	\| \log |A| - \log |A'| \|_{W^{1,\infty}(X)} \le C \|H-H'\|_{W^{1,\infty}}. 
	\label{eq:stablogdetA}
    \end{align}
\end{theorem}

A second approach consists in inserting the expression in equation \eqref{eq:nlda} into the elliptic equation \eqref{eq:diffi} and deriving a strongly coupled elliptic system for the unknown $(u_1,u_2)$. In two dimensions, this system turns out to have a variational formulation with a coercive bilinear form: 
\begin{align}
    \nabla\cdot (\wtA^2 |H|^{\frac{1}{2}} H^{ij} \nabla u_j ) = 0, \quad u_i|_{\partial X} = g_i. \quad i=1,2,
    \label{eq:sces2d}
\end{align}
It thus admits a unique solution in the following space (up to an additive $H^1$-lifting of $(g_1,g_2)$) 
\begin{align}
    \H := (H_0^1(X))^2, \quad \bfu=(u_1,u_2) \in\H \quad\text{iff}\quad \|\bfu\|_\H^2:= \int_X |\nabla u_1|^2+|\nabla u_2|^2\ dx < \infty.
    \label{eq:H}
\end{align}
Using a Fredholm-type argument, we obtain that this solution is also stable with respect to changes in the data, as stated in the following result:
\begin{proposition}[Stability of the strongly coupled elliptic system] \label{prop:sces}
    Let $H, H'$ have their components in $W^{1,\infty}(X)$ and satisfy \eqref{eq:positivity}. If $\bfu, \bfu'$ are the unique solutions to \eqref{eq:sces2d} with the same illumination $\bfg$, then $\bfu-\bfu'\in\H$ and satisfies the stability estimate:
    \begin{align}
	\|\bfu-\bfu'\|_\H \le C \|H-H'\|_{W^{1,\infty}}. 
	\label{eq:stabu}
    \end{align}    
\end{proposition}

Once the couple $(u_1,u_2)$ is reconstructed throughout $X$, one may reconstruct $|A|^{-1}$ using (a modified version of) the gradient equation \eqref{eq:nlda}. 

\begin{theorem}[Stability of $|A|$, second approach] \label{thm:known_isotropy}
    Let the conditions of proposition \ref{prop:sces} be satisfied. Then the reconstructed determinants $|A|$ and $|A'|$ satisfy the stability estimate:
    \begin{align}
	\| |A|^{-1} - |A'|^{-1} \|_{H^1(X)} \le C \|H-H'\|_{W^{1,\infty}(X)}. 
	\label{eq:stabdetA}
    \end{align}
\end{theorem}

The rest of the paper is structured as follows.
We first derive equations \eqref{eq:nlda} and \eqref{eq:gradtheta} in section \ref{sec:lemma}, which form the cornerstone of all subsequent derivations. Section \ref{sec:recons} covers the three reconstruction algorithms mentioned above. Section \ref{sec:Ggamma} provides a proof of lemma \ref{lem:Ggamma} while section \ref{sec:numerics} concludes with numerical examples for each reconstruction algorithm.

\section{Proof of equations \ref{eq:nlda} and \eqref{eq:gradtheta}} \label{sec:lemma}

\subsection{Geometrical setting and preliminaries}

In this section, we work on an open connected subset $\Omega\subseteqq X$, over which $(S_1,S_2)$ satisfy the positivity condition \eqref{eq:positivity}. For a matrix $M= PDP^T \in \M_\kappa^s (2)$ with $P\in O(2,\Rm), D = \text{diag }(\lambda_1,\lambda_2)$, and a scalar $r\in\Rm$, we can define uniquely $M^r := P D^r P^T \in \M^s_{\kappa^r}(2)$ by taking the positive $r$-root of each eigenvalue. Now, because $A^r = \gamma^\frac{r}{2}$ is uniformly elliptic for any $r\in\Rm$, the vector fields $(A^rS_1, A^r S_2)$ also form an oriented frame (denoted $A^r S$). The measurements can also be represented as 
\begin{align}
    H_{ij} = A^r S_i\cdot A^{-r} S_j, \quad 1\le i,j\le 2,\quad r\in \Rm.
    \label{eq:measr}
\end{align}
From this assumption, one can deduce that any vector field $V$ can be represented by means of the formula 
\begin{align}
    V = H^{ij} (V\cdot A^r S_i) A^{-r}S_j, \quad r\in\Rm.
    \label{eq:representation}
\end{align}
Both equations \eqref{eq:measr} and \eqref{eq:representation} only ``see'' $A$ up to a scalar multiplicative constant, thus these equations still hold if we replace $A$ by $\wtA = |A|^{-\frac{1}{2}} A$.

Finally, we give the following important relation (only valid when $n=2$) true for any $M \in \M^s(2,\Rm)$:
\begin{align}
    MJM = (\det M)J, \quad J:= \left[
    \begin{array}{cc}
	0 & -1 \\ 1 & 0
    \end{array}
    \right].
    \label{eq:MJM}
\end{align}

\subsection{Proof of the gradient equation \eqref{eq:nlda}} \label{sec:nlda}
The derivation relies on the analysis of the properties of the vector fields $JA^{-1}S_i$ for $i=1,2$. First notice that since $J$ is skew-symmetric, we have
\begin{align*}
    JA^{-1} S_i\cdot A^{-1} S_i = 0, \quad i=1,2,
\end{align*}
Then, using the relation \eqref{eq:MJM} with $M = A^{-1}$ and the fact that $JS_1\cdot S_2 = \det (S_1,S_2) =: |H|^\frac{1}{2}$,
\begin{align*}
    JA^{-1} S_1\cdot A^{-1} S_2 = - JA^{-1} S_2\cdot A^{-1} S_1 = (A^{-1}JA^{-1} S_1)\cdot S_2  = |A|^{-1} JS_1\cdot S_2 = |A|^{-1} |H|^\frac{1}{2}.
\end{align*}
This means that the vector fields $JA^{-1} S_i$ can be expressed using the representation \eqref{eq:representation} with $r=-1$:
\begin{align}
    \begin{split}
	JA^{-1}S_1 = H^{pq} (JA^{-1}S_1\cdot A^{-1} S_p) AS_q = H^{2q} |A|^{-1} |H|^\frac{1}{2} AS_q, \\
	JA^{-1}S_2 = H^{pq} (JA^{-1}S_2\cdot A^{-1} S_p) AS_q = -H^{1q} |A|^{-1} |H|^\frac{1}{2} AS_q.
    \end{split}
    \label{eq:JAS}
\end{align}
We now apply the divergence operator to \eqref{eq:JAS}. Together with the fact that $\nabla\cdot (JA^{-1}S_i) = -(J\nabla)\cdot (A^{-1}S_i) = 0$ and equation \eqref{eq:divSi}, and using the identity $\nabla (fV) = \nabla f\cdot V + f\nabla\cdot V$, we derive
\begin{align*}
    \nabla |A|^{-1}\cdot |H|^\frac{1}{2} H^{qp} AS_p + |A|^{-1} \nabla (|H|^\frac{1}{2} H^{qp}) \cdot AS_p = 0, \quad q=1,2.
\end{align*}
Multiplying the last equation by $A^{-1} S_q$, summing over $q$ and dividing by $|A|^{-1} |H|^{\frac{1}{2}}$, we obtain
\begin{align*}
    H^{qp} (-\nabla\log|A|\cdot AS_p) A^{-1} S_q + |H|^{-\frac{1}{2}} ( \nabla (|H|^\frac{1}{2} H^{qp}) \cdot AS_p) A^{-1}S_q = 0.
\end{align*}
The first term is of the form \eqref{eq:representation} with $r=1$ and $V= - \nabla\log |A|$ and thus equals $-\nabla\log |A|$. We obtain the second term of the r.h.s. of \eqref{eq:nlda}. The first term of the r.h.s. of \eqref{eq:nlda} is obtained from the second one by expanding the term $\nabla (|H|^\frac{1}{2}H^{pq})$ and using the product rule and identity \eqref{eq:representation} to obtain the $\frac{1}{2}\nabla\log |H|$ term.

\subsection{Proof of \eqref{eq:gradtheta}} \label{sec:LieASi}

We now orthonormalize $S$ into a frame $R$ via the transformation $R = S T^T$, also written as 
\begin{align*}
    R_i = t_{ij} S_j, \qquad S_j = t^{ij} R_j, \quad 1\le i\le n.
\end{align*}
The matrix $T$ satisfies $T^T T = H^{-1}$, also written as $t_{ki}t_{kj} = H^{ij}$ for $1\le i,j\le n=2$. It can be constructed by the Gram-Schmidt procedure or by writing $T = H^{-\frac{1}{2}}$. With the $V_{ij}$'s defined in \eqref{eq:Vij}, the following important identity holds 
\begin{align}
    (\nabla H^{ij}) t^{ik} t^{jl} = (\nabla (t_{pi}t_{pj})) t^{ik} t^{jl} = \delta_{pk} (\nabla t_{pj})t^{jl} + \delta_{pl} (\nabla t_{pi})t^{ik} = V_{kl} + V_{lk}, \quad 1\le l,k\le 2. 
    \label{eq:HTT}
\end{align}
Therefore, starting from \eqref{eq:nlda}, we have 
\begin{align}
    \nabla\log |A| &= N + (\nabla H^{jl}\cdot \wtA S_l ) \wtA^{-1} S_j = N + ( (\nabla H^{jl}) t^{lp} t^{jq} \cdot \wtA R_p ) \wtA^{-1} R_q \nonumber\\
    &= N + ((V_{pq} + V_{qp}) \cdot \wtA R_p ) \wtA^{-1} R_q,	
    \label{eq:nldaR}    
\end{align}
where we have used \eqref{eq:HTT} in the last equality. Now, in order to derive equation \eqref{eq:gradtheta}, we must write the Lie bracket $[\wtA R_2,\wtA R_1]$ in two different manners. 

On the one hand, writing $[\wtA R_2, \wtA R_1]$ in the canonical basis $(\bfe_1,\bfe_2)$ and using the identity $[aX,bY] = a (X\cdot\nabla)(b) Y - b (Y\cdot\nabla)(a) X + ab[X,Y]$, we have that 
\begin{align*}
    [\wtA R_2,\wtA R_1] &= [\wtA_{ij}R_2^j \bfe_i, \wtA_{kl}R_1^l \bfe_k] \\
    &= \left( \wtA_{ij}\wtA_{kl} (R_2^j \partial_i R_1^l - R_1^j \partial_i R_2^l) + \wtA_{ij} \partial_i \wtA_{kl} (R_2^j R_1^l - R_1^j R_2^l) \right)\bfe_k,	
\end{align*}
after renumbering indices properly. Moreover, in the parameterization $R(\theta)$ we have 
\begin{align*}
    R_2^j \partial_i R_1^l - R_1^j \partial_i R_2^l = \left\lbrace
    \begin{array}{cc}
	\partial_i\theta & \text{if } j=l \\
	0 & \text{if } j\ne l
    \end{array}
    \right. \qandq R_2^jR_1^l-R_1^jR_2^l = \left\lbrace
    \begin{array}{cc}
	0 & \text{if } j=l \\
	1 & \text{if } (j,l) = (2,1) \\
	-1 & \text{if } (j,l) = (1,2)
    \end{array}
    \right. ,
\end{align*}
thus we obtain 
\begin{align}
    [\wtA R_2,\wtA R_1] = \left( (\wtA_{i1} \wtA_{k1} + \wtA_{i2}\wtA_{k2}) \partial_i\theta + \wtA_{i2}\partial_i \wtA_{k1} - \wtA_{i1} \partial_i \wtA_{k2} \right)\bfe_k = \wtA^2\nabla\theta + [\wtA_2, \wtA_1].
    \label{eq:gradtheta_lhs}
\end{align}

On the other hand, we compute $[\wtA R_2,\wtA R_1]$ using \eqref{eq:divSi}. First, deriving a divergence equation for $\wtA R_i, i=1,2$, and using the fact that \eqref{eq:divSi} can be recast as $\nabla\cdot(\wtA S_i) = -\frac{1}{2}\nabla\log |A|\cdot \wtA S_i$, we have
\begin{align}
    \nabla\cdot(\wtA R_i) &= \nabla\cdot (\wtA t_{ij}S_j) = \nabla t_{ij}\cdot \wtA S_j + t_{ij} \nabla\cdot (\wtA S_j) \nonumber\\
    &= \nabla t_{ij} \cdot \wtA t^{jk} R_k - t_{ij} \frac{1}{2} \nabla\log|A|\cdot \wtA S_j = V_{ik}\cdot \wtA R_k - \frac{1}{2} \nabla\log |A|\cdot \wtA R_i \nonumber\\
    &= -\frac{1}{2} N\cdot \wtA R_i + V_{ik}^a \cdot \wtA R_k, \label{eq:divARi}
\end{align}
where we have used \eqref{eq:nldaR} in the last equality. We now use the following 2d vector calculus identity
\begin{align}
    [U,V] = (U\cdot\nabla) V - (V\cdot\nabla) U = \nabla\cdot (V\otimes U-U\otimes V ) - (\nabla\cdot U) V + (\nabla\cdot V)U,
    \label{eq:vci2d}
\end{align}
where $\nabla\cdot M:= \partial_i M_{ji} \bfe_j$ for $M$ a $2\times 2$ matrix. With $U=\wtA R_2$ and $V= \wtA R_1$, we have 
\begin{align*}
    \wtA R_1\otimes \wtA R_2 - \wtA R_2\otimes \wtA R_1 = \wtA (R_1\otimes R_2 - R_2\otimes R_1)\wtA = - \wtA J\wtA = - (\det\wtA) J = -J,
\end{align*}
so the first term in the r.h.s. of \eqref{eq:vci2d} is zero. Thus we have
\begin{align*}
    [\wtA R_2, \wtA R_1] &= (\nabla\cdot \wtA R_1) \wtA R_2 - (\nabla\cdot\wtA R_2) \wtA R_1 \\
    &= - \frac{1}{2} (N\cdot \wtA R_1) \wtA R_2 + \frac{1}{2} (N\cdot\wtA R_2) \wtA R_1 + (V_{12}^a\cdot\wtA R_2) \wtA R_2 + (V_{12}^a\cdot\wtA R_1)\wtA R_1 \\
    &= \wtA (R_1\otimes R_1+R_2\otimes R_2)\wtA V_{12}^a - \frac{1}{2} \wtA (R_2\otimes R_1 - R_1\otimes R_2)\wtA N \\
    &= \wtA^2 V_{12}^a - \frac{1}{2} JN, 
\end{align*}
where we have used the properties $R_1\otimes R_1+R_2\otimes R_2 = \Imm_2$ and $R_2\otimes R_1 - R_1\otimes R_2 = J$. Combining \eqref{eq:gradtheta_lhs} with the last r.h.s. yields \eqref{eq:gradtheta}. 

\section{Reconstruction procedures} \label{sec:recons}

This section is devoted to the reconstruction procedures. We first reconstruct the anisotropic part of the conductivity tensor $\wtA$ and second reconstruct $(\theta, \log |A|)$ and $(u_1,u_2, |A|^{-1})$. 

\subsection{Reconstruction of the anisotropy $\wtA = \tilde\gamma^{\frac{1}{2}}$ with $m=3$ or $4$} \label{sec:anis}

Let us now consider a quadruple $(g_1,g_2,g_3,g_4)\in \G_\gamma$ with possibly $g_2 = g_3$. Condition \eqref{eq:cond1} ensures that the matrices $S^{(1)} = [S_1^{(1)}|S_2^{(1)}] = [S_1|S_2]$ and $S^{(2)} = [S_1^{(2)}|S_2^{(2)}] = [S_3|S_4]$ satisfy the positivity condition \eqref{eq:positivity}. Let us denote $R^{(i)} = S^{(i)} T^{(i)T}$ the $SO(2,\Rm)$-valued matrix obtained after Gram-Schmidt orthonormalization, parameterized by an angle function $\theta_i$, and denote
\begin{align*}
    H^{(i)} = S^{(i)T} S^{(i)}, \quad N^{(i)} := \frac{1}{2} \nabla\log |H^{(i)}|, \quad V_{12}^{a(i)} := \frac{1}{2} (V_{12}^{(i)} - V_{21}^{(i)}),\quad i=1,2.
\end{align*}
For each pair of solutions, we have the equation
\begin{align}
    \wtA^2\nabla\theta_i + [\wtA_2, \wtA_1] = \wtA^2 V_{12}^{a(i)} - \frac{1}{2} JN^{(i)}, \quad i=1,2.
    \label{eq:LieARi2d3}
\end{align}
Taking the difference of both systems cancels the term $[\wtA_2, \wtA_1]$, and we obtain equation \eqref{eq:algA}, with vector fields 
\begin{align*}
    X_\bfg = (x_1,x_2)^T := \nabla(\theta_2-\theta_1) - V_{12}^{a(2)} + V_{12}^{a(1)} \qandq Y_\bfg = (y_1,y_2)^T := -\frac{1}{2} J(N^{(2)}-N^{(1)}).
\end{align*}
Now we claim that although the angle functions $\theta_1, \theta_2$ are unknown, $\nabla (\theta_2-\theta_1)$ is known from the data. Indeed, we have that
\begin{align*}
    \nabla (\theta_2-\theta_1) = \cos(\theta_2-\theta_1) \nabla (\sin(\theta_2-\theta_1)) - \sin(\theta_2-\theta_1) \nabla (\cos(\theta_2-\theta_1)),
\end{align*}
and then,
\begin{align*}
    \cos(\theta_2-\theta_1) = R_1^{(1)}\cdot R_1^{(2)} = t_{1i}^{(1)} t_{1j}^{(2)} S^{(1)}_i\cdot S^{(2)}_j, \quad \sin(\theta_2-\theta_1) = R_2^{(1)}\cdot R_1^{(2)} = t_{2i}^{(1)} t_{1j}^{(2)} S^{(1)}_i\cdot S^{(2)}_j,
\end{align*}
where both r.h.s. only depend on the data. As a result, the vector fields $X_\bfg, Y_\bfg$ are completely known from the data $\{H_{ij}\}_{i,j=1}^4$.

\paragraph{Parameterization of $\wtA^2$ and inversion:}
The matrix $\wtA^2$ is symmetric and has unit determinant and as such is characterized by only two scalar parameters. This is why equation \eqref{eq:algA} is enough to reconstruct it algebraically wherever $X_\bfg \ne 0$ or $Y_\bfg \ne 0$. We now parameterize $\wtA^2$ as follows
\begin{align}
    \wtA^2 (\xi,\zeta) = \left[
    \begin{array}{cc}
	\xi & \zeta \\ \zeta & \frac{1+\zeta^2}{\xi}
    \end{array}
    \right], \quad \xi > 0,
    \label{eq:paramA2}
\end{align}
where the second row is deduced from the first one by constructing a symmetric matrix with unit determinant. With $X_\bfg = (x_\bfg^1, x_\bfg^2)^T$ and $Y_\bfg = ( y_\bfg^1, y_\bfg^2)^T$, equation \eqref{eq:algA} becomes
\begin{align*}
    \xi x_\bfg^1 + \zeta x_\bfg^2 = y_\bfg^1 \qandq \xi \zeta x_\bfg^1 + (1+\zeta^2)x_\bfg^2 = \xi y_\bfg^2,
\end{align*}
which can be rewritten as 
\begin{align}
    \left[
    \begin{array}{cc}
	x_\bfg^1 & x_\bfg^2 \\ y_\bfg^2 & -y_\bfg^1
    \end{array}
    \right] \left[
    \begin{array}{c}
	\xi \\ \zeta
    \end{array}
    \right] = \left[
    \begin{array}{c}
	y_\bfg^1 \\ x_\bfg^2
    \end{array}
    \right],
    \label{eq:sysxizeta}
\end{align}
and can therefore be inverted as 
\begin{align}
    \xi = (X_\bfg\cdot Y_\bfg)^{-1} \left( (y_\bfg^1)^2+ (x_\bfg^2)^2 \right) \qandq \zeta = (X_\bfg\cdot Y_\bfg)^{-1} \left( y_\bfg^1 y_\bfg^2 - x_\bfg^1 x_\bfg^2  \right). 
    \label{eq:reconsxizeta}
\end{align}

\begin{proof}[Proof of theorem \ref{thm:anisotropy}]
    The only important point is to show that $X_\bfg\cdot Y_\bfg$ is never zero. Since condition \eqref{eq:cond2} is satisfied, then $Y_\bfg$ never vanishes over $X$. Since $\wtA\in \M_{\kappa_0^2}^s$, we have $X_\bfg\cdot Y_\bfg = \|\wtA^{-1} Y_\bfg\|^2 \ge \kappa_0^{-2} \|Y_\bfg\|^2$. Therefore, $X_\bfg\cdot Y_\bfg=0$ wherever $Y_\bfg=0$, that is, nowhere in this case. Inequality \eqref{eq:stabA} follows provided that the inequalities \eqref{eq:XYGS} hold in the Gram-Schmidt case, and the expressions \eqref{eq:reconsxizeta} are smooth in the components of $X_\bfg,Y_\bfg$ away from $\{X_\bfg\cdot Y_\bfg = 0\}$. 
\end{proof}

\begin{remark}[An unstable parameterization of $\wtA$]
    Another way (seemingly geometrically more meaningful) to parameterize $\wtA$ is, for $(a,\alpha) \in (0,\infty)\times \Sone$, to write
    \begin{align}
	\wtA(a,\alpha) = \left[
	\begin{array}{cc}
	    c_\alpha & -s_\alpha \\ s_\alpha & c_\alpha
	\end{array}
	\right] \left[
	\begin{array}{cc}
	    a & 0 \\ 0 & a^{-1}
	\end{array}
	\right] \left[
	\begin{array}{cc}
	    c_\alpha & s_\alpha \\ -s_\alpha & c_\alpha
	\end{array}
	\right], \quad (c_\alpha, s_\alpha) := (\cos\alpha, \sin\alpha).
	\label{eq:paramA1}
    \end{align}
    $a$ describes the anisotropy and $\alpha$ locates the main axes of the ellipse. However, besides the fact that this representation is not injective (indeed we have $\wtA(a,\alpha) = \wtA(a,\alpha+\pi) = \wtA(a^{-1}, \alpha+\pi/2)$), the reconstruction of $\alpha$ becomes unstable as $a$ approaches $1$.
\end{remark}

\subsection{Reconstruction of $|A| = |\gamma|^{\frac{1}{2}}$} \label{sec:deta}

We now consider the problem of reconstructing $|A|$ from $m=n=2$ measurements, assuming that $\wtA$ is known. We assume that the positivity condition is satisfied throughout the domain $X$. We propose two approaches, which we analyse consecutively in the next two sections.  

\subsubsection{Reconstruction of $(\theta,\log |A|)$}
This approach consists in reconstructing $\theta$ and then $|A|$ using the gradient equations \eqref{eq:gradtheta} and \eqref{eq:nlda}. We first isolate $\nabla\theta$ in \eqref{eq:gradtheta} by writing 
\begin{align}
    \nabla\theta = V_{12}^a - \wtA^{-2} \left( \frac{1}{2} JN + [\wtA_2, \wtA_1] \right).
    \label{eq:gradtheta_alone}
\end{align}
We thus require an expression for $[\wtA_2, \wtA_1]$. In the case where $\wtA^2$ was reconstructed in the previous section using the $(\xi,\zeta)$ variables, we need to compute $\wtA$ from knowledge of $\wtA^2$, which we do by first introducing a parameterization of $\wtA$ similar to \eqref{eq:paramA2}, called $\wtA(\lambda,\mu)$ with $\lambda>0$. Then, equating the terms in the first row of both representations $(\wtA(\lambda,\mu))^2$ and $\wtA^2 (\xi,\zeta)$, we obtain the relations
\begin{align*}
    \lambda^2 + \mu^2 = \xi \qandq \frac{\mu}{\lambda}(1+\lambda^2+\mu^2) = \zeta,
\end{align*}
which, using the condition $\lambda>0$ we invert as,
\begin{align}
    \lambda = (1+\xi) \left( \frac{\xi}{\zeta^2 + (1+\xi)^2} \right)^{\frac{1}{2}} \qandq \mu = \zeta \left( \frac{\xi}{\zeta^2 + (1+\xi)^2} \right)^{\frac{1}{2}}. 
    \label{eq:reconslambdamu}
\end{align}
In the variables $(\lambda,\mu)$, we now obtain the following expression for the term $[\wtA_2, \wtA_1]$:
\begin{align}
    [\wtA_2, \wtA_1] = \left[
    \begin{array}{cc}
	\mu & \frac{1+\mu^2}{\lambda} \\ 
	\frac{1+\mu^2}{\lambda} & \frac{\mu(1+\mu^2)}{\lambda^2} 
    \end{array}
    \right] \nabla\lambda - \left[
    \begin{array}{cc}
	\lambda & \mu \\ \mu & \frac{\mu^2-1}{\lambda}
    \end{array}
    \right] \nabla\mu. 
    \label{eq:A2A1}
\end{align}
Using \eqref{eq:A2A1} in \eqref{eq:gradtheta_alone} allows us to reconstruct $\theta$ via integrating the gradient equation \eqref{eq:gradtheta_alone} along curves (and assuming that $\theta$ is known at one point of the domain). 
Alternatively, if $\theta$ is known on the whole boundary, one may apply the divergence operator on both sides of \eqref{eq:gradtheta_alone} and solve a Poisson equation with Dirichlet boundary conditions. 

On to the reconstruction of $|A|$, we now use equation \eqref{eq:nlda} in the $R(\theta)$ frame to obtain:
\begin{align}
    \nabla\log |A| = \frac{1}{2} \nabla\log |H| + 2\sum_{p,q=1}^2 (V_{pq}^s \cdot \wtA R_p) \wtA^{-1} R_q,
    \label{eq:nla2d}
\end{align}
whose r.h.s. is completely known. For its resolution, equation \eqref{eq:nla2d} may be simplified using a calculation similar to \cite[Sec. 6.2]{MB}. To this end, we define $\Phi_{ij} (\theta) = R_i\otimes R_j$ for $1\le i,j\le 2$ and compute
\begin{align*}
    \nabla\log |A| &= \frac{1}{2} \nabla\log |H| + 2 \sum_{p,q=1}^2 \wtA^{-1} \Phi_{pq} \wtA V_{pq}^s \\
    &= - V_{11} - V_{22} + 2 \wtA^{-1} \Phi_{11} \wtA V_{11} + 2 \wtA^{-1} \Phi_{22} \wtA V_{22} + \wtA^{-1} (\Phi_{12} + \Phi_{21}) \wtA (V_{12} + V_{21}) \\
    &= \wtA^{-1} (\Phi_{11}-\Phi_{22}) \wtA (V_{11}-V_{22}) + \wtA^{-1} (\Phi_{12} + \Phi_{21}) \wtA (V_{12} + V_{21}),
\end{align*}
where we have used the facts that $\wtA^{-1} (\Phi_{11} + \Phi_{22})\wtA= \Imm_2$ and $- V_{11} - V_{22} = \nabla\log |H|^{\frac{1}{2}}$. The matrices $\Phi_{11}-\Phi_{22}$ and $\Phi_{12} + \Phi_{21}$ are symmetric matrices that can be expressed in the following manner
\begin{align}
    \begin{split}
	\Phi_{11}-\Phi_{22} &= c\Um + s J\Um \qandq \Phi_{12}+\Phi_{21} = -s\Um+c J\Um, \where \\
	(c,s)&:= (\cos(2\theta),\sin(2\theta)), \quad \Um := \left[
	\begin{array}{cc}
	    1 & 0 \\ 0 & -1
	\end{array}
	\right].
    \end{split}
    \label{eq:reflexions}
\end{align}
From this we deduce the final expression of $\nabla\log |A|$:
\begin{align}
    \begin{split}
	\nabla\log |A| = \wtA^{-1} \,\big( \cos(2\theta) F_c + \sin(2\theta) JF_c \big), \qquad
	F_c := \Um \wtA (V_{11}-V_{22}) + J\Um \wtA (V_{12}+V_{21}).
    \end{split}
    \label{eq:nla2d2}
\end{align}

\begin{proof}[Proof of Theorem \ref{thm:stab1}]
    The proof is very similar to \cite[Theorem 3.2]{BBMT}. For two sets of measurements $H$ and $H'$, the error function on $\theta$ satisfies 
    \begin{align*}
	\nabla (\theta-\theta') = V_{12}^a - V_{12}^{a'} - \frac{1}{2} \wtA^{-2} J (N-N').
    \end{align*}
    Assuming that $\theta(x_0)=\theta'(x_0)$ for some $x_0\in X$ and using Gronwall's lemma along (bounded) paths joining any $x\in X$ to $x_0$, and provided that $\|V_{ij} - V_{ij}'\|_{L^\infty}\le C\|H-H'\|_{W^{1,\infty}}$ in the Gram-Schmidt case, we arrive at the inequality
    \begin{align}
	\|\theta-\theta'\|_{W^{1,\infty}(X)} \le C \|H-H'\|_{W^{1,\infty}(X)}. 
	\label{eq:stabtheta}
    \end{align}
    Similarly, using the difference of equation \eqref{eq:nla2d} for $\log |A|$ and $\log |A|'$ and using \eqref{eq:stabtheta}, we arrive at \eqref{eq:stablogdetA}. 
\end{proof}

\begin{remark}
    As previously pointed out in \cite{BBMT,MB}, solving gradient equations may require the enforcement of compatibility conditions (i.e. the r.h.s must be curl-free), in order to ensure that the computed solution does not depend on the choice of integration curve. 
\end{remark}

\subsubsection{Reconstruction of $(u_1,u_2,|A|^{-1})$}

This approach, first introduced in \cite{MB}, consists in writing a strongly coupled elliptic system for the unknowns $(u_1,u_2)$, whose properties are particularly appealing in two dimensions. We proceed as follows. 

In the frame $(\nabla u_1, \nabla u_2)$, equation \eqref{eq:nlda} reads
\begin{align}
    \nabla\log |A| = |A| |H|^{-\frac{1}{2}} (\nabla ( |H|^{\frac{1}{2}} H^{pq})\cdot \wtA^2 \nabla u_p) \nabla u_q. 
    \label{eq:nlda_u}
\end{align}
Now, the diffusion equation \eqref{eq:diffi} can be rewritten as
\begin{align*}
    \nabla\cdot( \wtA^2 \nabla u_i ) + \nabla\log |A|\cdot \wtA^2 \nabla u_i = 0.
\end{align*}
Plugging \eqref{eq:nlda_u} into the latter equation and using the fact that $|A| \nabla u_q\cdot \wtA^2 \nabla u_i = H_{qi}$ yields the system 
\begin{align}
    \nabla\cdot( \wtA^2 \nabla u_i ) &+ W_{ip}\cdot \wtA^2 \nabla u_p = 0, \quad u_i|_{\partial X} = g_i, \quad i=1,2, \where \label{eq:sces} \\
    W_{ip} &:= H_{qi} |H|^{-\frac{1}{2}} \nabla (|H|^\frac{1}{2} H^{pq}), \quad 1\le i,p\le 2. \label{eq:W}
\end{align}
Upon multiplying \eqref{eq:sces} by $|H|^\frac{1}{2} H^{ji}$ and summing over $j$, we obtain an equivalent formulation in divergence form
\begin{align}
    \nabla\cdot( |H|^\frac{1}{2} H^{ji} \wtA^2 \nabla u_i ) = 0, \quad u_j|_{\partial X} = g_j, \quad j=1,2.
    \label{eq:sces_div}
\end{align}
We assume that $g_1,g_2 \in H^{\frac{1}{2}}(\partial X)$ and define $v_i$ to be a $H^1$-lifting of $g_i$ inside $X$. Defining the new unknown $\bfw = (w_1,w_2) := (u_1-v_1,u_2-v_2)$, $\bfw$ satisfies the following two equivalent systems whenever $\bfu = (u_1,u_2)$ satisfies \eqref{eq:sces} or \eqref{eq:sces_div} and vice-versa
\begin{align}
    \nabla\cdot( \wtA^2 \nabla w_i ) + W_{ip}\cdot \wtA^2 \nabla w_p = h_i := -\nabla\cdot( \wtA^2 \nabla v_i ) + W_{ip}\cdot \wtA^2 \nabla v_p, \quad w_i|_{\partial X} = 0, \quad i=1,2. \label{eq:scesw}    \\
    -\nabla\cdot (|H|^{\frac{1}{2}} \wtA^2 H^{ki} \nabla w_i) = f_k := \nabla\cdot (|H|^{\frac{1}{2}} \wtA^2 H^{ki} \nabla v_i), \quad x\in X,\quad w_k|_{\partial X} = 0,  \quad k=1,2. \label{eq:sces_divw}
\end{align}
System \eqref{eq:sces_divw} allows us to establish existence and uniqueness of $\bfw$ while \eqref{eq:scesw} is used to establish the stability of $\bfw$ with respect to the data $H$.

\paragraph{Uniqueness of $(u_1,u_2)$:}
System \eqref{eq:sces_divw} can be recast as the variational problem of finding $\bfw\in\H:= H_0^1(X)^2$ such that for every $\bfw'\in \H$, we have
\begin{align}
    B(\bfw,\bfw') = \int_X f_k w'_k\ dx, \where\quad  B(\bfw,\bfw') := \int_X |H|^{\frac{1}{2}} H^{ki} (\wtA^2\nabla w_i)\cdot\nabla w'_k\ dx.
    \label{eq:sces_var}
\end{align}
With $\H$ endowed with the norm $\|\bfw\|_\H^2 = \int_X |\nabla w_1|^2 + |\nabla w_2|^2\ dx$, assuming the positivity condition \eqref{eq:positivity}, the matrices $H$, $H^{-1}$ and $\wtA^2$ are all uniformly elliptic over $X$, which guarantees that the bilinear form is coercive. The continuity of $B$ and of the linear form $\bfw'\mapsto \int_X f_k w'_k\ dx$ over $\H$ are straightforward. As a result, Lax-Milgram's theorem establishes existence and uniqueness of $\bfw\in\H$ solving \eqref{eq:sces_divw} and \eqref{eq:scesw}, and thus of $\bfu\in \bfv+\H$ solving \eqref{eq:sces} and \eqref{eq:sces_div}. 

\paragraph{Stability of $(u_1,u_2)$ with respect to the data:}
Let us define the operator $L^{-1}: L^2(X)\ni f\mapsto w$, where $w$ is the unique solution to the problem 
\begin{align}
    Lw := - \nabla\cdot(\wtA^2\nabla w) = f\quad \mbox{ in } X, \quad w|_{\partial X} = 0.
    \label{eq:L}
\end{align}
Since $\wtA^2$ is uniformly elliptic, the operator $L^{-1}:L^2(X) \to H^2(X)$ is bounded (see e.g. \cite{E}), thus compact $L^2(X)\to H_0^1(X)$ by the Rellich compactness theorem. Therefore, applying $L^{-1}$ to \eqref{eq:scesw}, we obtain the integro-differential system
\begin{align*}
    w_i + L^{-1} (-W_{ip}\cdot \wtA^2 \nabla w_p) = - L^{-1} h_i, \quad i=1,2,
\end{align*}
which in vector notation may be recast as
\begin{align}
    (I + \bfP_W) \bfw = \{L^{-1} f_i\}_{i=1}^2, \where\quad \bfP_W \bfw = \left\{ L^{-1} (-W_{ip}\cdot \wtA^2 \nabla w_p) \right\}_{i=1}^2.
    \label{eq:sces_PW}
\end{align}
Similarly to \cite[Lemma 5.1]{MB}, the operator $\bfP_W:\H\to\H$ is compact and its operator norm satisfies 
\begin{align}
    \|\bfP_W\| \le C \|W\|_\infty, \quad \|W\|_\infty = \max_{1\le i,j\le 2} \|W_{ij}\|_{L^\infty(X)},
    \label{eq:PWnorm}
\end{align}
see \cite{MB} for details. Therefore, equation \eqref{eq:sces_PW} is a Fredholm equation and the boundedness of $(I+\bfP_W)^{-1}$ holds as soon as $-1$ is not an eigenvalue of $\bfP_W$. This is the case here, since the previous paragraph proves exactly this fact. On to the stability of $\bfu$ w.r.t. the data $H$, we use the fact that the vector fields $W$ defined in \eqref{eq:W} satisfy estimates of the form 
\begin{align*}
    \| W-W' \|_{L^\infty(X)} \le C \|H-H\|_{W^{1,\infty}(X)},
\end{align*}
whenever $H$ and $H'$ have their components in $W^{1,\infty}(X)$ and satisfy the positivity condition \eqref{eq:positivity}. With the previous estimate, the proof of proposition \ref{prop:sces} is similar to \cite[Proposition 2.6]{MB} so we do not reproduce it here. 

\paragraph{Reconstruction of $|A|^{-1}$:}

On to the reconstruction of $|A|$, equation \eqref{eq:nlda_u} can be recast as 
\begin{align}
    \nabla |A|^{-1} = - |H|^{-\frac{1}{2}} (\nabla ( |H|^\frac{1}{2} H^{pq})\cdot \wtA^2 \nabla u_p) \nabla u_q,
    \label{eq:nlda_m2}
\end{align}
the r.h.s. of which is now completely known. One may thus choose to solve this equation either solving ODE's along curves or taking divergence on both sides and solving a Poisson equation. The stability of such a reconstruction scheme was already stated in theorem \ref{thm:known_isotropy}, whose proof may be found in the very similar (isotropic) context of \cite[Theorem 2.8]{MB}. 

\section{Proof of lemma \ref{lem:Ggamma}} \label{sec:Ggamma}

\paragraph{Isotropic case $\gamma \equiv \sigma\Imm_2$: }

Consider the problem $\nabla\cdot\sigma(x)\nabla u=0$ on $\Rm^2$ with $\sigma(x)$ extended in a continuous manner outside of $X$ and such that $\sigma$ equals $1$ outside of a large ball. Let $q(x)=-\frac{\Delta\sqrt\sigma}{\sigma}$ on $\Rm^2$. We assume that $q\in H^{3+\eps}(\Rm^2)$, which holds if $\sigma-1\in H^{5+\eps}(\Rm^2)$ for some $\eps>0$, i.e., the original $\sigma_{|X}\in H^{5+\eps}(X)$. Note that by Sobolev imbedding, $\sigma$ is of class $\C^4(\overline X)$ while $q$ is of class $\C^2(\overline X)$.

Let $v=\sqrt\sigma u$ so that $\Delta v+qv=0$ on $\Rm^2$. Let $\bm\rho\in\Cm^2$ be of the form $\bm\rho = \rho(\bfk+i\bfk^\perp)$ with $\bfk \in\Sm^1$ and $\bfk^\perp = J\bfk$ so that $\bfk\cdot\bfk^\perp=0$, and $\rho = |\bm\rho|/\sqrt{2}>0$. Thus, $\bm\rho$ satisfies $\bm\rho\cdot\bm\rho=0$ and $e^{\bm\rho\cdot x}$ is a harmonic complex plane wave. Now, it is shown in \cite{BU-IP-10}, following works in \cite{calderon80,Syl-Uhl-87}, that 
\begin{align} 
    v_{\bm\rho} = \sqrt\sigma u_{\bm\rho} = e^{\bm\rho\cdot x}(1+\psi_{\bm\rho}),
\end{align} 
with $(\Delta + q)v_{\bm\rho}=0$ and hence $\nabla\cdot \sigma\nabla u_{\bm\rho}=0$ in $\Rm^2$. Furthermore, with the assumed regularity, \cite[Proposition 3.3]{BU-IP-10} shows the existence of a constant $C$ such that    \begin{align}
    \rho \|\psi_{\bm\rho}\|_{\C^2(\overline X)} \le C \|q\|_{H^{3+\eps}}(X), \quad\text{ and so}\quad \lim_{\rho\to\infty} \|\psi_{\bm\rho}\|_{\C^2(\overline X)} =0.
    \label{eq:limpsirho}
\end{align} 
Taking gradients of the previous equation and rearranging terms, we obtain that   
\begin{align}
    \sqrt{\sigma}\nabla u_{\bm\rho} = e^{\bm\rho\cdot x} (\bm\rho + \bm\varphi_{\bm\rho}), \quad\text{with}\quad \bm\varphi_{\bm\rho} := \nabla\psi_{\bm\rho} + \psi_{\bm\rho} \bm\rho - (1+\psi_{\bm\rho}) \nabla\sqrt{\sigma}. 
    \label{eq:phirho}
\end{align}
Note that $u_{\bm\rho}$ is complex-valued and since $\sigma$ is real-valued, both the real and imaginary parts $u_{\bm\rho}^\Re$ and $u_{\bm\rho}^\Im$ are solutions of $\nabla\cdot(\sigma \nabla u) = 0$. More precisely, we have 
\begin{align}
    \begin{split}
	\sqrt\sigma\nabla u_{\bm\rho}^\Re  &= \rho e^{\rho\bfk\cdot x} \left( (\bfk+ \rho^{-1}\bm\varphi_{\bm\rho}^\Re )\cos(\rho\bfk^\perp\cdot x) - (\bfk^\perp+ \rho^{-1}\bm\varphi_{\bm\rho}^\Im ) \sin (\rho\bfk^\perp\cdot x) \right),  \\
	\sqrt\sigma\nabla u_{\bm\rho}^\Im  &= \rho e^{\rho\bfk\cdot x} \left( (\bfk^\perp + \rho^{-1}\bm\varphi_{\bm\rho}^\Im) \cos(\rho\bfk^\perp\cdot x) + (\bfk + \rho^{-1} \bm\varphi_{\bm\rho}^\Re) \sin (\rho\bfk^\perp\cdot x) \right). 
    \end{split}
    \label{eq:urho}
\end{align}
We now denote $d_{\bm\rho} := \det(\sqrt{\sigma}\nabla u_{\bm\rho}^\Re, \sqrt\sigma\nabla u_{\bm\rho}^\Im)$, and straightforward computations lead to
\begin{align*}
    d_{\bm\rho} = \rho^2 e^{2\rho\bfk\cdot x} (1 + f_{\bm\rho}), \quad f_{\bm\rho} := \rho^{-1} \left( \bfk^\perp \cdot\bm\varphi_{\bm\rho}^\Im + \bfk\cdot\bm\varphi_{\bm\rho}^\Re \right) + \rho^{-2} J\bm\varphi_{\bm\rho}^\Re\cdot \bm\varphi_{\bm\rho}^\Im,
\end{align*}
where $\lim_{\rho\to\infty} \sup_{\overline X} |f_{\bm\rho}| = 0$. Letting $\rho$ so large that $\sup_{\overline X} |f_{\bm\rho}| \le \frac{1}{2}$, the function $d_{\bm\rho}$ is now bounded away from zero over $\overline X$. Taking logarithm and gradient, we now obtain 
\begin{align}
    \nabla\log d_{\bm\rho} = \rho \left(2\bfk + \rho^{-1} \frac{\nabla f_{\bm\rho}}{1+f_{\bm\rho}} \right).
    \label{eq:nldrho}
\end{align}

We now define $\bfk_1, \bfk_2 \in \Sm^1$ such that $\bfk_1\ne \bfk_2$ and define, for $j=1,2$, $\bm\rho_j := \rho (\bfk_j + i \bfk_j^\perp)$. Considering the solutions $(u_{\bm\rho_1}^\Re, u_{\bm\rho_1}^\Im, u_{\bm\rho_2}^\Re, u_{\bm\rho_2}^\Im)$, the previous calculations show that, for $\rho$ large enough, we have
\begin{align*}
    \inf_{\overline X} \min (d_{\bm\rho_1}, d_{\bm\rho_2}) \ge c_0 >0,
\end{align*}
which means that condition \eqref{eq:cond1} is satisfied. Furthermore, using \eqref{eq:nldrho}, we have
\begin{align*}
    \nabla\log \frac{d_{\bm\rho_1}}{d_{\bm\rho_2}} = \rho \left( 2(\bfk_1 - \bfk_2) +  \bfr \right), \quad \bfr := \rho^{-1} \left(\frac{\nabla f_{\bm\rho_1}}{1+f_{\bm\rho_1}} - \frac{\nabla f_{\bm\rho_2}}{1+f_{\bm\rho_2}} \right),
\end{align*}
where the remainder $\bfr$ may be made negligible by virtue of \eqref{eq:limpsirho} and the smoothness assumption on $\sigma$. For $\rho$ such that $\sup_{\overline X} \|\bfr\| \le \|\bfk_1-\bfk_2\|$, the quantity $\nabla\log (d_{\bm\rho_1}/d_{\bm\rho_2})$ never vanishes over $X$ and thus condition \eqref{eq:cond2} is satisfied. In this case, let $\bfg_\rho = \{g_i\}_{1\le i\le 4}$ be the traces of the above CGO solutions $(u_{\bm\rho_1}^\Re, u_{\bm\rho_1}^\Im, u_{\bm\rho_2}^\Re, u_{\bm\rho_1}^\Im)$. These illuminations generate solutions that fulfill conditions \eqref{eq:cond1} and \eqref{eq:cond2}, so $\bfg_\rho \in \G_{\sigma\Imm_2}$. By continuity, any $\bfg$ in an open set (of sufficiently smooth boundary conditions) in the vicinity of  $\bfg_\rho$ also belongs to $\G_{\sigma\Imm_2}$ and the isotropic case is proved. 

\paragraph{General case:}

Since $\tilde\gamma = |\gamma|^{-\frac{1}{2}}\gamma$ is a $\kappa^2$-conformal structure on $\Cm$, \cite[Theorem 10.2.2]{AIM} implies that there exists a unique diffeomorphism $\phi:\Cm\ni z \mapsto \phi(z)=z' \in \Cm$ satisfying the {\em Beltrami system} (normalized at infinity)
\begin{align*}
    D^t \phi\ \tilde\gamma\circ\phi\ D\phi = J_\phi \Imm_2, \quad J_\phi := \det (D\phi), \quad z \in \Cm, \quad \phi(z) \stackrel{z\to\infty}{=} z + \O( z^{-1}),
\end{align*}
alternatively formulated as the following {\em complex Beltrami equation}
\begin{align*}
    \frac{\partial\phi}{\partial\bar z} = \nu(\phi(z)) \overline{ \frac{\partial\phi}{\partial z} }, \where\quad \nu = \frac{\tilde\gamma_{22}- \tilde\gamma_{11} - 2 i \tilde\gamma_{12}}{2 + \tilde\gamma_{11} + \tilde\gamma_{22}}, \quad z\in \Cm. 
\end{align*}
$\nu$ defined above coincides with \eqref{eq:nu_lem} and thus belongs to $\C_{loc}^4$ by assumption. By virtue of \cite[Theorem 15.0.7]{AIM}, this implies that $\phi$ is a $\C_{loc}^5$-diffeomorphism. We further have $J_\phi \ge c_1 >0$ throughout $\Cm$. Using this change of variables and denoting $\nabla\equiv\nabla_{x,y}$ and $\nabla'\equiv\nabla_{x',y'}$ in the sequel with $x'+iy'=z'$ and $x+iy=z$, it is well-known that a function $u$ solves
\begin{align}
    \nabla\cdot (\gamma\nabla u) = 0, \quad z \in \Cm,
    \label{eq:ellz}
\end{align}
if and only if the function $v = u\circ\phi^{-1}$ solves
\begin{align}
    \nabla' \cdot (\phi_\star\gamma \nabla' v) &= 0, \where\quad \phi_\star\gamma (z') = \frac{1}{J_\phi(z)} D\phi^t(z)\ \gamma(\phi(z))\ D\phi(z) \Big|_{z=\phi^{-1}(z')} = \sigma(z') \Imm_2, 
    \label{eq:ellzp}
\end{align}
with $\sigma(z') := |\gamma|^{\frac{1}{2}}\circ\phi^{-1} (z')$. Using the fact that $|\gamma|^{\frac{1}{2}} \in H^{5+\varepsilon}(X)$ by assumption and $\phi\in\C_{loc}^5$, the change of variable for Sobolev spaces implies that $\sigma\in H^{5+\varepsilon} (\phi(X))$. Thus by virtue of the first part of the proof, $\G_{\sigma\Imm_2}\ne\emptyset$. 

Let $\bfg\in \G_{\sigma\Imm_2}$ defined on the boundary $\phi (\partial X) = \partial(\phi(X))$. Then it is easy to see that the illumination $\bfg\circ\phi\in\G_\gamma$, so that $\G_\gamma\ne\emptyset$. Indeed, if for $1\le i\le 4$, $v_i$ designates the unique solution to \eqref{eq:ellzp} over $\phi(X)$ with boundary condition $v_i|_{\partial(\phi(X))} = g_i$, then the function $u_i := v_i\circ\phi$ satisfies \eqref{eq:ellz} over $X$ with boundary condition $u_i|_{\partial X} = g_i\circ\phi$. Using the chain rule $\nabla u = \nabla (v\circ\phi) = D\phi (\nabla' v)\circ\phi$ and properties of the determinant, routine computations yield the following relations, true for every $z = x+iy\in X$
\begin{align*}
    \det(\nabla u_i, \nabla u_j) (z) &= J_\phi (z) \det (\nabla' v_i, \nabla' v_j) (\phi(z)), \quad (i,j)\in \{(1,2), (3,4)\} \\
    JY_{\bfg\circ\phi} (z) &= D\phi(z) JY_\bfg (\phi(z)),
\end{align*}
with $Y_\bfg$ defined in \eqref{eq:cond2}. Since $D\phi$ is everywhere invertible with $J_\phi\ge c_0>0$, the previous relations imply that $(u_1,u_2,u_3,u_4)$ satisfy conditions \eqref{eq:cond1}-\eqref{eq:cond2} over $X$ if and only if $(v_1,v_2,v_3,v_4)$ satisfy \eqref{eq:cond1}-\eqref{eq:cond2} over $\phi(X)$, that is, $\bfg\in\G_{\sigma\Imm_2}$ if and only if $\bfg\circ\phi\in\G_\gamma$. The lemma is proved. 

\begin{remark}[On the regularity of $\sigma$] The existence of CGO solutions can be established in two dimensions assuming mere boundedness of the isotropic diffusion coefficient $\sigma$, as was established in \cite{AP}. However, due to the necessity of estimate \eqref{eq:limpsirho} for our purposes, we need the results in \cite{BU-IP-10}, which in turn require higher regularity on $\sigma$.
\end{remark}

\section{Numerical examples}\label{sec:numerics}
In order to validate the reconstruction algorithms presented in the previous sections, we implemented a forward solver for the anisotropic diffusion equation on a two-dimensional square grid, using a finite difference method implemented with the software MatLab. 

We use a {\tt N+1}$\times${\tt N+1} square grid with {\tt N=128}, the tensor product of the equi-spaced subdivision {\tt x = -1:h:1} with {\tt h = 2/N}. Partial derivatives are performed using the operators {\tt DX = kron(D,I)} for $\partial/\partial x$ and {\tt DY = kron(I,D)} for $\partial/\partial y$, where {\tt D} designates the one-dimensional finite difference derivative matrix and {\tt I=speye(N+1)} is the {\tt N+1}$\times${\tt N+1} (sparsified) identity matrix. {\tt D} is the second-order centered stencil {\tt [-1 0 1]/(2h)} with, at the boundary {\tt D(1,1:3) = [-3 4 -1]/(2h)} and {\tt D(N+1,N-1:N+1) = [1 -4 3]/(2h)}. 

In the following examples, the conductivity tensor is described by the triplet of scalar functions $(|\gamma|^{\frac{1}{2}}, \xi, \zeta)$ such that $\gamma = |\gamma|^{\frac{1}{2}} \tilde\gamma (\xi,\zeta)$ with $\tilde \gamma$ given in \eqref{eq:paramA2}. Note that $|A|=|\gamma|^{\frac{1}{2}}$. The anisotropy coefficients $(\xi,\zeta)$ in Fig. \ref{fig:coeffs}\subref{fig:coeffs1}\&\subref{fig:coeffs2} are used in all experiments, while the determinant $|\gamma|^{\frac{1}{2}}$ may be either one of the two functions displayed in Figs. \ref{fig:coeffs}\subref{fig:coeffs3}\&\subref{fig:coeffs4}.

\begin{figure}[htpb]
    \centering 
    \subfigure[$\xi$]{
    \includegraphics[width=0.22\textwidth]{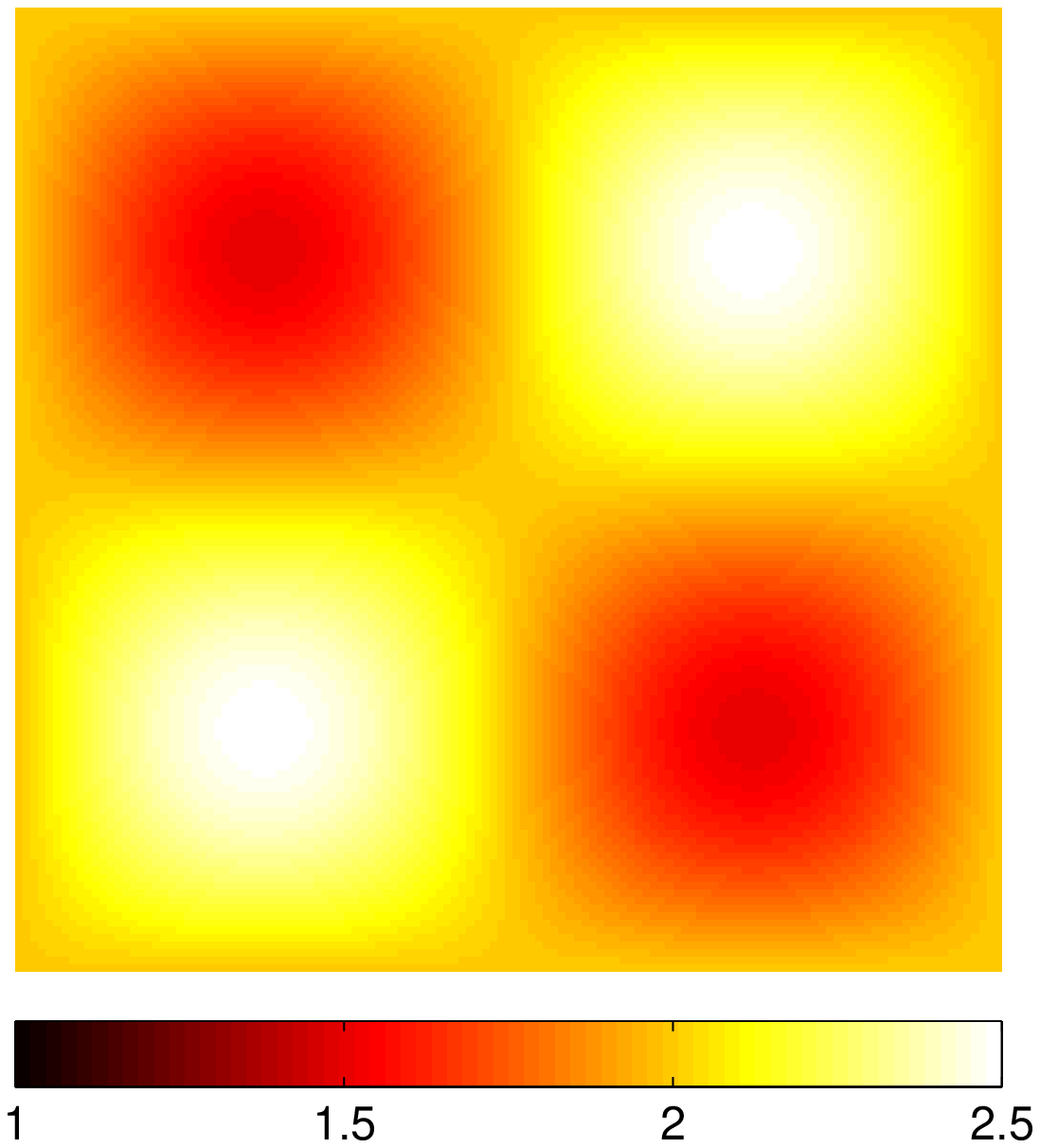}
    \label{fig:coeffs1}
    }
    \subfigure[$\zeta$]{
    \includegraphics[width=0.22\textwidth]{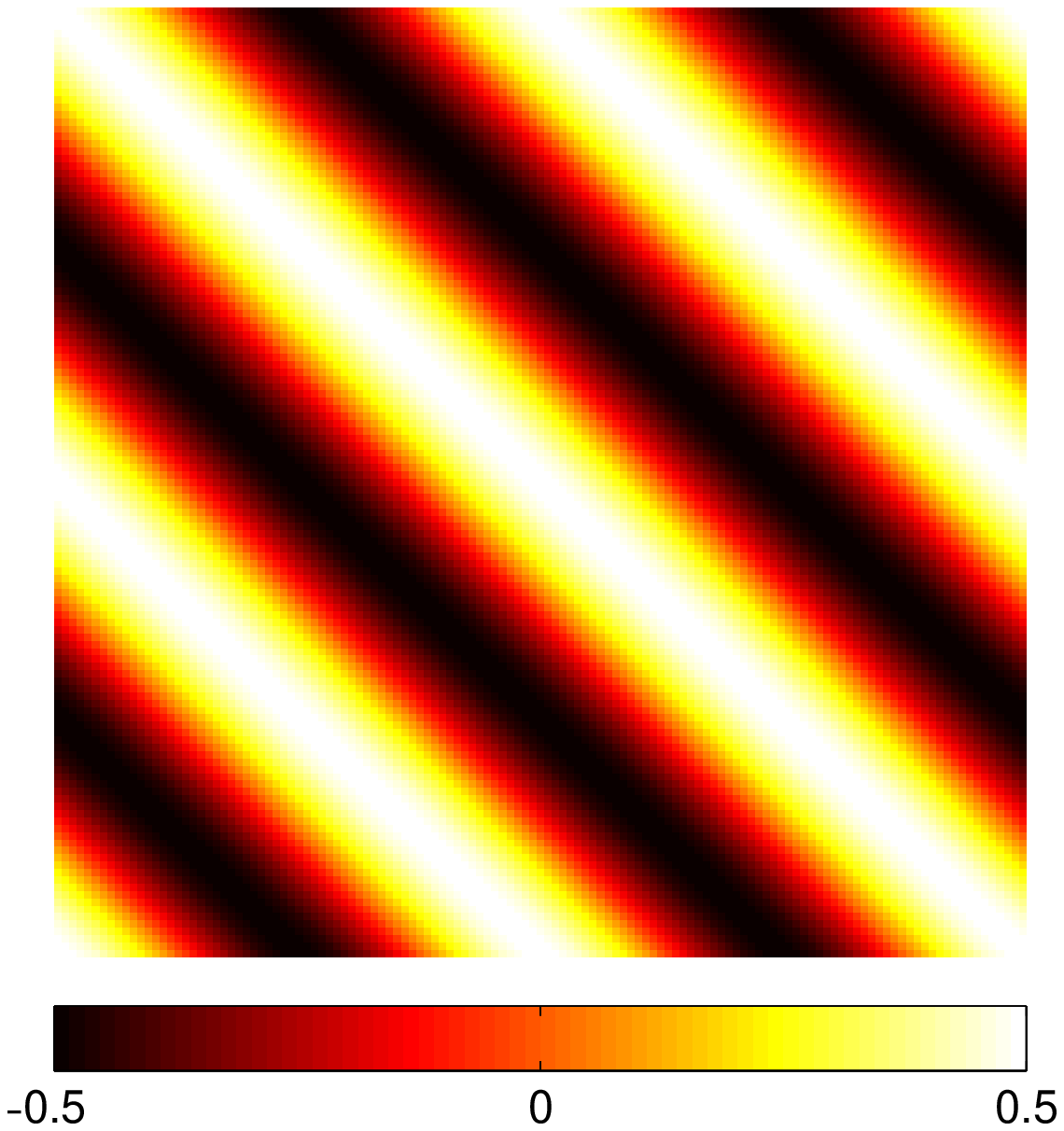}
    \label{fig:coeffs2}
    }
    \subfigure[$|\gamma|^{\frac{1}{2}}$ smooth]{
    \includegraphics[width=0.22\textwidth]{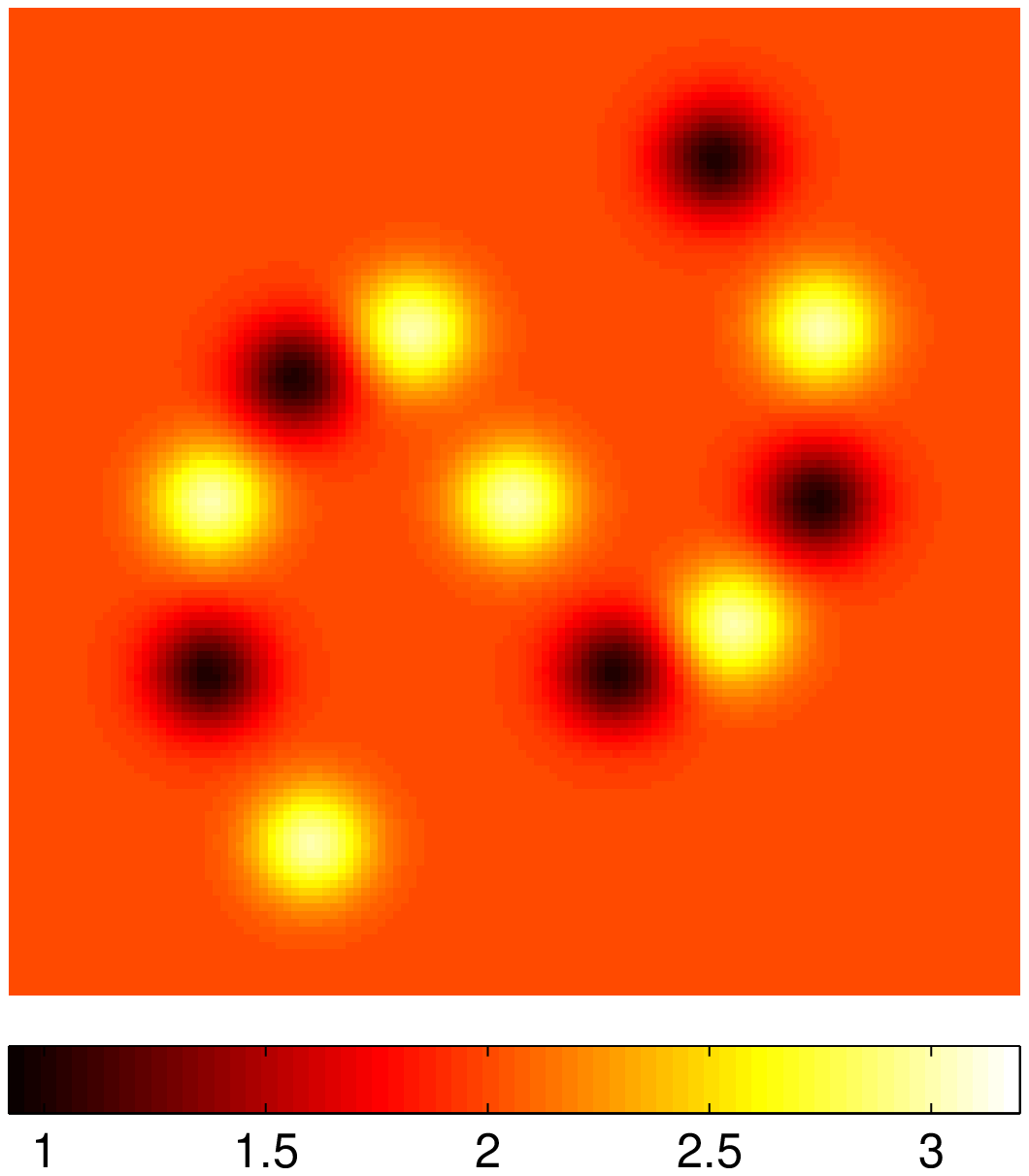}
    \label{fig:coeffs3}
    }
    \subfigure[$|\gamma|^{\frac{1}{2}}$ rough]{
    \includegraphics[width=0.22\textwidth]{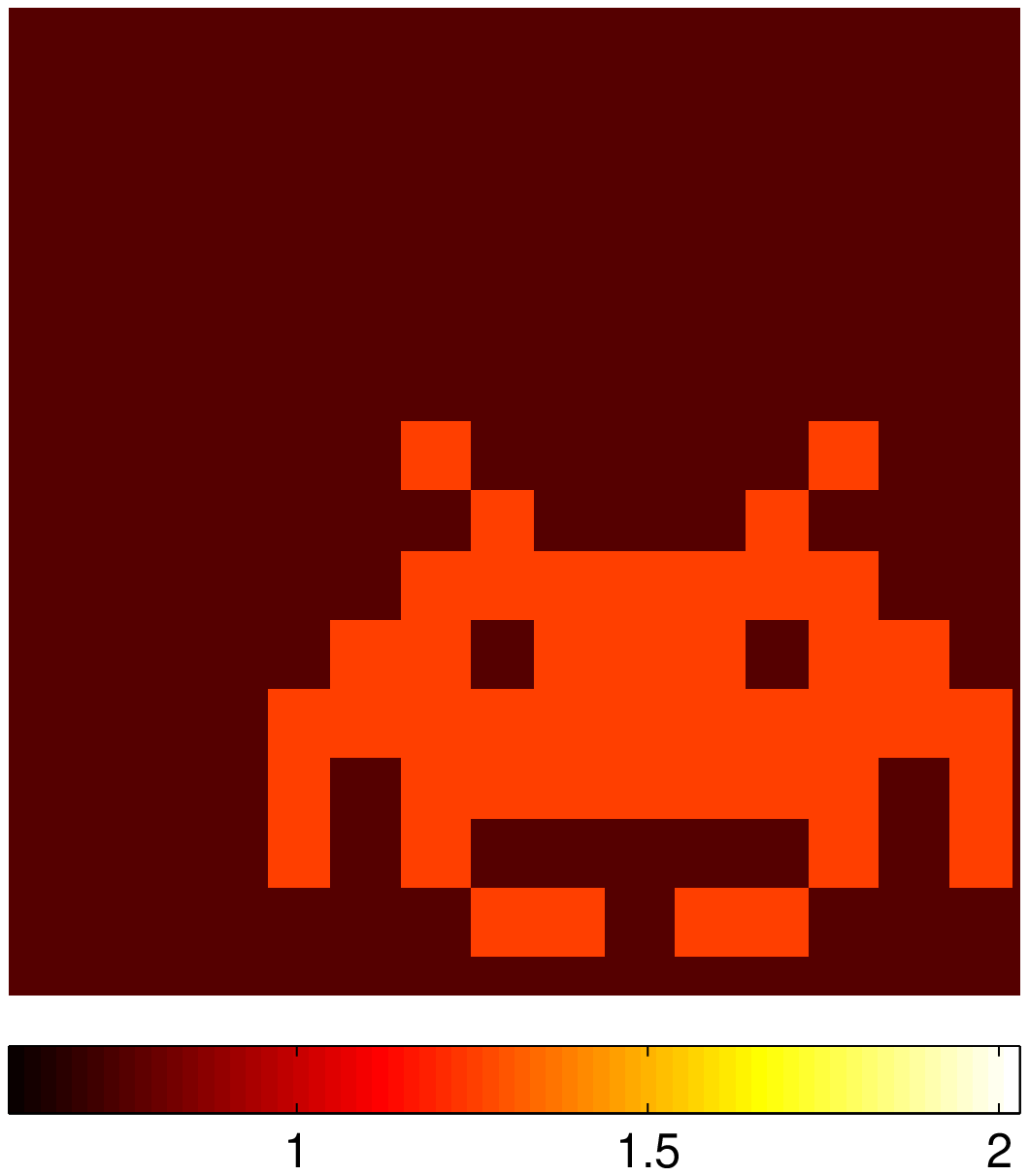}
    \label{fig:coeffs4}
    }
    \caption{The coefficients used in the numerical simulations}%for the computations below}
    \label{fig:coeffs}
\end{figure}

We sometimes perturb the internal functionals $H_{ij}$ with random noise of the form 
\begin{align}
    H_{noisy} = H \text{{\tt .* }} (1 + \frac{\alpha}{100} \text{ {\tt random}}),
    \label{eq:noise}
\end{align}
where $\alpha$ is the noise level and {\tt random} is a {\tt N+1}$\times${\tt N+1} matrix of random entries with uniform density over $[-1,1]$, to which we have applied a slight regularization by convolving it with the averaging filter {\tt ones(3)/9} (e.g. using the MatLab {\tt imfilter} function). 

\paragraph{Reconstruction of the anisotropy $\wtA$ from noiseless data and $m=3$ solutions.}

%\noindent{\bf Noiseless data and $m=3$ solutions.}
Let the two systems $S^{(1)}=(S_1,S_2)$ and $S^{(2)}=(S_2,S_3)$ be associated with $(g_1,g_2)$ and $(g_2,g_3)$ respectively, where $(g_1,g_2,g_3)$ are given by
\begin{align*}
    (g_1,g_2,g_3)(x,y) = (x+y, y + 0.1 y^2, -x+y) , \quad (x,y) \in \partial [-1,1]^2.
\end{align*}
We first compute the data $\{H_{ij}\}$ by solving the three forward problems \eqref{eq:diffi} and then computing $H_{ij} = \gamma\nabla u_i\cdot\nabla u_j$ over the grid. We then compute the vector fields $X$ and $Y$ in the Gram-Schmidt case, where the transfer matrices $T^{(l)} = \{t_{ij}^{(l)}\}_{i,j=1}^2$ such that $R^{(l)} = S^{(l)}T^{(l)T}$ for $l=1,2$ are given by
\begin{align*}
    T^{(1)} = \left[
    \begin{array}{cc}
	H_{11}^{-\frac{1}{2}} & 0 \\
	-H_{12} H_{11}^{-\frac{1}{2}} d_1^{-1} & H_{11}^{\frac{1}{2}} d_1^{-1}
    \end{array}
    \right] \qandq 
    T^{(2)} = \left[
    \begin{array}{cc}
	H_{22}^{-\frac{1}{2}} & 0 \\
	-H_{23} H_{22}^{-\frac{1}{2}} d_2^{-1} & H_{22}^{\frac{1}{2}} d_2^{-1}
    \end{array}
    \right],
\end{align*}
with $d_1 := (H_{11}H_{22}-H_{12}^2)^\frac{1}{2}$ and $d_2 = (H_{22}H_{33}-H_{23}^2)^\frac{1}{2}$. $X_\bfg$ and $Y_\bfg$ then admit the following expressions:
\begin{align}
    X_\bfg = -\frac{H_{22}}{2} \left( \frac{1}{d_1}\nabla \left( \frac{H_{12}}{H_{22}}  \right) + \frac{1}{d_2} \nabla\left( \frac{H_{23}}{H_{22}} \right) \right) \qandq Y_\bfg = -\frac{1}{2} J \nabla (\log d_2 - \log d_1).
    \label{eq:XY}
\end{align}
Once $X_\bfg, Y_\bfg$ are generated numerically, we then reconstruct the functions $(\xi,\zeta)$ that characterize $\tilde\gamma =\wtA^2$ via formulas \eqref{eq:reconsxizeta}. 

The simulation of $\log d_1-\log d_2$ that appears in \eqref{eq:XY} is presented for the smooth $\gamma$ in the absence of noise on Fig. \ref{fig:2noise}\subref{fig:2noise1} and in the presence of one percent of noise on Fig. \ref{fig:2noise}\subref{fig:2noise2}.
%\begin{figure}[ht]
%    \begin{center}
%	\includegraphics[width=0.22\textwidth]{sqdetDsmooth}
%	\includegraphics[width=0.22\textwidth]{sqdetDrough}
%	\includegraphics[width=0.22\textwidth]{logd1d2}
%	\includegraphics[width=0.22\textwidth]{logd1d2noisy1pc}
%    \end{center}
%    \caption{From left to right: a smooth $|\gamma|^\frac{1}{2}$, a rough $|\gamma|^{\frac{1}{2}}$, the function $\log d_1-\log d_2$ with true and noisy data ($\alpha = 1\%$).}
%    \label{fig:1}
%\end{figure}

The reconstruction of $(\xi,\zeta)$ is performed for both forms of $|\gamma|^\frac{1}{2}$ in Figs.\ref{fig:coeffs}\subref{fig:coeffs3}\&\subref{fig:coeffs4} and the results are presented in Fig.\ref{fig:2}. In the smooth case, the reconstructed $\xi,\zeta$ cannot be distinguished visually from the exact coefficients and are thus not represented. Relative $L^2$ ($L^\infty$) errors for $\xi$ and $\zeta$ are $0.1\%$ ($8.6\%$) and $0.8\%$ ($13.7\%$), respectively. In the second case, the singularities of $|\gamma|^\frac{1}{2}$ create artifacts on the reconstructions, see Fig. \ref{fig:2}\subref{fig:22}\&\subref{fig:26}, and the relative errors for $\xi$ and $\zeta$ increase to $5.4\%$ ($99\%$) and $15.8\%$ ($167\%$), respectively. %The reconstructed $(\xi,\zeta)$ are plotted at the interior of the domain. 
\begin{figure}[ht]
    \centering 
    \subfigure[true $\xi$]{
    \includegraphics[width=0.22\textwidth]{xi}
    \label{fig:21}
    }
    \subfigure[$\xi$ (``1'' on \subref{fig:24})]{
    \includegraphics[width=0.22\textwidth]{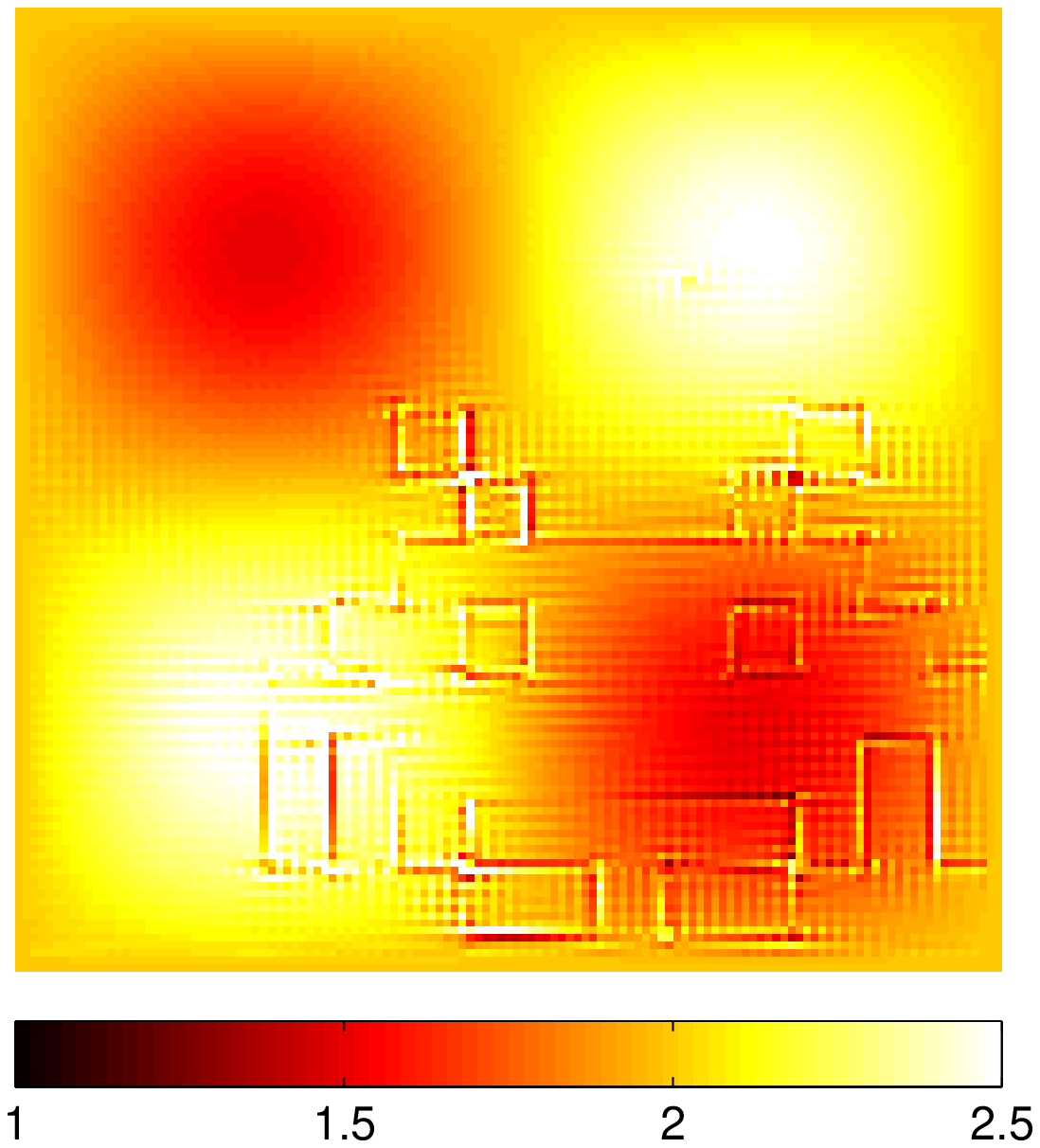}
    \label{fig:22}
    }
    \subfigure[$\xi$ (``2'' on \subref{fig:24})]{
    \includegraphics[width=0.22\textwidth]{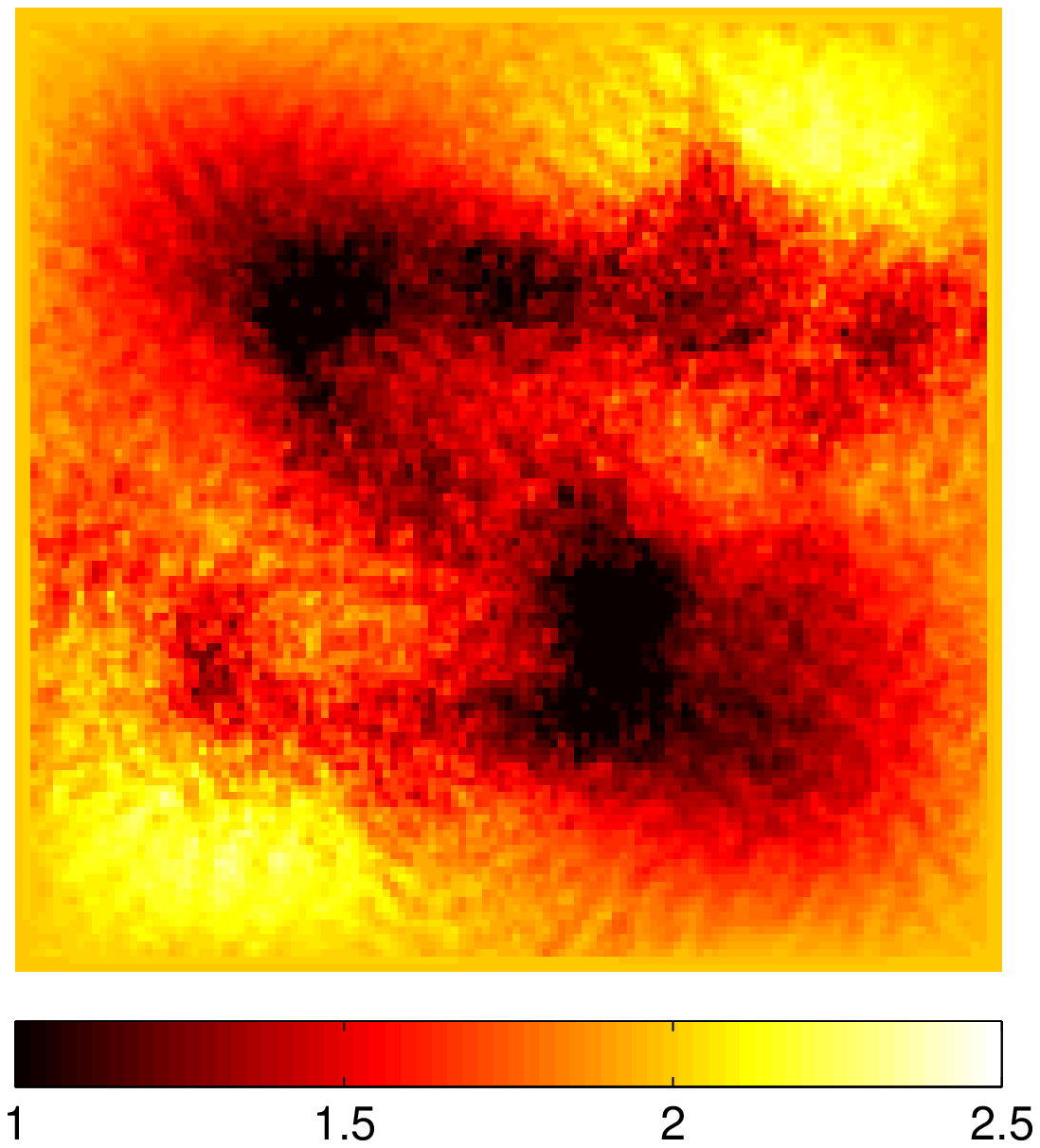}
    \label{fig:23}
    }
    \subfigure[$\xi$ at $\{x=0.5\}$]{
    \includegraphics[width=0.22\textwidth]{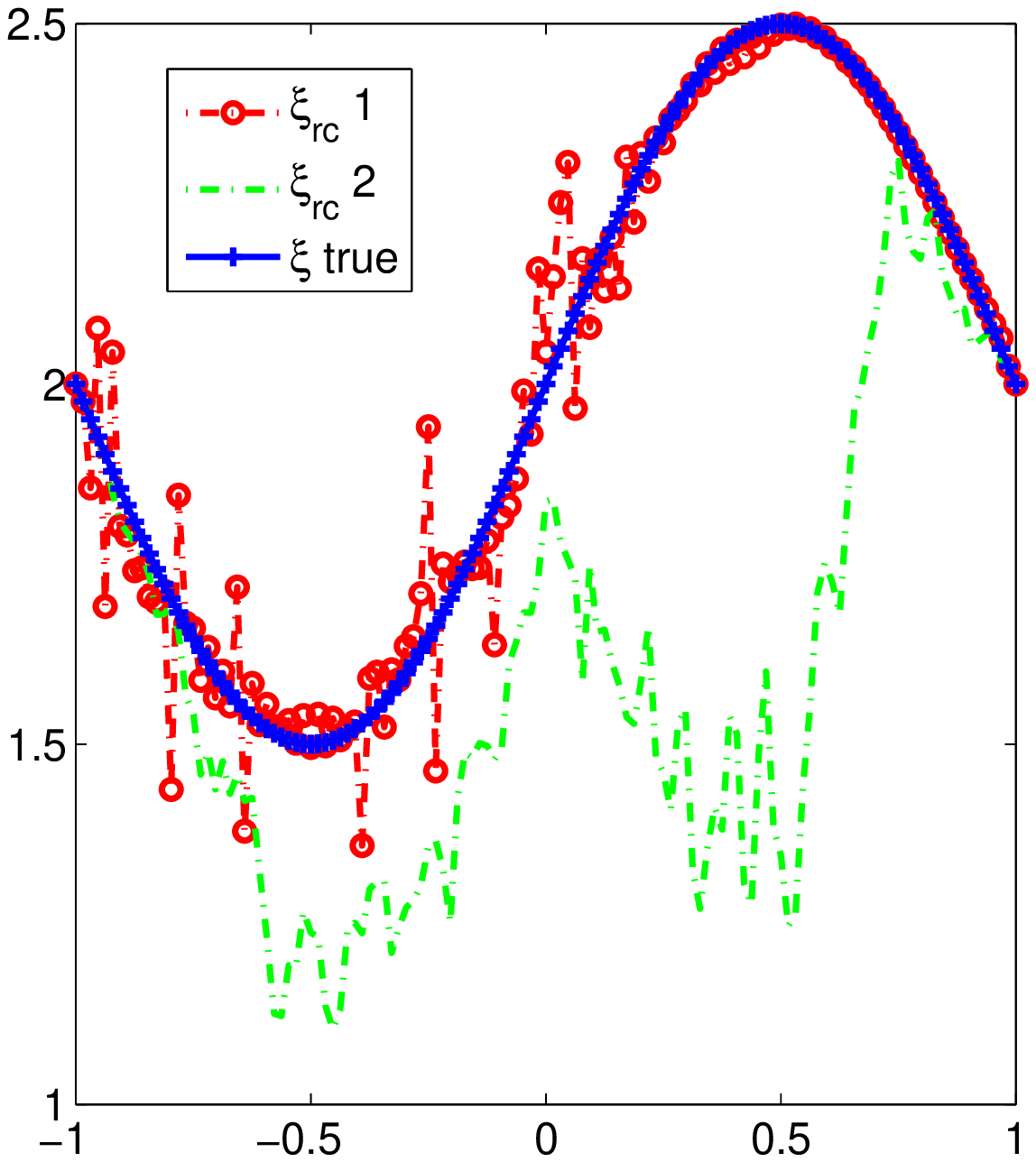}
    \label{fig:24}
    }    
    \subfigure[true $\zeta$]{
    \includegraphics[width=0.22\textwidth]{zeta}
    \label{fig:25}
    }
    \subfigure[$\zeta$ (``1'' on \subref{fig:28})]{
    \includegraphics[width=0.22\textwidth]{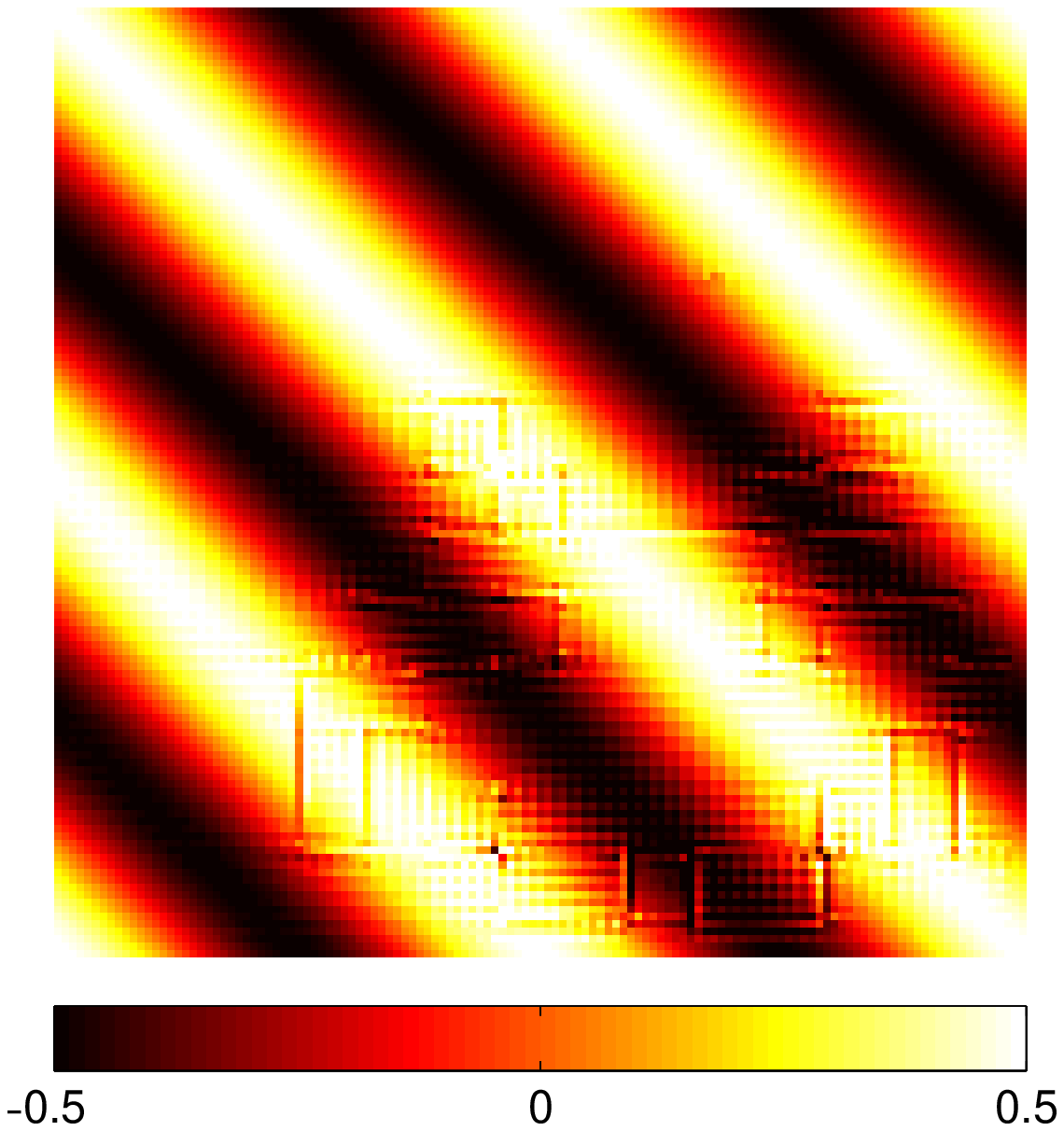} 
    \label{fig:26}
    }
    \subfigure[$\zeta$ (``2'' on \subref{fig:28})]{
    \includegraphics[width=0.22\textwidth]{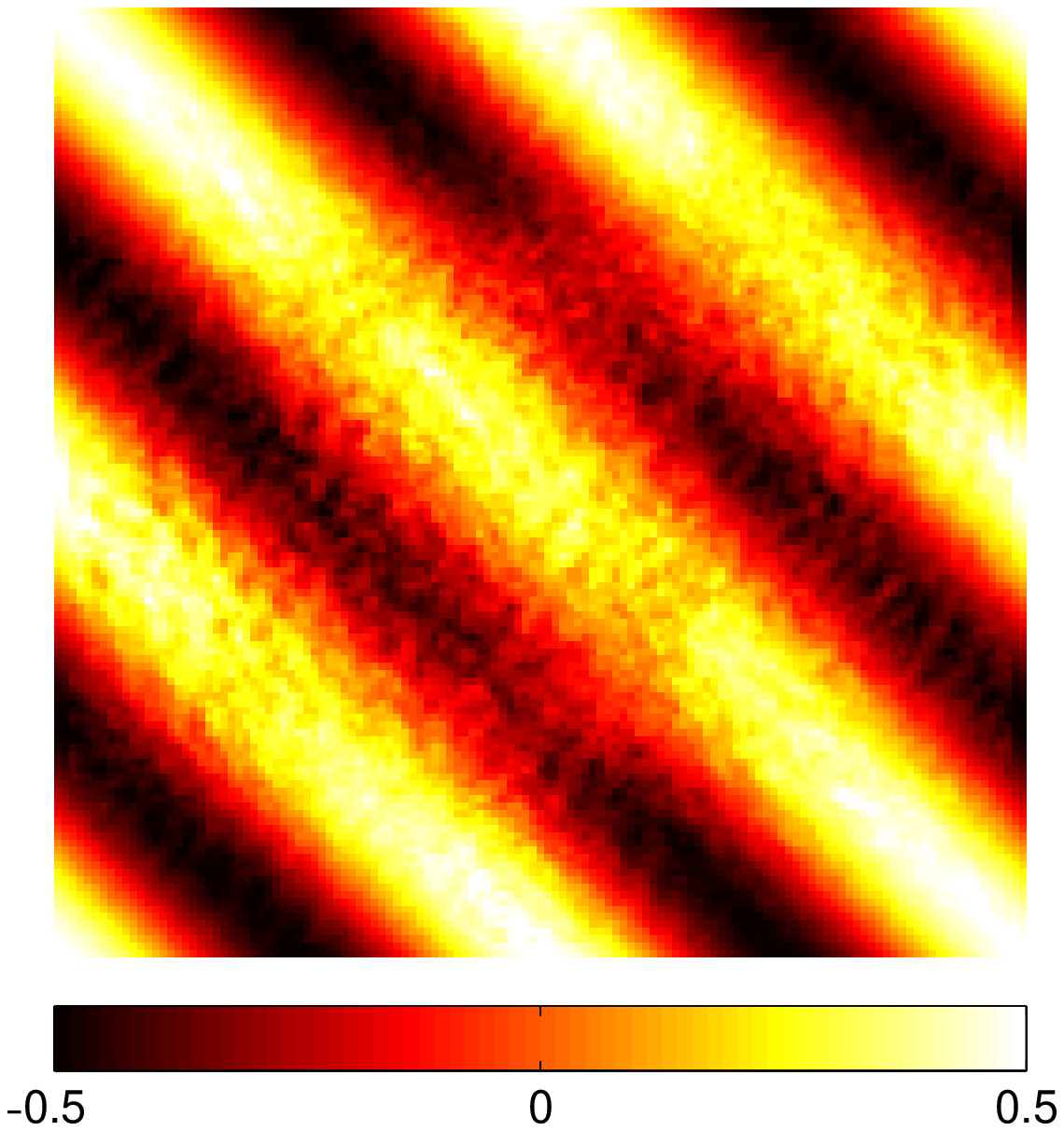}
    \label{fig:27}
    }
    \subfigure[$\zeta$ at $\{x=0.5\}$]{
    \includegraphics[width=0.22\textwidth]{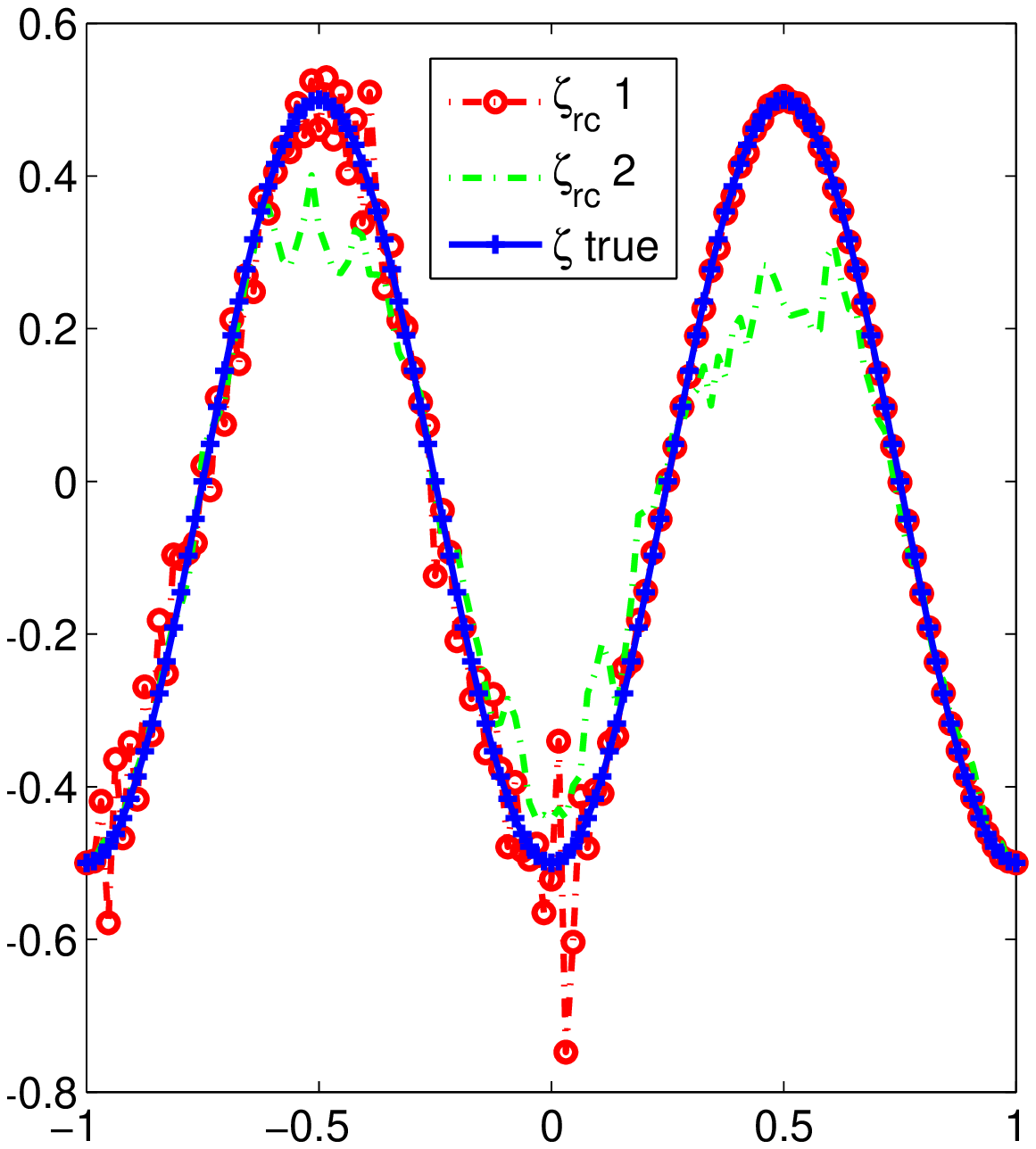}	
    \label{fig:28}
    }    
    \caption{Anisotropy reconstructions. \subref{fig:22}\&\subref{fig:26}: with rough $|\gamma|^{\frac{1}{2}}$ and noiseless data. \subref{fig:23}\&\subref{fig:27}: with smooth $|\gamma|^{\frac{1}{2}}$, noisy data ($\alpha=0.1\%$) and $p=100$ measurements.}
    \label{fig:2}
\end{figure}

\paragraph{Reconstruction of the anisotropy $\wtA$ from noisy data and $m=3p$ solutions, $p \ge 1$.}

Figs. \ref{fig:2noise}\subref{fig:2noise1}\&\subref{fig:2noise2} show that even very small amounts of noise in the three available internal functionals significantly affect the reconstruction of the anisotropic coefficients. 
The presence of noise creates local extrema for the function $\log d_1 - \log d_2$ so that its gradient and $Y_\bfg$ may vanish and \eqref{eq:reconsxizeta} may no longer hold.

To address this lack of robustness, we increase the number of available internal functionals to $3p$ for $p\geq1$, which correspond to the boundary conditions $(\bfg_1,\dots,\bfg_p)$. Instead of solving the linear system \eqref{eq:sysxizeta}, we solve the normal equation (in the $L^2$-minimizing sense) associated with the over-determined $6p$ linear equations, each pair of which looks like \eqref{eq:sysxizeta} with $\bfg = \bfg_j$ for $1\le j\le p$. The normal system reads
\begin{align*}
    \Xi \left[
    \begin{array}{c}
	\xi \\ \zeta	
    \end{array}
    \right] = \sum_{j=1}^p \left[
    \begin{array}{c}
	X_{\bfg_j}\cdot Y_{\bfg_j} \\ 0
    \end{array}
    \right], \quad \Xi := \sum_{j=1}^p \left[
    \begin{array}{cc}
	(x_{\bfg_j}^1)^2 + (y_{\bfg_j}^2)^2 & \text{sym} \\
	x_{\bfg_j}^1 x_{\bfg_j}^2 - y_{\bfg_j}^1 y_{\bfg_j}^2 & (x_{\bfg_j}^2)^2 + (y_{\bfg_j}^1)^2
    \end{array}
    \right],
\end{align*}
and thus may be inverted as
\begin{align}
    (\xi, \zeta) = (\det \Xi)^{-1} \Big( \sum_{j=1}^p X_{\bfg_j}\cdot Y_{\bfg_j} \Big) (\Xi_{22}, -\Xi_{21}). 
    \label{eq:xizetap}
\end{align}
Results are shown after $p=100$ iterations in Fig. \ref{fig:2}, where for $1\le j\le p$, we used the following boundary conditions
\begin{align*}
    g_{1,j}(x,y) &= (3 + (x,y)\cdot\hat\beta)^{1 + j/p},\quad \beta := \frac{2\pi j}{p},\quad \hat\beta:=(\cos\beta,\sin\beta), \\
    g_{2,j}(x,y) &= (x,y)\cdot \widehat{\beta+\frac{\pi}{4}} + 0.01 \Big( 2 + (x,y)\cdot \widehat{\beta+\frac{\pi}{4}} \Big)^{2+\frac jp}, \quad  g_{3,j}(x,y) = \left( 3 + (x,y)\cdot\widehat{\beta+\frac{\pi}{2}}\right)^{1 + \frac jp}.
\end{align*}
The relative $L^2$ errors for $\xi$ and $\zeta$ are $22\%$ and $27\%$, respectively. The effect of the number $p$ of measurements on the reconstruction can been on Fig. \ref{fig:2noise}. We observe a very slow convergence, which is consistent with the central limit theory, as $p$ increases. The slow convergence may be considerably sped up by adding more constraining prior information on $(\xi,\zeta)$ such as for instance regularity or sparsity constraints.  We do not explore this aspect here.

%$L^1$-minimization techniques depending on additional prior information on these coefficients, we do not explore this here. 

\begin{figure}[htpb]
    \centering 
    \subfigure[$\log d_1-\log d_2$($\alpha=0\%$)]{
    \includegraphics[width=0.22\textwidth]{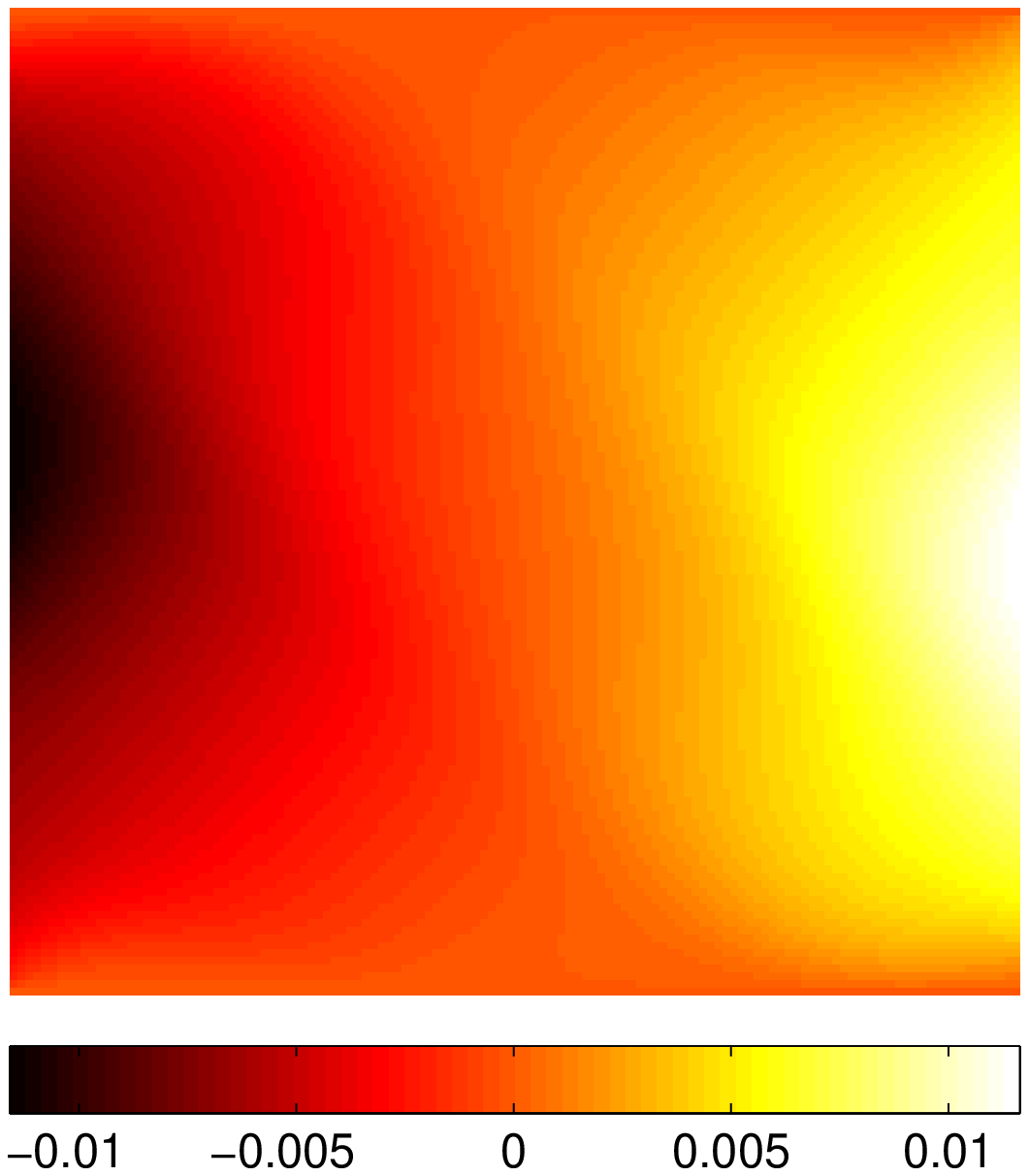}
    \label{fig:2noise1}
    }
    \subfigure[$\log d_1-\log d_2$($\alpha=1\%$)]{
    \includegraphics[width=0.22\textwidth]{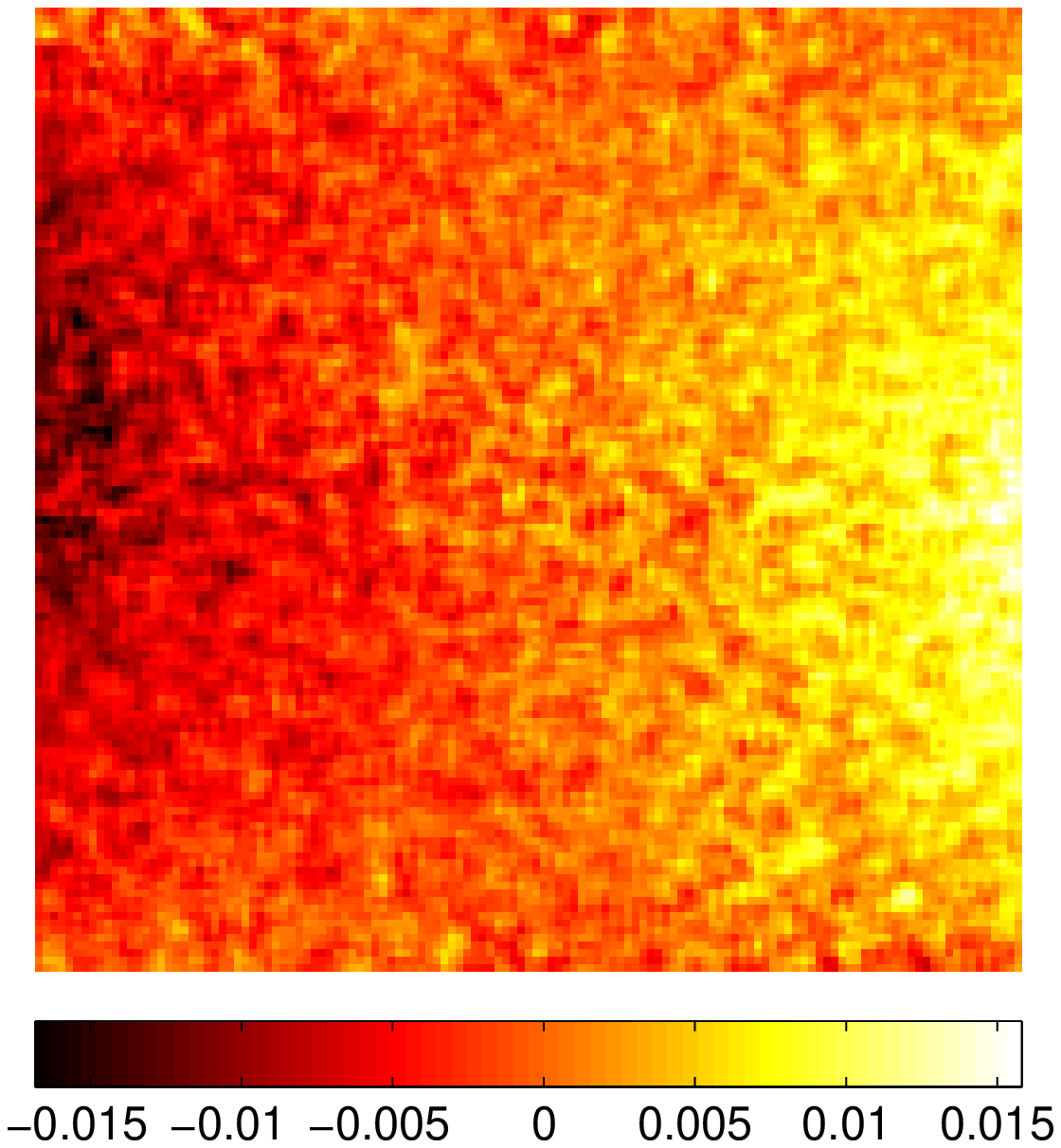}
    \label{fig:2noise2}
    }
    \subfigure[$\xi$ at $\{x=0.5\}$]{
    \includegraphics[width=0.22\textwidth]{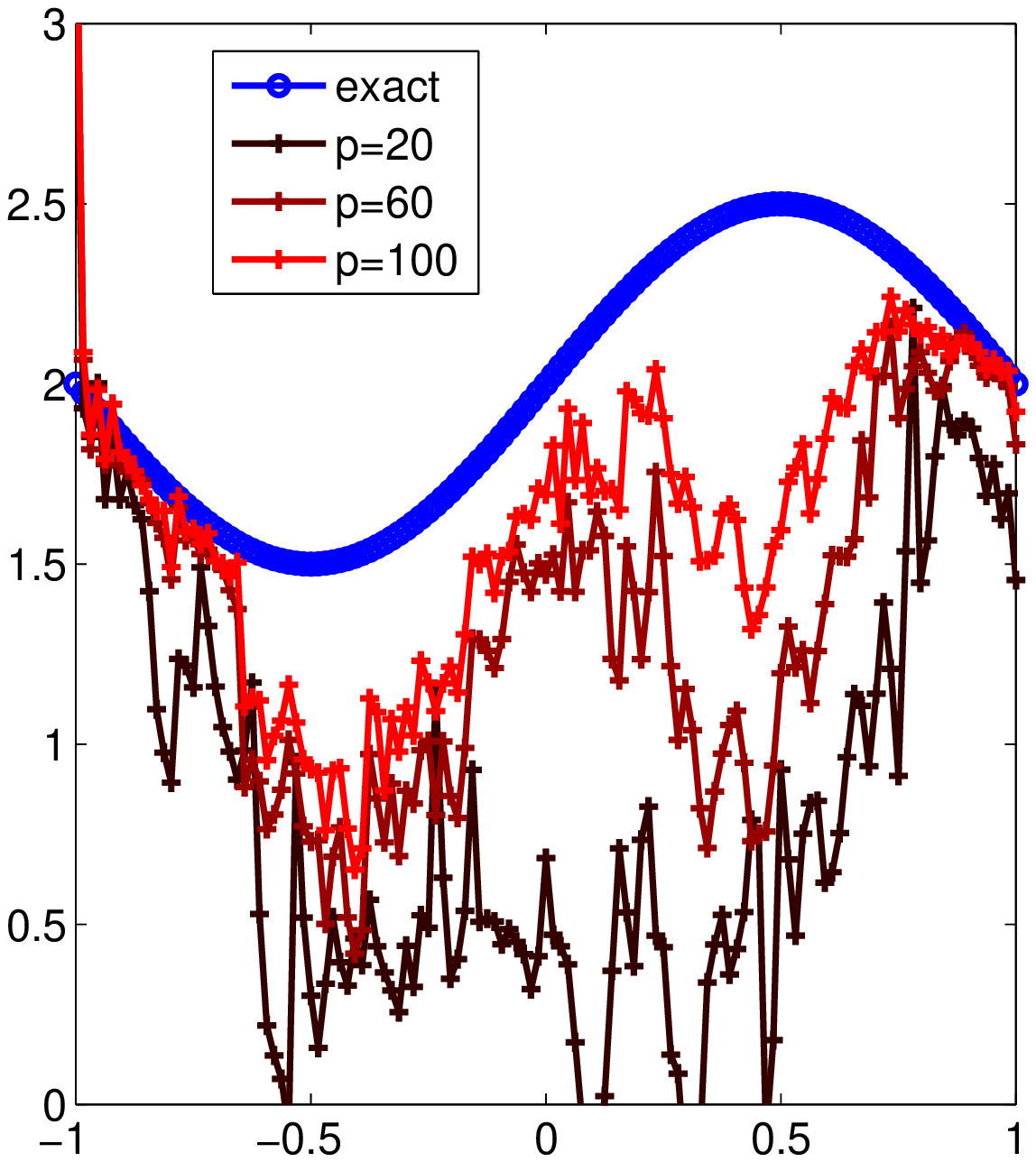}
    \label{fig:2noise3}
    }
    \subfigure[$\zeta$ at $\{x=0.5\}$]{
    \includegraphics[width=0.22\textwidth]{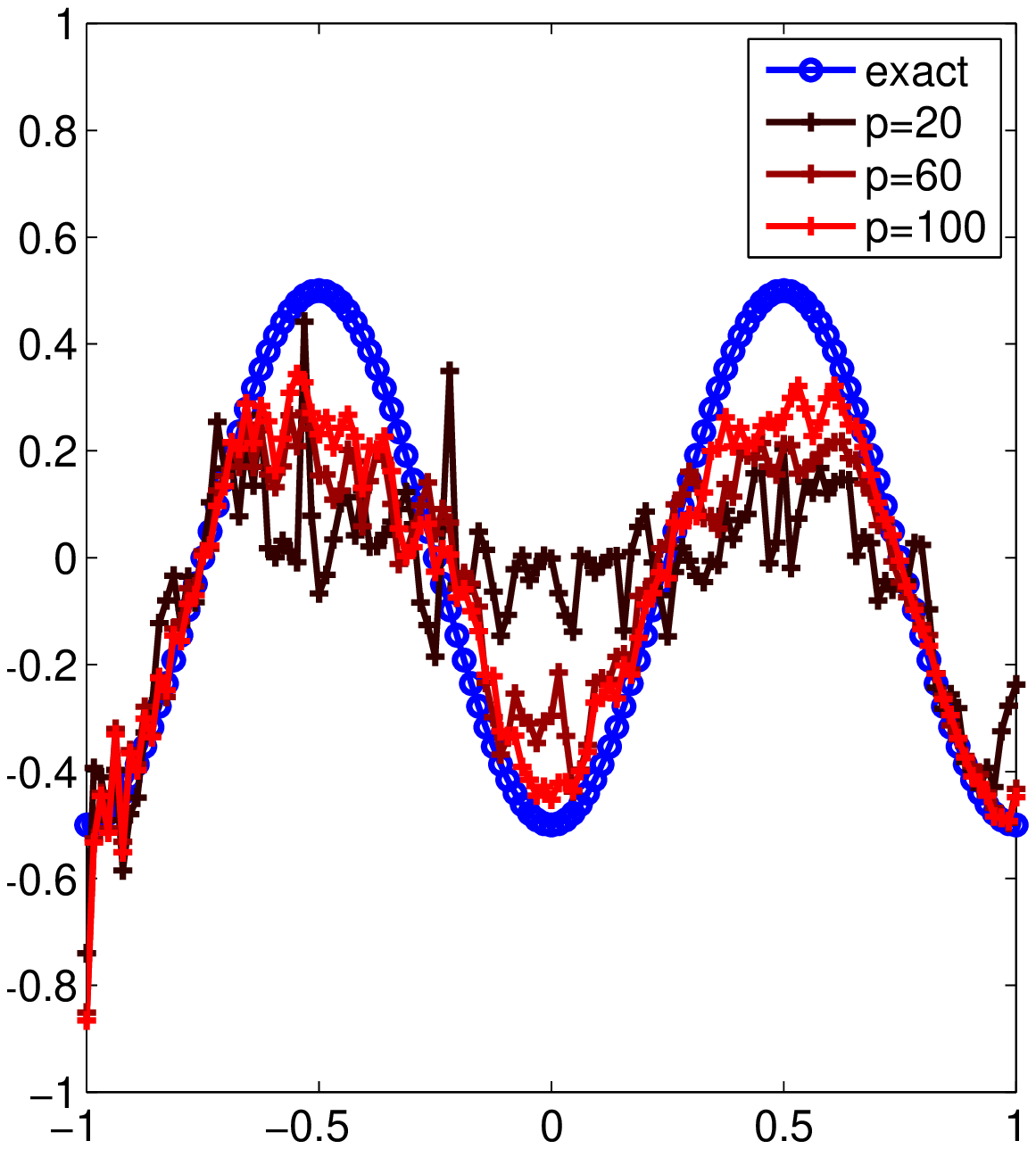}
    \label{fig:2noise4}
    }
    \caption{\subref{fig:2noise1}\&\subref{fig:2noise2}: influence of the noise on the function $\log d_1 - \log d_2$. \subref{fig:2noise3}\&\subref{fig:2noise4}: cross sections of $\xi$ and $\zeta$ using reconstruction formulas \eqref{eq:xizetap} for a few values of $p$.}
    \label{fig:2noise}
\end{figure}

\paragraph{Reconstruction of $|\gamma|^{\frac{1}{2}}$ via $(\theta, \log |\gamma|^{\frac{1}{2}})$.}
We now assume that the anisotropy $\wtA$ is known or reconstructed, and solve for $\theta$ by first applying the divergence operator to \eqref{eq:gradtheta_alone} and then for $\log |\gamma|^{\frac{1}{2}}$ by applying the divergence operator to \eqref{eq:nla2d2}. For both Poisson equations, we use exact data as boundary conditions. In the Gram-Schmidt case, $\hat\theta = \widehat{\wtA S_1}$ and the vector fields $\{V_{ij}\}$ are given by (recall that $d:= (H_{11}H_{22}-H_{12}^2)^\frac{1}{2}$) 
\begin{align*}
    V_{11} = \nabla\log H_{11}^{-\frac{1}{2}}, \quad V_{12} = 0,\quad V_{21} = -\frac{H_{11}}{d}\nabla\left( \frac{H_{12}}{H_{11}} \right), \quad V_{22} = \nabla \log (H_{11}^\frac{1}{2} d^{-1}),
\end{align*}
where the data come from solutions $(u_1,u_2)$ with boundary conditions $(g_1,g_2)(x,y) = (x,y)$ for $(x,y)\in\partial (-1,1)^2$. 

\begin{figure}[htpb]
    \centering 
    \subfigure[$H_{11}$ ($\alpha=0\%$)]{
    \includegraphics[width=0.22\textwidth]{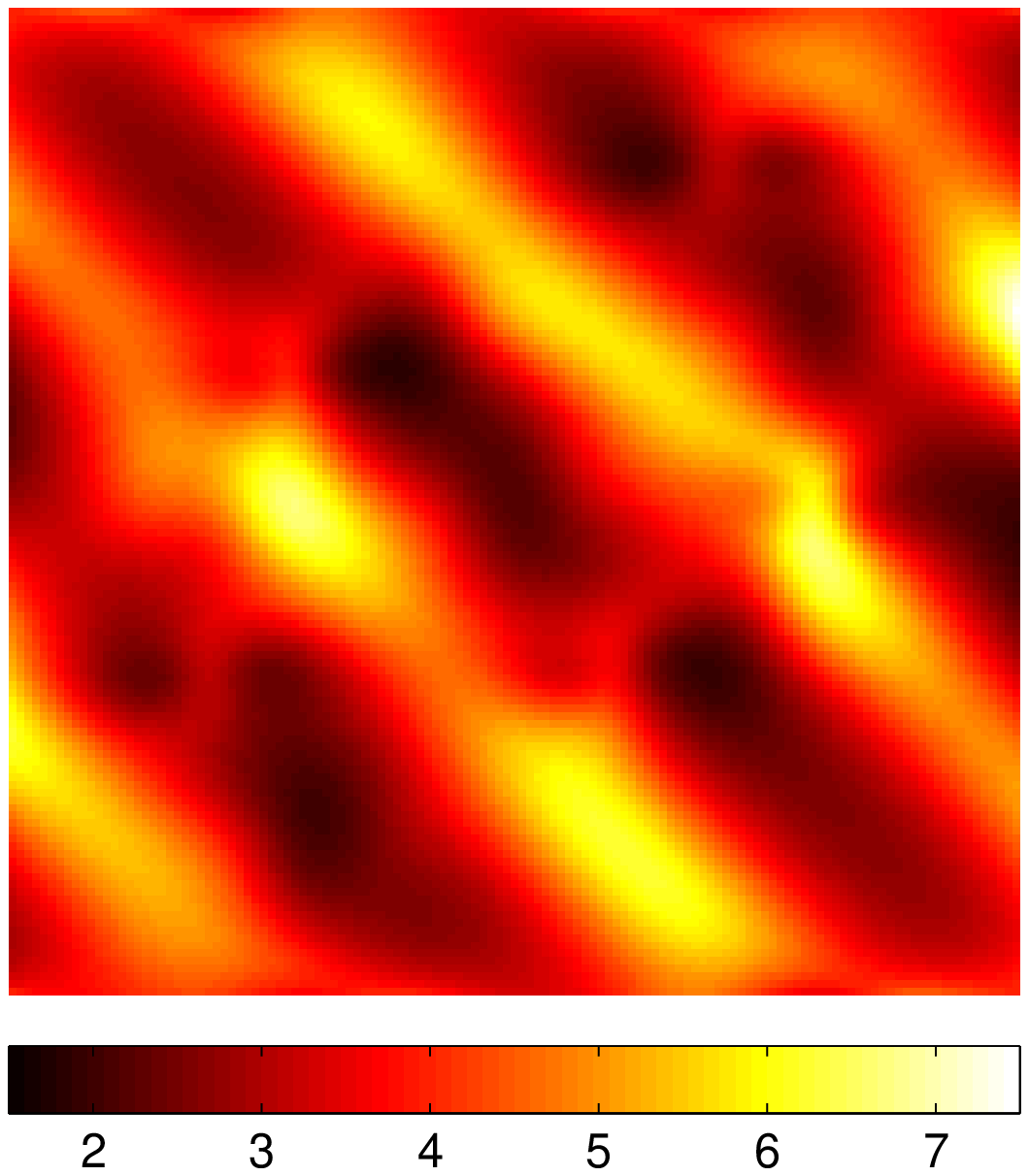}
    \label{fig:31}
    }
    \subfigure[$H_{11}$ ($\alpha=30\%$)]{
    \includegraphics[width=0.22\textwidth]{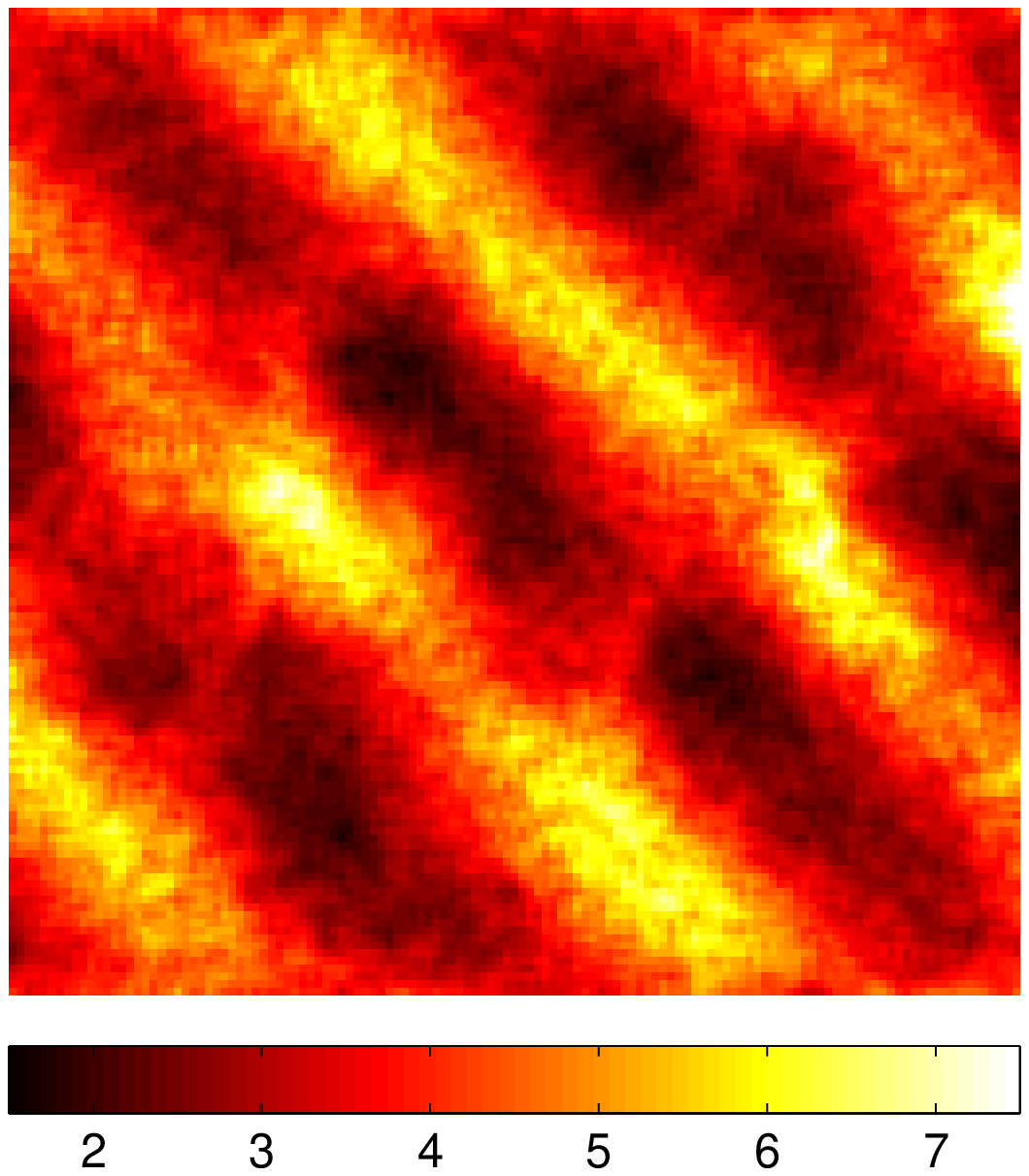}
    \label{fig:32}
    }
    \subfigure[$H_{12}$ ($\alpha=0\%$)]{
    \includegraphics[width=0.22\textwidth]{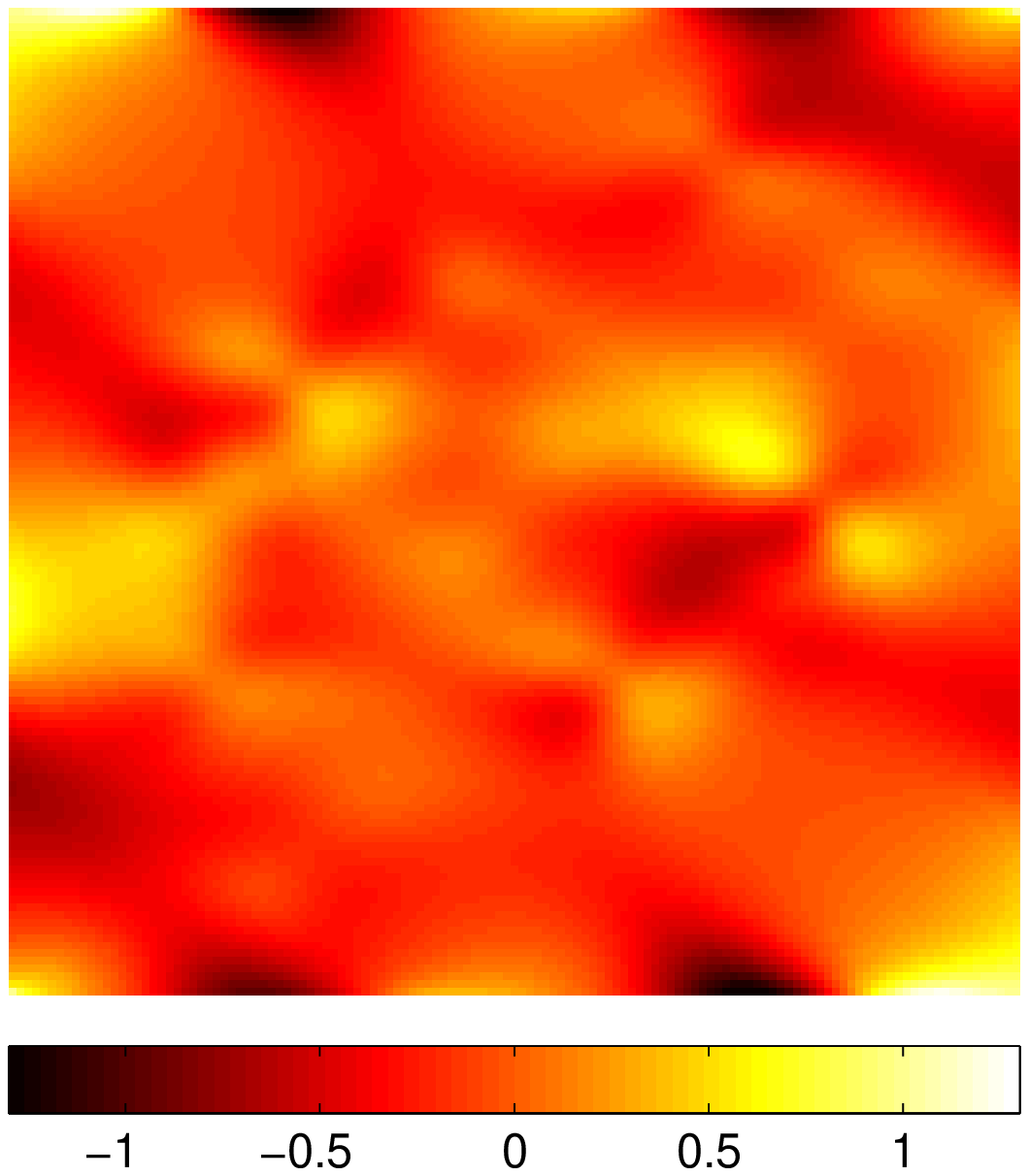}	
    \label{fig:33}
    }
    \subfigure[$H_{12}$ ($\alpha=0\%$)]{
    \includegraphics[width=0.22\textwidth]{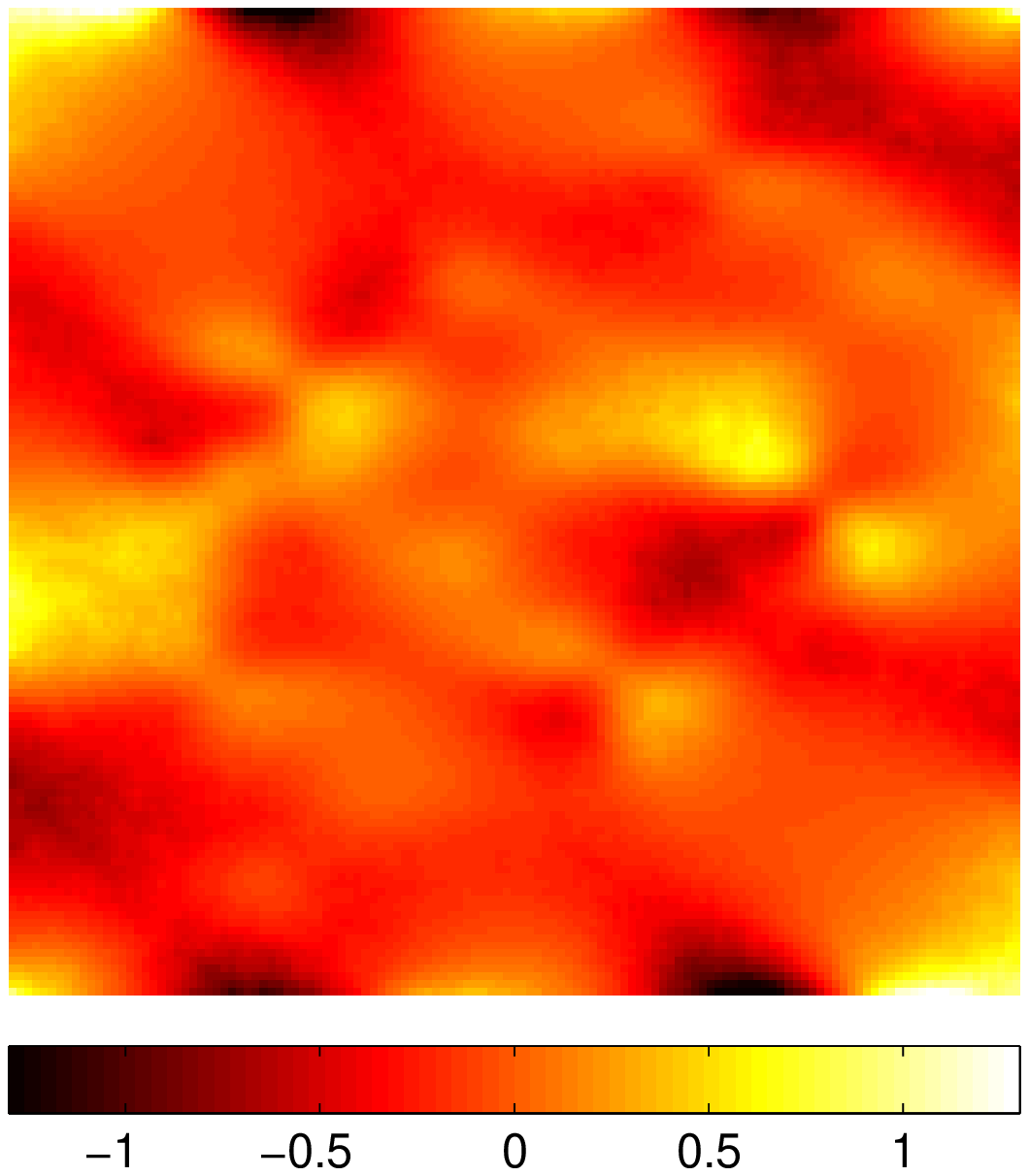}		
    \label{fig:34}
    }    
    \caption{Examples of measurement data.}
    \label{fig:3}
\end{figure}
The isotropic part modeled by $|\gamma|^{\frac{1}{2}}$ is given in Fig. \ref{fig:4}\subref{fig:45}. Fig. \ref{fig:3} displays the corresponding internal functionals with and without noise. Fig. \ref{fig:4} displays the reconstructed $\theta$ and $|\gamma|^\frac{1}{2}$. These reconstructions are quite robust to noise when the anisotropy is known: the relative $L^2$ ($L^\infty$) errors are $0.06\%$ ($0.14\%$) for $|\gamma|^\frac{1}{2}$ and $0.04\%$ ($0.4\%$) for $\theta$ with noiseless data, and $3.2\%$ ($12.5\%$) for $|\gamma|^\frac{1}{2}$ and $14.2\%$ ($24.5\%$) for $\theta$ with $30\%$ noise. If the anisotropy $\wtA$ is first reconstructed from noisy data, this certainly has repercussions on the reconstructed $(\theta, |\gamma|^\frac{1}{2})$, as can be seen in Fig. \ref{fig:4}\subref{fig:43}\&\subref{fig:47}. In this case, the $L^2$ ($L^\infty$) relative errors are $2.2\%$ ($9\%$) for $|\gamma|^\frac{1}{2}$ and $42.7\%$ ($60\%$) for $\theta$. 

\begin{figure}[htpb]
    \centering 
    \subfigure[true $\theta$]{
    \includegraphics[width=0.22\textwidth]{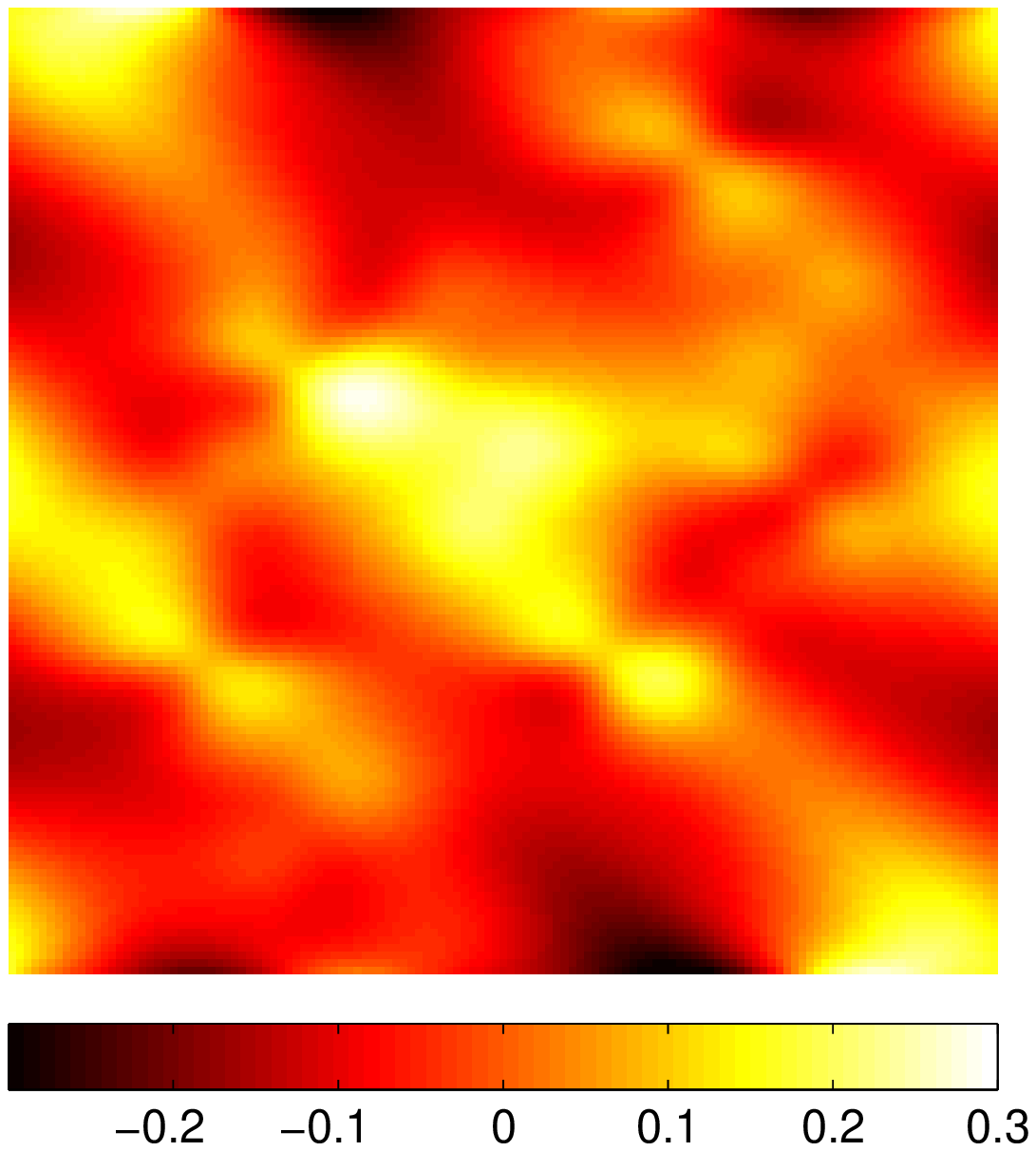}
    \label{fig:41}
    }
    \subfigure[$\theta$ (``1'' on \subref{fig:24})]{
    \includegraphics[width=0.22\textwidth]{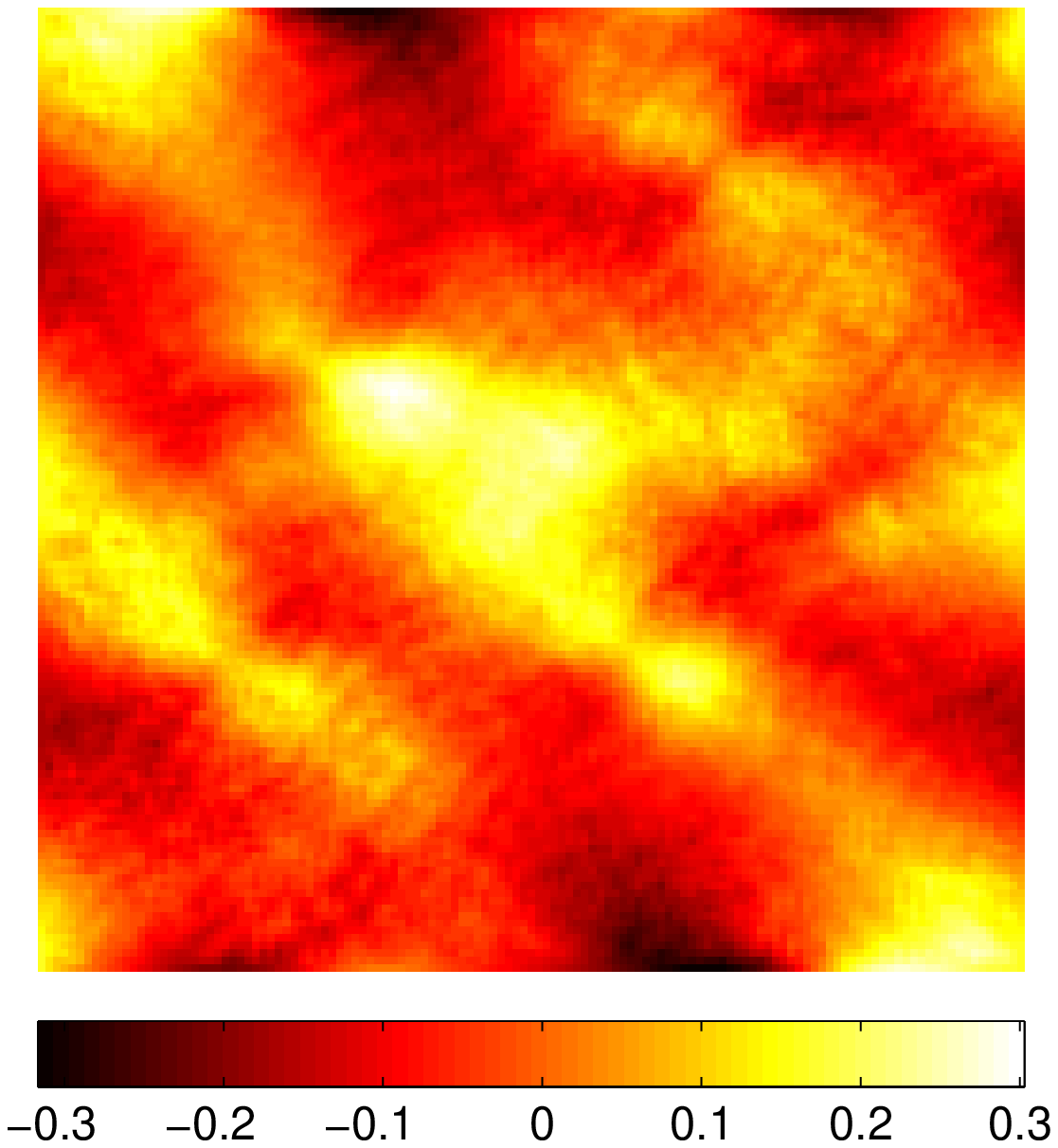}
    \label{fig:42}
    }
    \subfigure[$\theta$ (``1'' on \subref{fig:24})]{
    \includegraphics[width=0.22\textwidth]{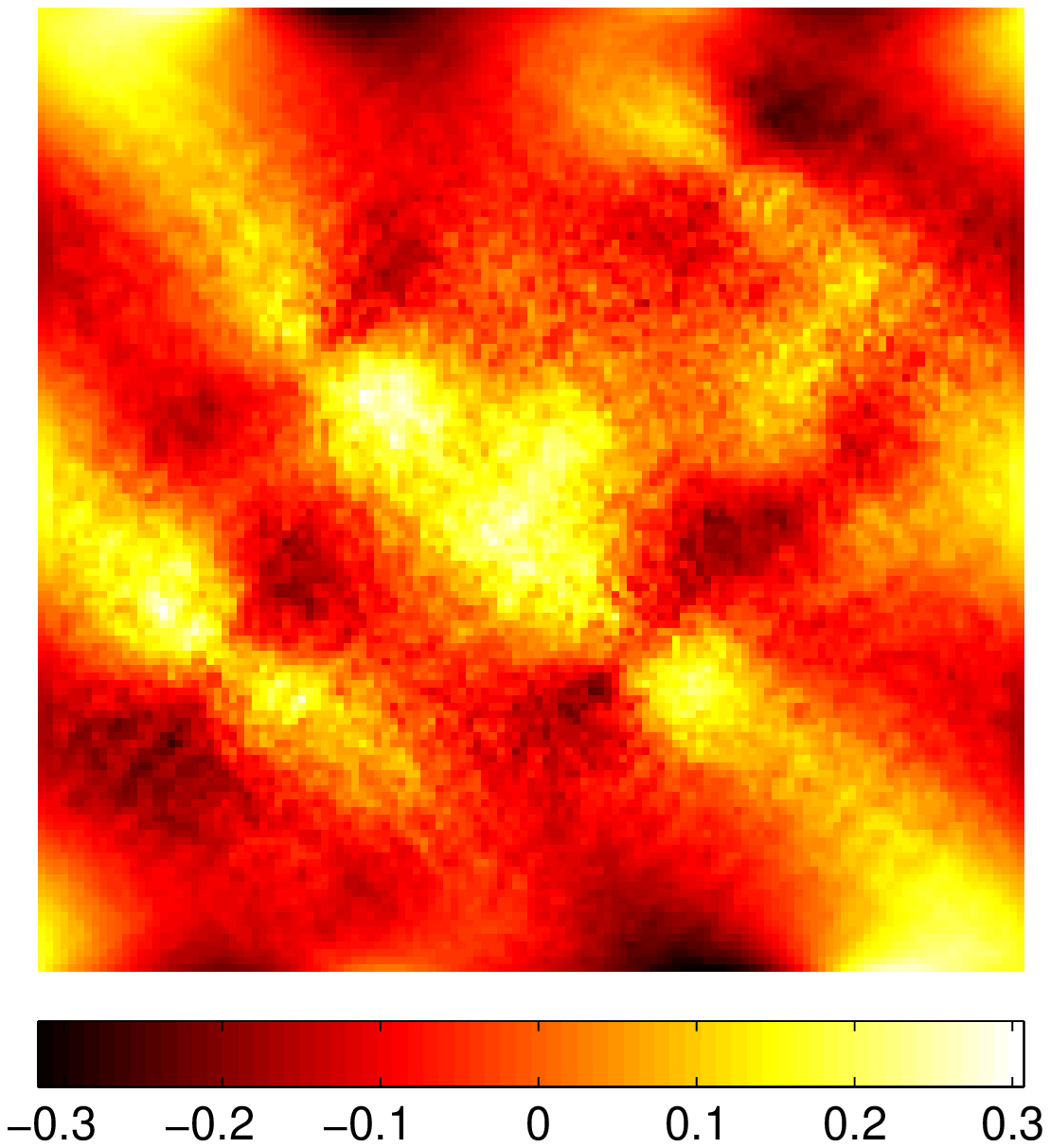}
    \label{fig:43}
    }
    \subfigure[$\theta$ at $\{x=0.5\}$]{
    \includegraphics[width=0.22\textwidth]{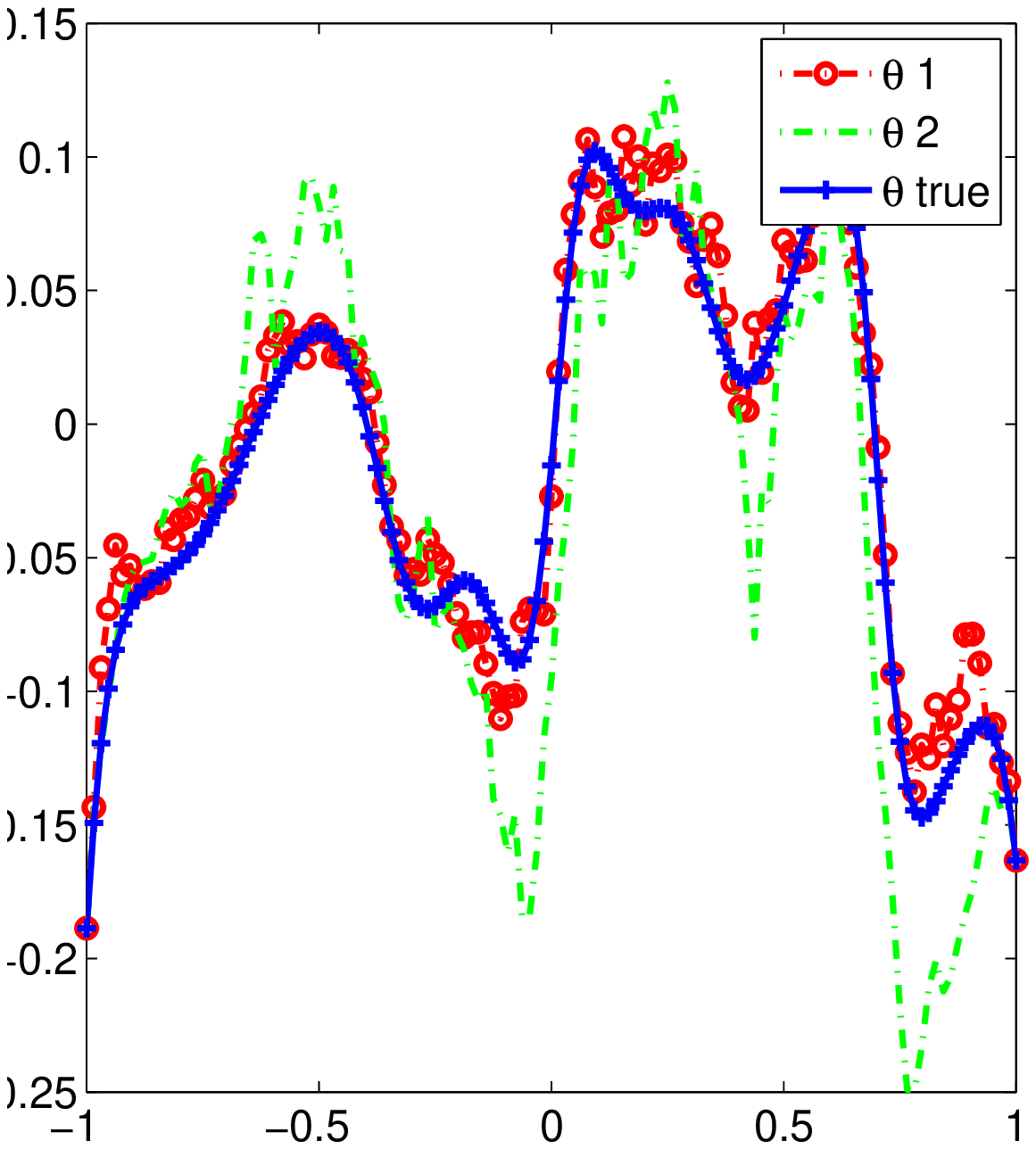}
    \label{fig:44}
    }    
    \subfigure[true $|\gamma|^{\frac{1}{2}}$]{
    \includegraphics[width=0.22\textwidth]{lda}
    \label{fig:45}
    }
    \subfigure[$|\gamma|^{\frac{1}{2}}$ (``1'' on \subref{fig:28})]{
    \includegraphics[width=0.22\textwidth]{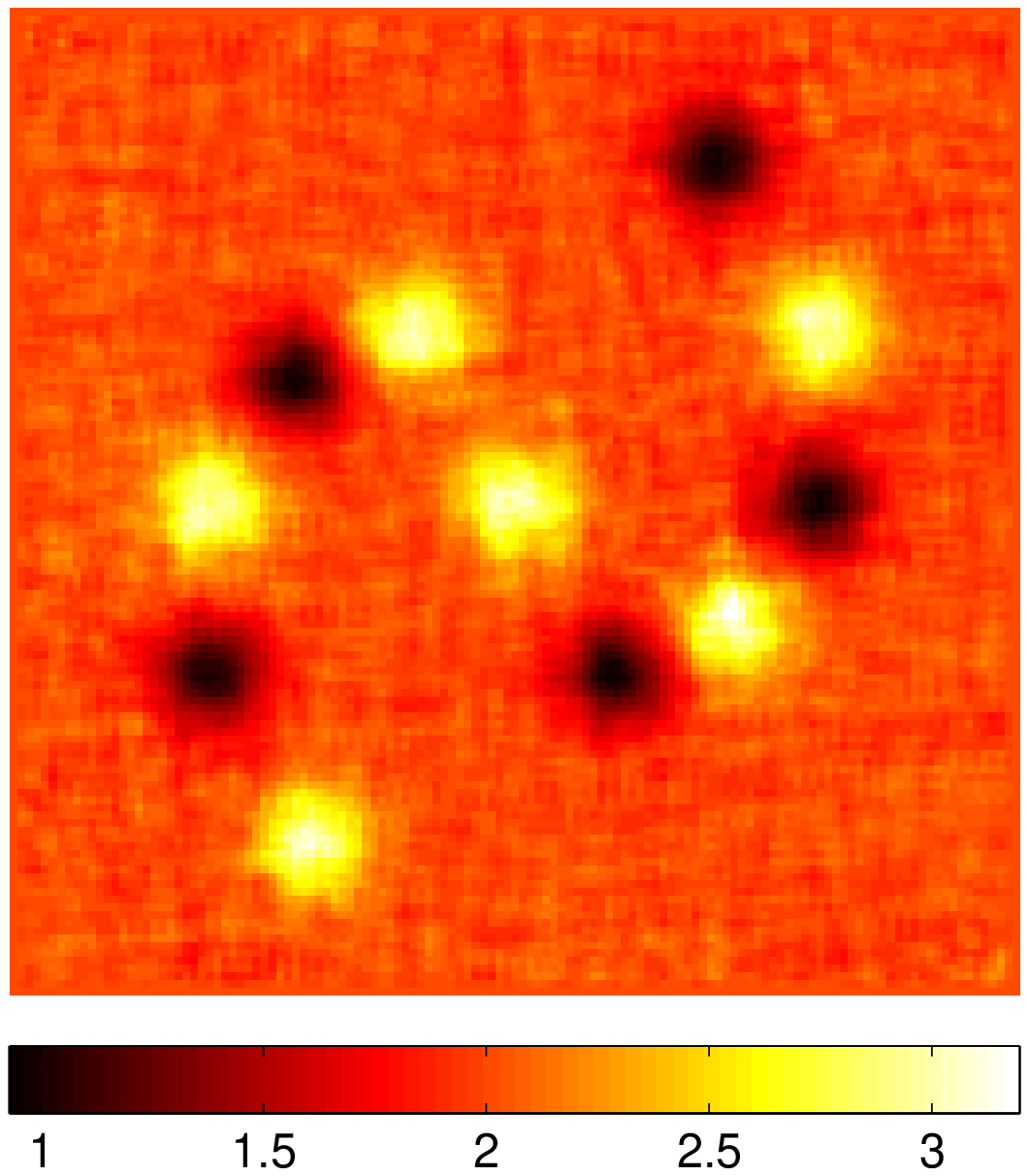} 
    \label{fig:46}
    }
    \subfigure[$|\gamma|^{\frac{1}{2}}$ (``2'' on \subref{fig:28})]{
    \includegraphics[width=0.22\textwidth]{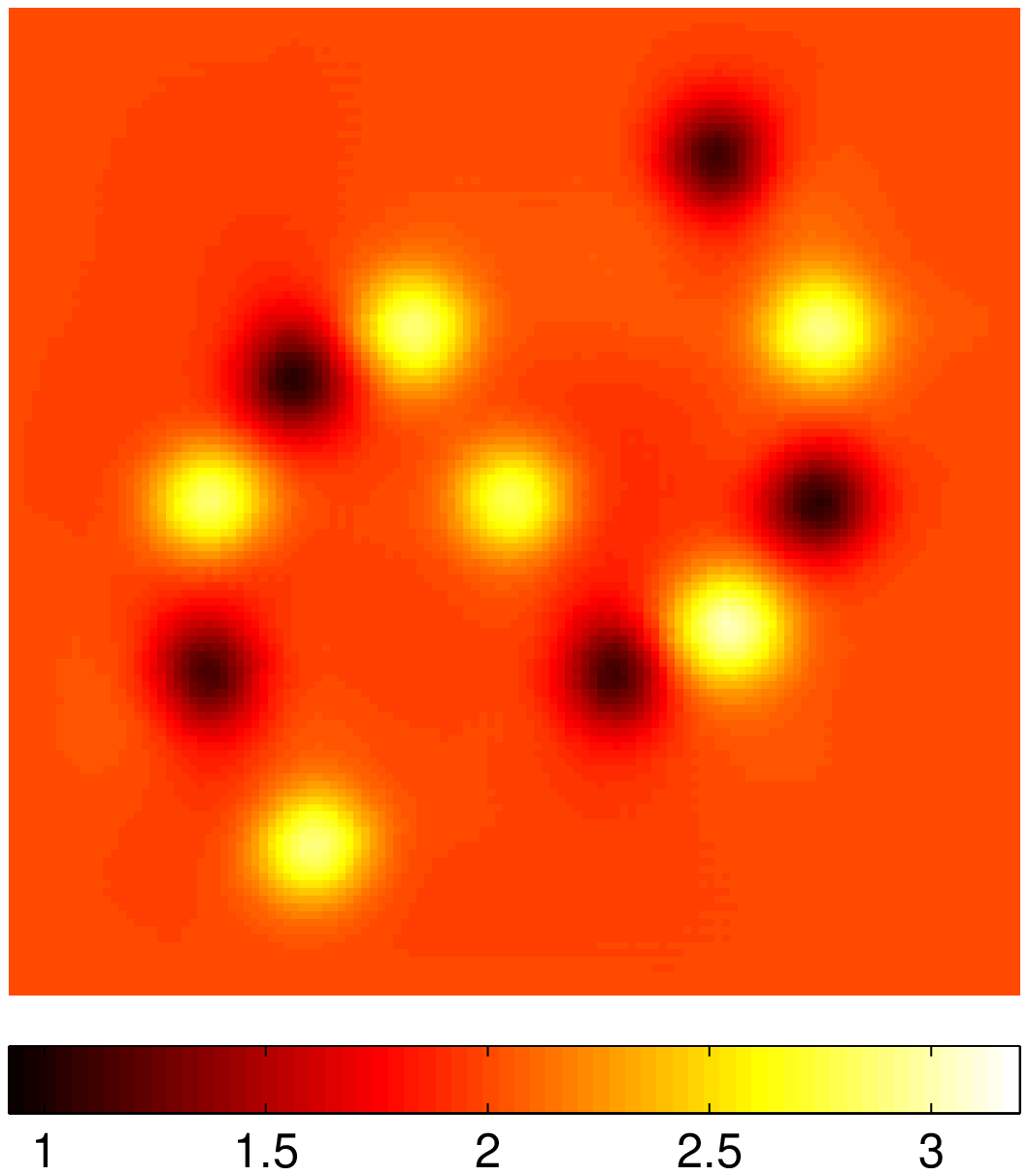}
    \label{fig:47}
    }
    \subfigure[$|\gamma|^{\frac{1}{2}}$ at $\{x=0.5\}$]{
    \includegraphics[width=0.22\textwidth]{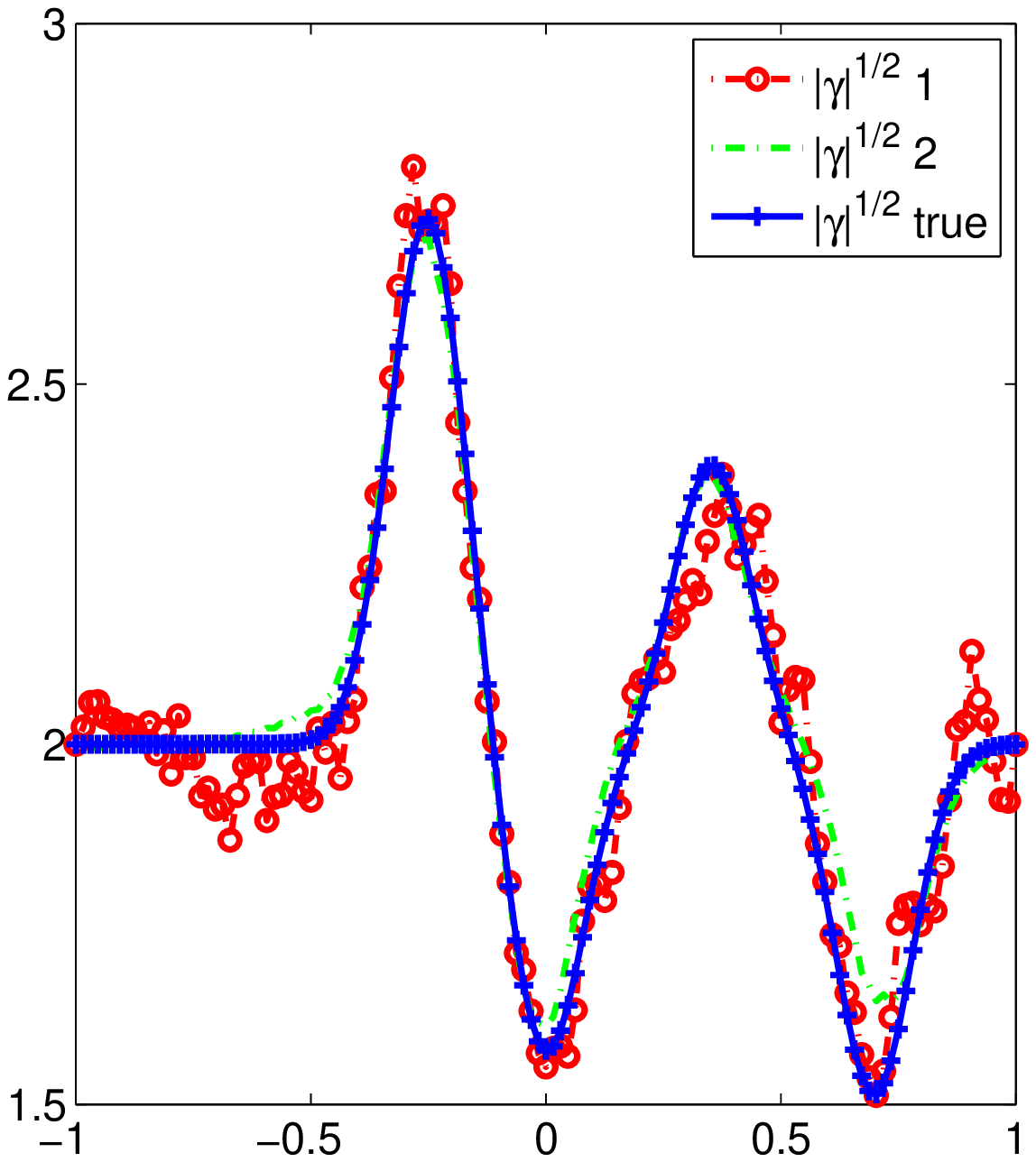}	
    \label{fig:48}
    }        
    \caption{Reconstruction of $(\theta,|\gamma|^{\frac{1}{2}})$. \subref{fig:42}\&\subref{fig:46}: with true anisotropy and noisy data ($\alpha=30\%$). \subref{fig:43}\&\subref{fig:47}: with noisy data ($\alpha = 0.1\%$) using reconstructed anisotropy from Fig.\ref{fig:2}\subref{fig:23}\&\subref{fig:27}.}
    \label{fig:4}
\end{figure}

\paragraph{Reconstruction of $|\gamma|^{\frac{1}{2}}$ via $(u_1, u_2, |\gamma|^{-\frac{1}{2}})$.}

We still assume $\wtA$ known with the coefficients $(\xi,\zeta)$ displayed in Fig. \ref{fig:2}. The boundary conditions are the same as in the preceding example. The isotropic component $|\gamma|^{\frac{1}{2}}$ is now given in Fig. \ref{fig:5}\subref{fig:55} and corresponds to a non-smooth coefficient. In this context, we first reconstruct $(u_1,u_2)$ by solving system \eqref{eq:sces_div}, after which we reconstruct $|\gamma|^{-\frac{1}{2}}$ by taking the divergence of \eqref{eq:nlda_m2}. %Again, we take exact data as boundary condition for the Poisson equation, and we conduct the experiments with both noisefree and noisy data. 
%This reconstruction method and the previous one have very comparable accuracy and robustness. this example uses a discontinuous $|\gamma|^{\frac{1}{2}}$ and is restricted to exact anisotropy $\wtA$. Fig. \ref{fig:5} displays the results. In this rather non-smooth case, the reconstructions present oscillations that can easily be dealt with using $L^1$ minimization methods. We do not explore it here. 
The relative $L^2$ ($L^\infty$) errors are $3.9e-13\%$ ($6.8e-13\%$) for $u_1$, $2.5e-13\%$ ($6.8e-13\%$) for $u_2$ and $13\%$ ($62\%$) for $|\gamma|^{\frac{1}{2}}$ in the case of noisefree data, and $0.2\%$ ($0.7\%$) for $u_1$, $0.1\%$ ($0.4\%$) for $u_2$ and $14\%$ ($71\%$) for $|\gamma|^{\frac{1}{2}}$ in the case of data polluted with $30\%$ noise. See Fig. \ref{fig:5} for the display of some reconstructions. Based on the numerical simulations that we have performed, this reconstruction method and the previous one are very comparable in terms of accuracy and robustness.
%In both cases, $u_1$ and $u_2$ are reconstructed spot on, so these reconstructions do not appear on Fig. \ref{fig:5}. 
\begin{figure}[htpb]
    \centering 
    \subfigure[$u_1$]{
    \includegraphics[width=0.22\textwidth]{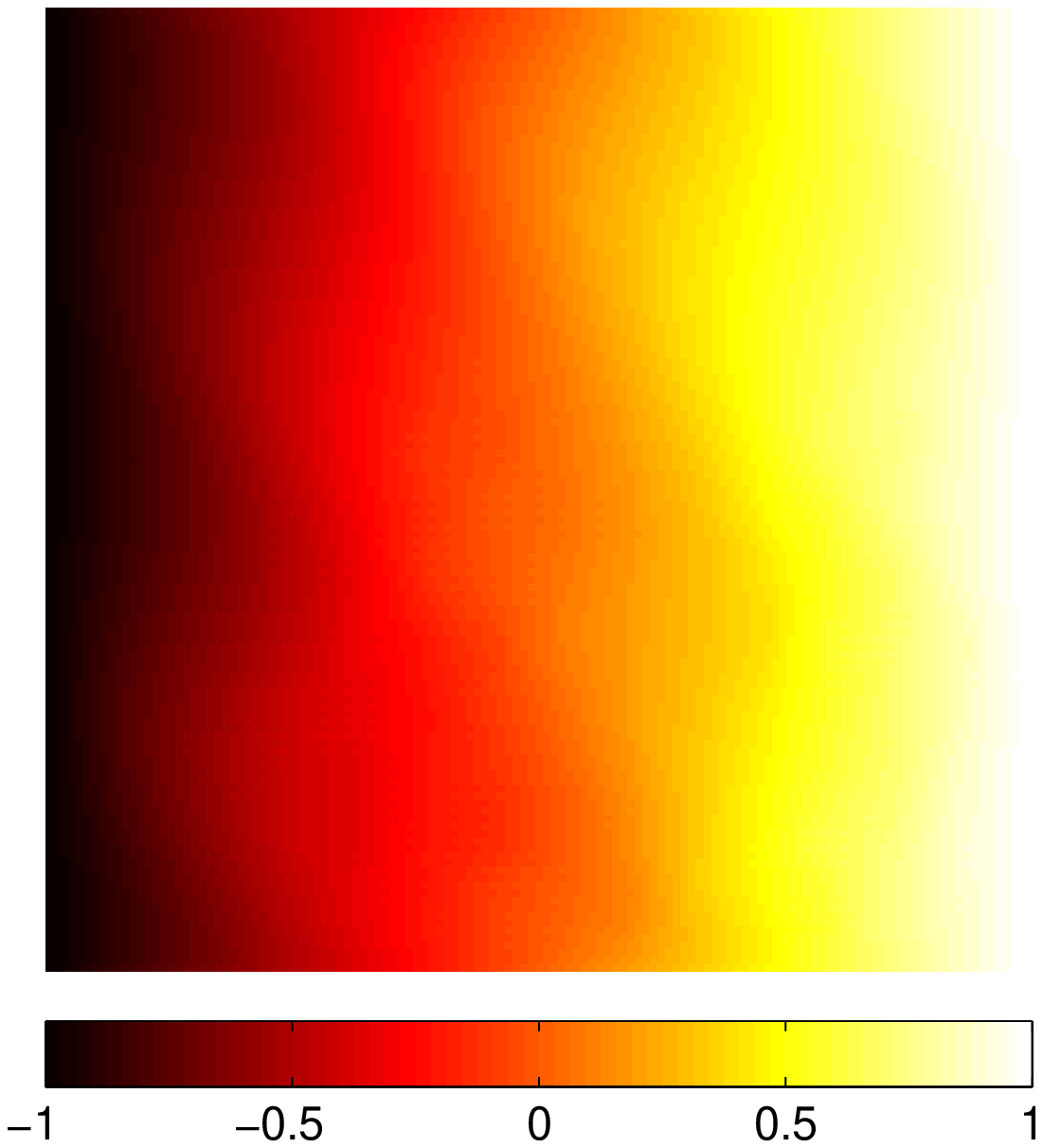}
    \label{fig:51}
    }
    \subfigure[$u_1$ at $\{x=0.5\}$]{
    \includegraphics[width=0.22\textwidth]{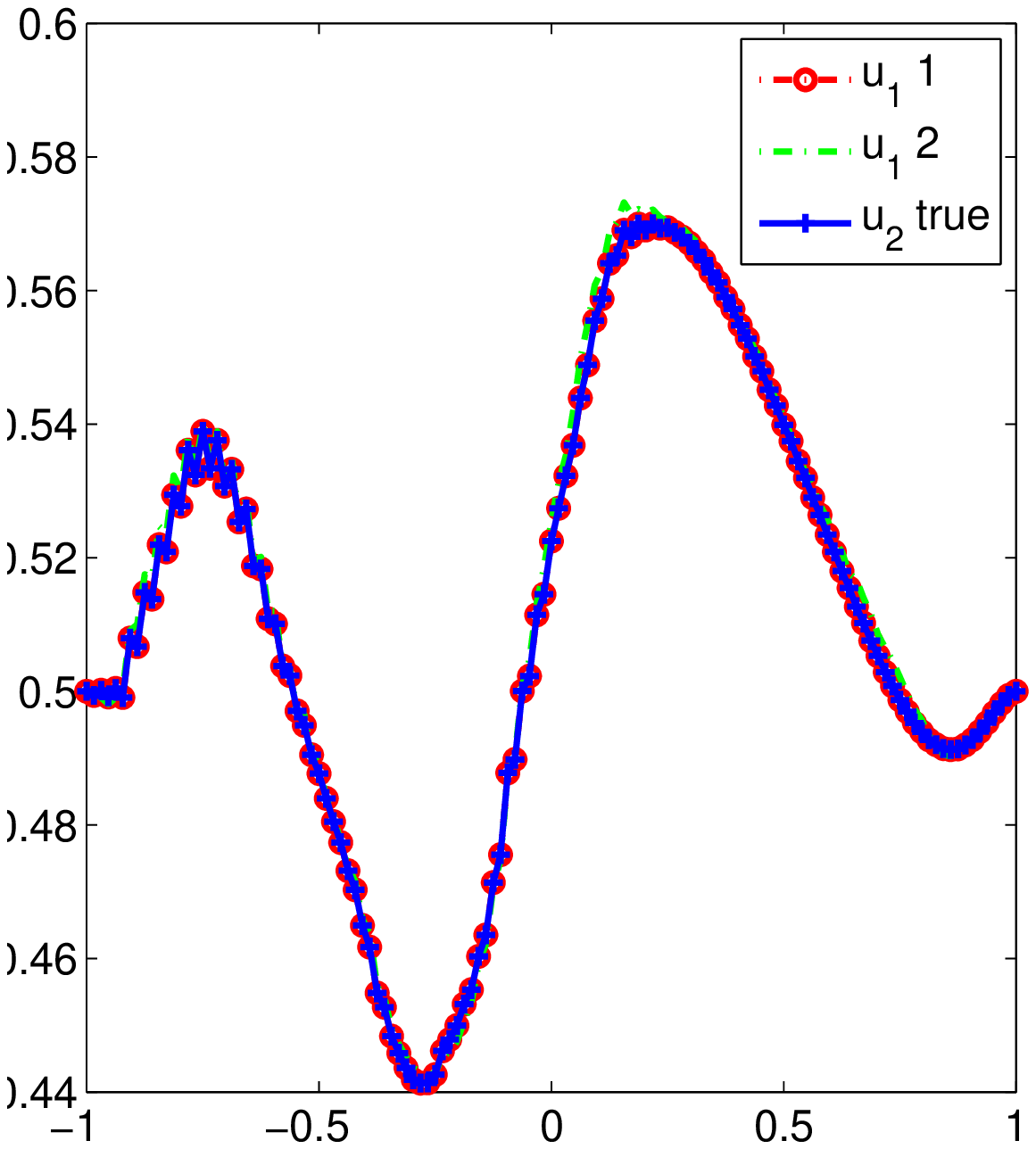}
    \label{fig:52}
    }
    \subfigure[$u_2$]{
    \includegraphics[width=0.22\textwidth]{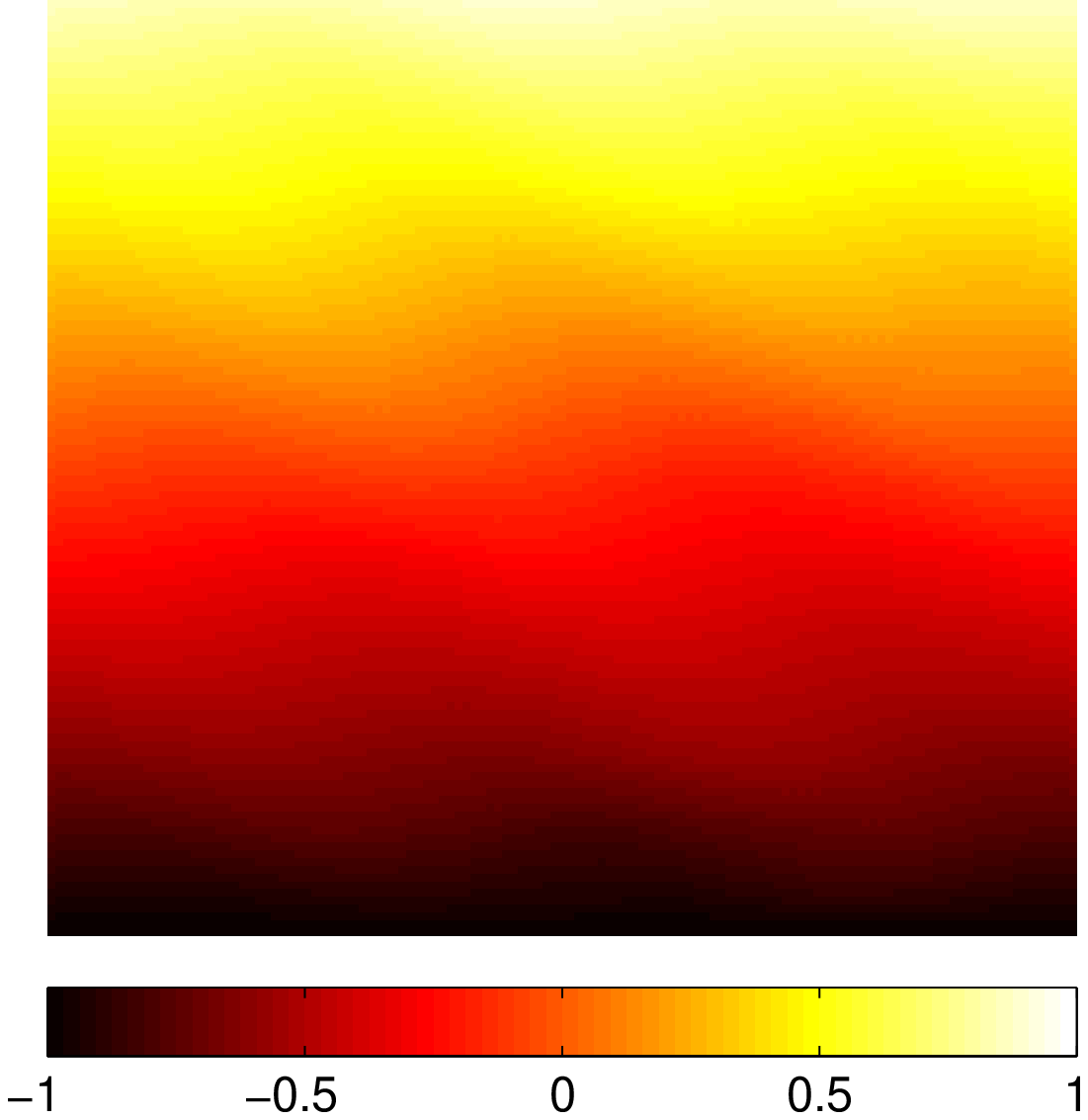}
    \label{fig:53}
    }
    \subfigure[$u_2$ at $\{x=0.5\}$]{
    \includegraphics[width=0.22\textwidth]{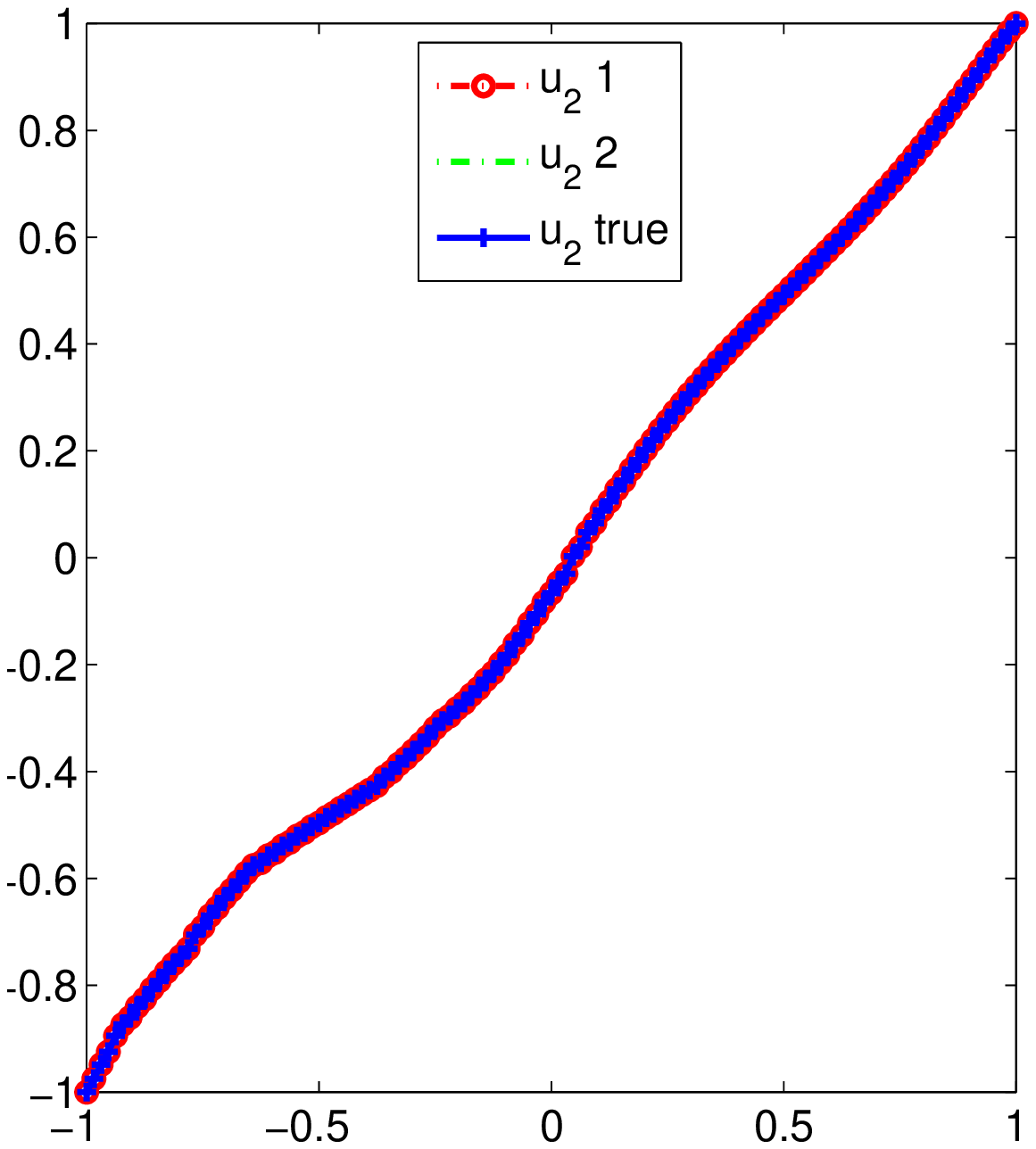}
    \label{fig:54}
    }    
    \subfigure[true $|\gamma|^{\frac{1}{2}}$]{
    \includegraphics[width=0.22\textwidth]{ldan1}
    \label{fig:55}
    }
    \subfigure[$|\gamma|^{\frac{1}{2}}$ (``1'' on \subref{fig:58})]{
    \includegraphics[width=0.22\textwidth]{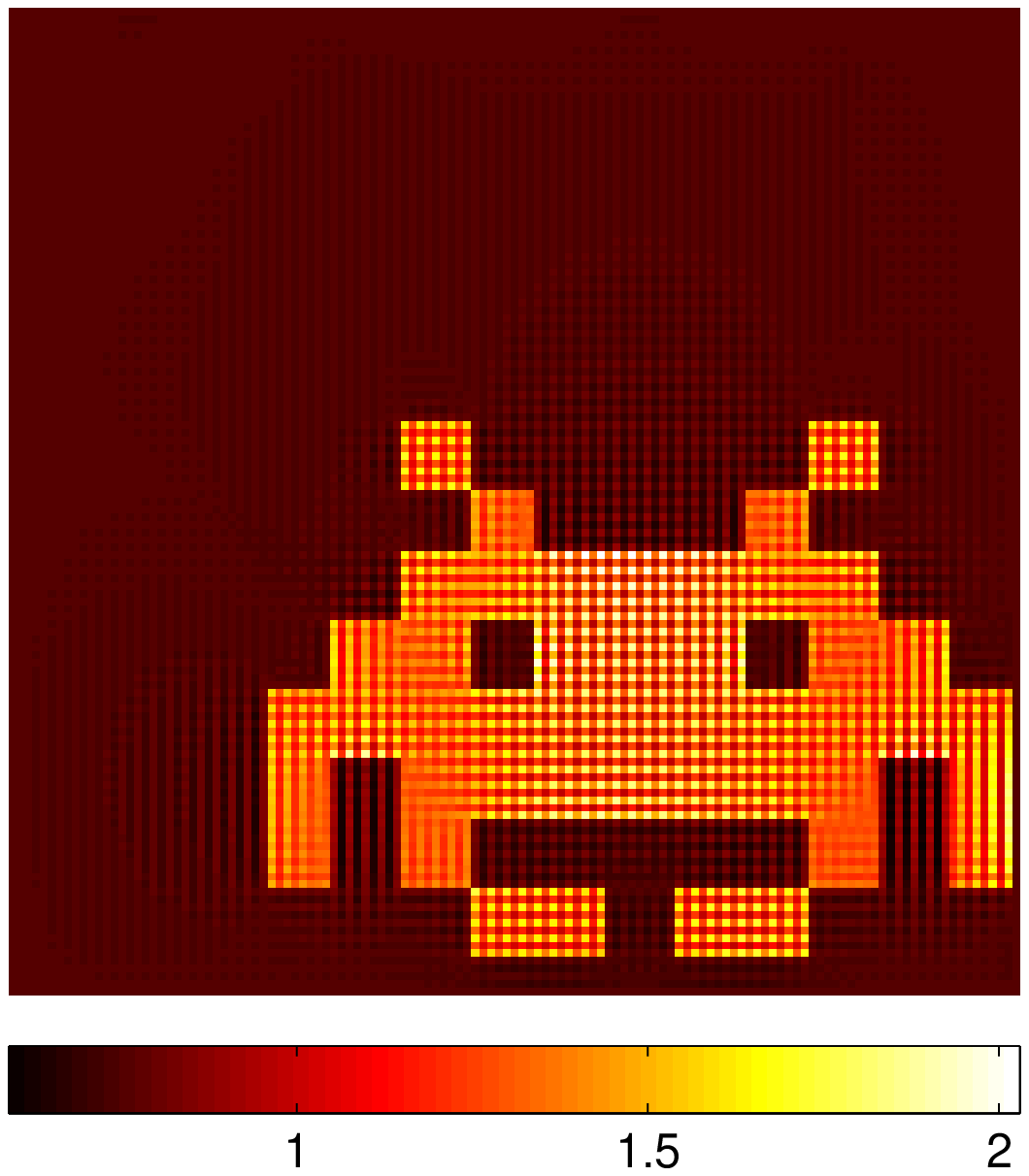} 
    \label{fig:56}
    }
    \subfigure[$|\gamma|^{\frac{1}{2}}$ (``2'' on \subref{fig:58})]{
    \includegraphics[width=0.22\textwidth]{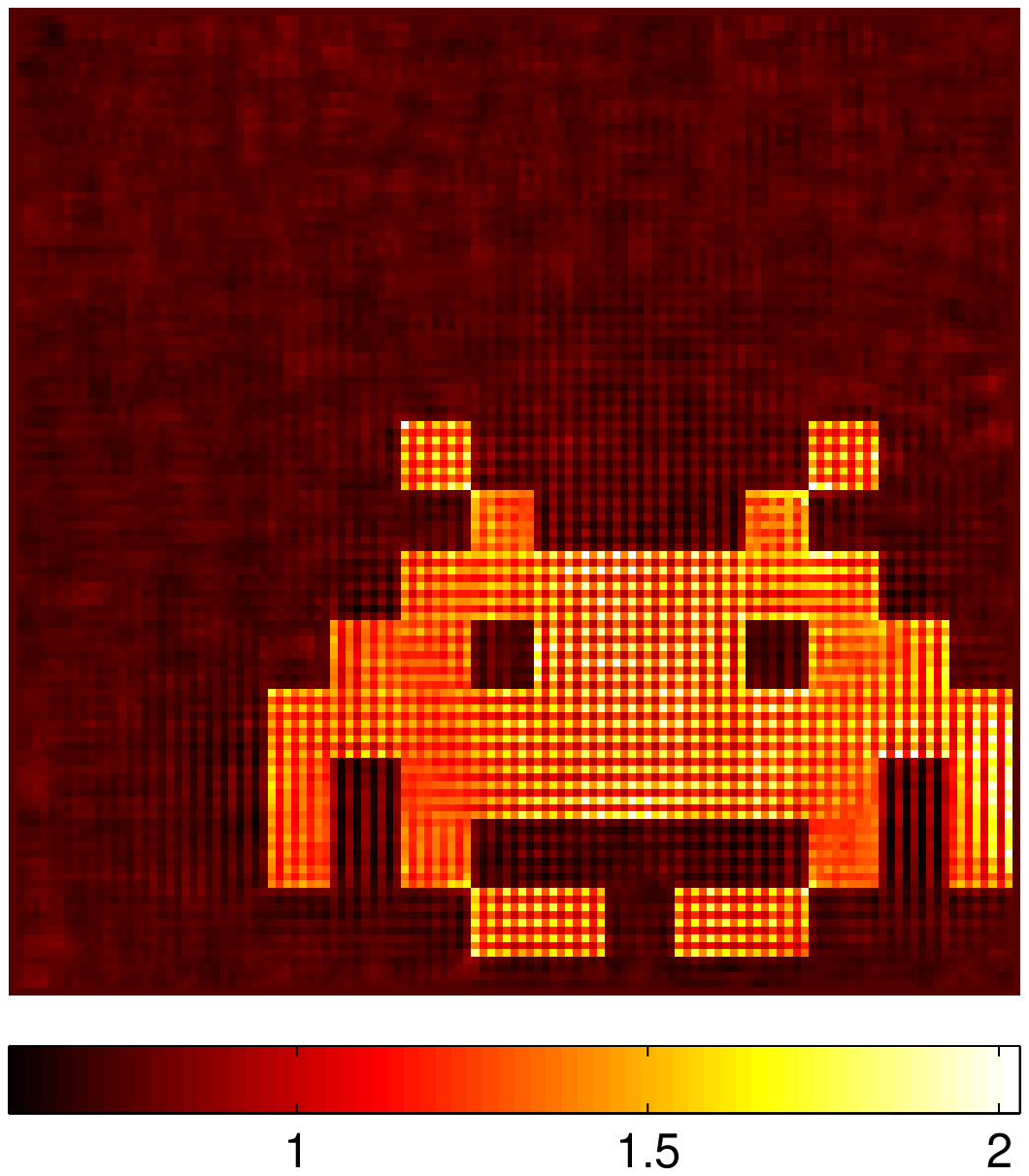}
    \label{fig:57}
    }
    \subfigure[$|\gamma|^{\frac{1}{2}}$ at $\{x=0.5\}$]{
    \includegraphics[width=0.22\textwidth]{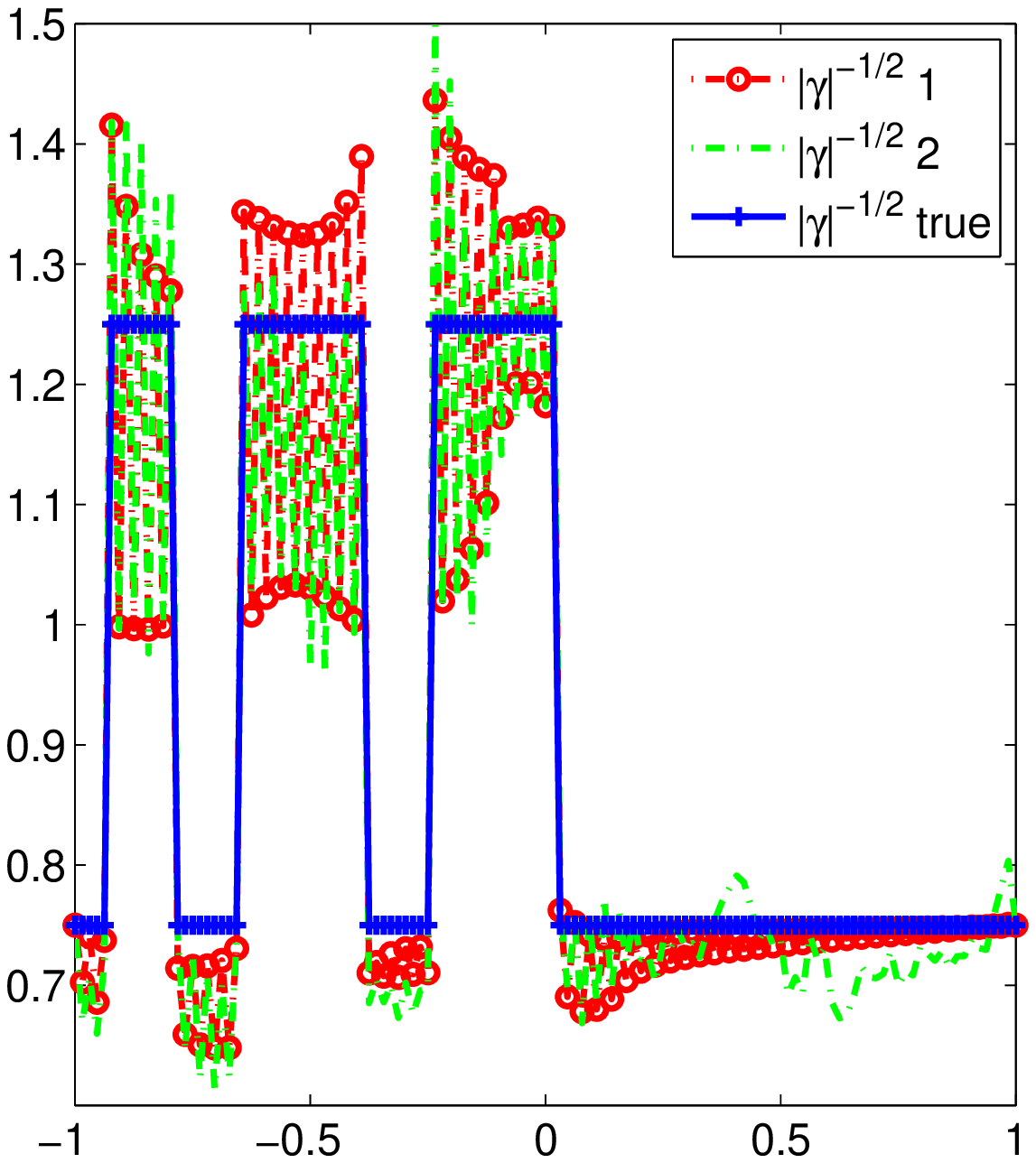}	
    \label{fig:58}
    }    
    \caption{Reconstruction of $(u_1,u_2,|\gamma|^{\frac{1}{2}})$ with true anisotropy and discontinuous $|\gamma|^{\frac{1}{2}}$. \subref{fig:56}: with noisefree data. \subref{fig:57}: with noisy data ($\alpha=30\%$).}
    \label{fig:5}
\end{figure}

\section{Conclusions}
\label{sec:conclu}

This paper presented an explicit reconstruction procedure for a diffusion tensor $\gamma=(\gamma_{ij})_{1\leq i,j\leq 2}$ from power density internal functionals
$H_{ij}=\gamma\nabla u_i\cdot \nabla u_j$ for $1\leq i,j\leq I$ with $I\leq 4$.

Provided that four illuminations $g_i$ are selected so that the  qualitative properties \eqref{eq:cond1} and \eqref{eq:cond2} of the corresponding solutions $u_i$ are verified, we obtained an explicit expression for the anisotropic part of $\gamma$ and showed in Theorem \ref{thm:anisotropy} that errors in the uniform norm of the anisotropy were controlled by errors in the uniform norm of derivatives of the functionals $H_{ij}$.

Once the anisotropy is reconstructed, or is known by other means, we have presented two methods to reconstruct the determinant of $\gamma$. The first functional is based on first reconstructing the angle $\theta$ between $\bfe_x$ and $\gamma\nabla u_1$. The second method is based on solving a coupled system of elliptic equations for $(u_1,u_2)$.  In both cases, we need that the three internal functionals $H_{11}, H_{22}$ and $H_{12}$ satisfy \eqref{eq:positivity} in $X$. And in both cases, we obtain that the error in the uniform norm of the derivative of the determinant was controlled by errors in the uniform norm of derivatives of the functionals $H_{ij}$. 

This shows that the reconstruction of the determinant of $\gamma$ is more stable than that of the anisotropy of $\gamma$. Such a statement was verified by numerical simulations. In the presence of very limited noise generated by the numerical discretization, we obtained accurate reconstructions of the full tensor $\gamma$. However, even in the presence of quite small additional noise on the functionals $H_{ij}$, we observed that the reconstruction of the anisotropy was degrading very rapidly. On the other hand, the reconstruction of the determinant of $\gamma$, assuming the anisotropy known, proved very stable even in the presence of significant noise in the available functionals. In practice, this shows that appropriate regularization procedures on the anisotropy need to be introduced, for instance based on regularity or sparsity assumptions. This now standard step was not considered here. 

The functionals emerging from ultrasound modulation experiments are thus sufficiently rich to provide unique reconstructions of arbitrary diffusion tensors. This should be contrasted with reconstruction procedures based on boundary measurements of elliptic solutions, in which diffusion tensors are defined up to an arbitrary change of variables inside the domain of interest \cite{ALP}. Moreover, reconstructions are stable with a resolution that is essentially independent of the location of the point inside the domain of interest. 

The reconstruction procedure presented here is two dimensional. Although this is not presented here, the reconstruction of the determinant of $\gamma$ knowing its anisotropy generalizes to the $n$-dimensional setting using techniques developed in \cite{BBMT,MB}. The reconstruction of the full diffusion tensor in dimension $n\geq3$ remains open at present.

\section*{Acknowledgment} This research was supported in part by National Science Foundation Grants DMS-0804696 and DMS-1108608. The authors would like to thank the referees for their valuable comments.

%\newpage

%\begin{figure}[htpb]
%    \begin{center}
%	\includegraphics[width=0.23\textwidth]{testH11} 
%	\includegraphics[width=0.23\textwidth]{testH11_jet} 
%	\includegraphics[width=0.23\textwidth]{testH118pc} 
%	\includegraphics[width=0.23\textwidth]{testH118pc_jet} 
%    \end{center}
%    \caption{test}
%    \label{fig:test}
%\end{figure}

%\input comments

\end{document}